\documentclass{amsart}  

\usepackage{amssymb} 
\usepackage{graphicx}   
\usepackage{enumerate}
\usepackage[dvipsnames]{xcolor} 
\usepackage[pagebackref]{hyperref}
\hypersetup{
    colorlinks=true,
    linkcolor=blue, 
    filecolor=blue,      
    urlcolor=blue, 
    citecolor=blue,
}
\usepackage{amsmath,amssymb,amsrefs}
\usepackage{tikz}
\usetikzlibrary{positioning, cd, arrows, backgrounds, arrows.meta,decorations,decorations.pathreplacing}
\usepackage{enumerate}
\usepackage{bm}
\usepackage{float}
\usepackage{wrapfig}
\usepackage{caption}
\usepackage{subcaption}
\usepackage{mathrsfs}
\usepackage{bbold}
\usepackage{appendix}
\usepackage{afterpage} 
\usepackage{wasysym}
\usepackage{longtable}
\usepackage{stmaryrd}

\newcommand{\tin}{:}

\usetikzlibrary{nfold}

\tikzset{Rightarrow/.style={double distance=2pt,scaling nfold=2,>={Implies},->}}

\makeatletter
\newcommand{\oset}[3][0ex]{\mathrel{\mathop{#3}\limits^{
    \vbox to#1{\kern-2\ex@
    \hbox{$\scriptstyle#2$}\vss}}}}
\makeatother
\newcommand{\Cat}[1]{\mathsf{#1}}
\newcommand{\cat}[1]{\Cat{#1}}
\newcommand{\acat}[1]{\mathsf{#1}}

\newcommand{\Caps}[1]{\mathsf{#1}}

\newcommand{\Core}[1]{\oset{\leftrightarrow}{\acat{#1}}}
\newcommand{\core}[1]{\Core{#1}}
\newcommand{\LongCore}[1]{\oset{\longleftrightarrow}{\acat{#1}}}
\newcommand{\lcore}[1]{\LongCore{#1}}
\DeclareRobustCommand\longtwoheadrightarrow
     {\relbar\joinrel\twoheadrightarrow}
\newcommand{\Epi}[1]{\oset{\twoheadrightarrow}{\acat{#1}}}
\newcommand{\epi}[1]{\Epi{#1}}
\newcommand{\LongEpi}[1]{\oset{\longtwoheadrightarrow}{\acat{#1}}}
\newcommand{\lepi}[1]{\LongEpi{#1}}  
  
\newcommand{\Mono}[1]{\oset{\hookrightarrow}{\acat{#1}}}
\newcommand{\mono}[1]{\Mono{#1}}

\usepackage[many]{tcolorbox}
\tcbuselibrary{listings}

\newlength\inwd
\newcounter{notebookcounter}

\newfloat{lstfloat}{htbp}{lop}
\floatname{lstfloat}{Listing}
\numberwithin{lstfloat}{section}

\colorlet{inputcolor}{blue!50!black}
\colorlet{outputcolor}{red!50!black}

\lstdefinelanguage{genericnotebookno}{
    morekeywords={print, def, return},
    morecomment=[l]{//},
    basicstyle=\footnotesize\ttfamily,
    keywordstyle=\color{green!40!black}\bfseries,
    commentstyle=\color{purple!40!black}\itshape,
}

\lstdefinelanguage{genericnotebookin}{
    morekeywords={print, def, return},
    morecomment=[l]{//},
    basicstyle=\footnotesize\ttfamily,
    keywordstyle=\color{green!40!black}\bfseries,
    commentstyle=\color{purple!40!black}\itshape,
}

\lstdefinelanguage{genericnotebookout}{
    morekeywords={:,Type},
    morecomment=[l]{//},
    basicstyle=\footnotesize\ttfamily,
    keywordstyle=\bfseries\color{green!40!black},
    commentstyle=\itshape\color{purple!40!black},
}

\newtcblisting{notebookno}[1][\thenotebookcounter]{
  enlarge left by=\inwd-1.2cm,
  width=\linewidth-\inwd,
  enhanced,
  boxrule=0.4pt,
  colback=ForestGreen!15,
  listing only,
  top=0pt,
  bottom=0pt,
  overlay={
    \node[
      anchor=north east,
      text width=\inwd,
      font=\footnotesize\ttfamily\color{inputcolor},
      inner ysep=2mm,
      inner xsep=0pt,
      outer sep=0pt
      ] 
      at (frame.north west)
      {\stepcounter{notebookcounter}};
  }
  listing engine=listing,
  listing options={
    aboveskip=1pt,
    belowskip=1pt,
    language=genericnotebookno,
    showstringspaces=false,
    breaklines=true,
  },
}

\newtcblisting{notebookin}[1][\thenotebookcounter]{
  enlarge left by=\inwd,
  width=\linewidth-\inwd,
  enhanced,
  boxrule=0.4pt,
  colback=ForestGreen!15,
  listing only,
  top=0pt,
  bottom=0pt,
  overlay={
    \node[
      anchor=north east,
      text width=\inwd,
      font=\footnotesize\ttfamily\color{inputcolor},
      inner ysep=2mm,
      inner xsep=0pt,
      outer sep=0pt
      ] 
      at (frame.north west)
      {\stepcounter{notebookcounter}In [#1]:};
  }
  listing engine=listing,
  listing options={
    aboveskip=1pt,
    belowskip=1pt,
    language=genericnotebookin,
    showstringspaces=false,
    breaklines=true,
  },
}
\newtcblisting{notebookout}[1][\thenotebookcounter]{
  enlarge left by=\inwd,
  width=\linewidth-\inwd,
  enhanced,
  boxrule=0.1pt,colback=White,
  listing only,
  top=0pt,
  bottom=0pt,
  overlay={
    \node[
      anchor=north east,
      text width=\inwd,
      font=\footnotesize\ttfamily\color{outputcolor},
      inner ysep=2mm,
      inner xsep=0pt,
      outer sep=0pt
      ] 
      at (frame.north west)
      {Out [#1]:};
  }
  listing engine=listing,
  listing options={
    aboveskip=1pt,
    belowskip=1pt,
    language=genericnotebookout,
    showstringspaces=false,
  },
}

\floatstyle{ruled}
\newfloat{typefloat}{htbp}{lop}
\floatname{typefloat}{Type}

\DeclareMathOperator{\Dom}{Dom}
\DeclareMathOperator{\Codom}{Codom}
\DeclareMathOperator{\im}{Im}

\newcommand{\Free}[2]{#1\langle #2\rangle}

\DeclareMathOperator{\End}{End}
\DeclareMathOperator{\Aut}{Aut}

\DeclareMathOperator{\GL}{GL}

\DeclareMathOperator{\id}{id}

\newcommand{\srcfunc}{\mathbin{\blacktriangleleft}}
\newcommand{\tgtfunc}{\mathbin{\blacktriangleleft}}
\newcommand{\src}[1]{#1\srcfunc}
\newcommand{\tgt}[1]{\tgtfunc #1}

\usepackage{amsfonts}
\makeatletter
\def\amsbb{\use@mathgroup \M@U \symAMSb}
\makeatother
\newcommand{\one}{\mathbb{1}}

\renewcommand{\ker}{{\rm ker\,}}
\newcommand{\coker}{{\rm coker\,}}

\newcommand{\type}[1]{#1}
\tikzcdset{
  cells={font=\everymath\expandafter{\the\everymath\displaystyle}},
}

\usepackage{scalerel} 
\newcommand{\defeq}{\mathrel{\hstretch{.13}{=}\hspace{.2ex}{=}}}
\renewcommand{\leq}{\leqslant}

\newcommand{\inputref}[2]{\hyperref[#1]{\texttt{{\color{inputcolor}In\,[#2]}}}}
\newcommand{\outputref}[2]{\hyperref[#1]{\texttt{{\color{outputcolor}Out\,[#2]}}}}

\newcommand{\eval}{{\rm eval}}

\makeatletter
\newcommand{\xRrightarrow}[2][]{\ext@arrow 0359\Rrightarrowfill@{#1}{#2}}
\newcommand{\Rrightarrowfill@}{\arrowfill@\equiv\equiv\Rrightarrow}
\newcommand{\xLleftarrow}[2][]{\ext@arrow 3095\Lleftarrowfill@{#1}{#2}}
\newcommand{\Lleftarrowfill@}{\arrowfill@\Lleftarrow\equiv\equiv}
\makeatother

\newcommand{\counital}{\iota}
\newcommand{\unital}{\pi}

\newcommand{\adjoint}{\dashv}

\newcommand{\Autcat}{\cat{Aut}}

\newcommand{\func}[1]{\mathcal{#1}}

\newcommand{\fC}{\func{C}}
\newcommand{\fD}{\func{D}}
\newcommand{\fE}{\func{E}}
\newcommand{\fF}{\func{F}}
\newcommand{\fG}{\func{G}}
\newcommand{\fH}{\func{H}}
\newcommand{\fI}{\func{I}}
\newcommand{\fJ}{\func{J}}
\newcommand{\fK}{\func{K}}
\newcommand{\fL}{\func{L}}
\newcommand{\fM}{\func{M}}
\newcommand{\fN}{\func{N}}

\newcommand{\fR}{\func{R}}
\newcommand{\fS}{\func{S}}

\newcommand{\fU}{\func{U}}

\newcommand{\cA}{\cat{A}}
\newcommand{\cB}{\cat{B}}
\newcommand{\cC}{\cat{C}}
\newcommand{\cD}{\cat{D}}
\newcommand{\cE}{\cat{E}}

\newcommand{\cG}{\cat{G}}

\newcommand{\cS}{\cat{S}}

\newcommand{\cX}{\cat{X}}

\newcommand{\aA}{\Caps{A}}

\newcommand{\aX}{\Caps{X}}
\newcommand{\aY}{\Caps{Y}}
\newcommand{\aZ}{\Caps{Z}}

\newcommand{\venturi}{%
  \mathrel{
    \begin{tikzpicture}[yscale=0.6,xscale=0.5]
        \draw (0ex,2.5ex) to (1ex,2ex);
        \draw (1ex,2ex) to (3ex,2ex);
        \draw (0ex,0.75ex) to (1ex,1.25ex);
        \draw (1ex,1.25ex) to (3ex,1.25ex);
    \end{tikzpicture}
  }
}
\newcommand{\exqed}{\hfill $\square$}

\usepackage[draft]{pdfcomment}
\usepackage{xparse}
    \DeclareDocumentCommand \EDITmargin { o m } {
        \IfNoValueTF {#1} {
            \pdfmargincomment[icon=Note]{#2}
        }{
            \pdfmargincomment[icon=NOte,author=#1]{#2}
        }
    }
    \DeclareDocumentCommand \EDITcomment { o m } {
        \IfNoValueTF {#1} {
            \pdfcomment[color=Blue!20]{#2}
        }{
            \pdfcomment[color=Blue!20,author=#1]{#2}
        }
    }
    \DeclareDocumentCommand \EDITalt { o m m } {
        \IfNoValueTF {#1} {
            \pdfmarkupcomment[markup=StrikeOut]{#2}{#3}
        }{
            \pdfmarkupcomment[markup=StrikeOut,icon=key, author=#1]{#2}{#3}
        }
    }
    \DeclareDocumentCommand \EDITtypo { o m o } {
        \IfNoValueTF {#3} {
            \pdfmarkupcomment[markup=Squiggly,author=#1]{#2}{}
        }{
            \pdfmarkupcomment[markup=Squiggly,author=#1]{#2}{#3}
        }
    }
    \DeclareDocumentCommand \EDIThighlight { o m m } {
\pdfmarkupcomment[markup=Highlight,author=#1,color=blue!20]{#2}{#3}
}

\usepackage{enumitem}  

\newenvironment{ithm}{\begin{enumerate}[label={\rm(\alph*)}, ref=(\alph*),
      labelwidth=18pt, leftmargin=18pt, topsep=3pt, itemsep=1pt, parsep=2pt]}
      {\end{enumerate}}
      
\newenvironment{iprf}{\begin{enumerate}[label=(\alph*), ref=(\alph*),
      labelwidth=-18pt, leftmargin=0pt, topsep=3pt, itemsep=2pt, parsep=2pt]}
      {\end{enumerate}} 
            
\renewcommand{\leq}{\leqslant} 

\newcommand{\pf}{p}

\newtheorem{thm}{Theorem}[section]
\newtheorem{mainthm}{Theorem}
\newtheorem{prop}[thm]{Proposition}
\newtheorem{fact}[thm]{Fact}

\newtheorem{lem}[thm]{Lemma}

\theoremstyle{definition}
\newtheorem{defn}[thm]{Definition}
\newtheorem{ex}[thm]{Example}
\newtheorem{rem}[thm]{Remark}

\theoremstyle{remark}
\numberwithin{equation}{section}

\title{Categorification of characteristic structures}
\date{\today}
\author[P.\ A.\ Brooksbank]{Peter A.\ Brooksbank}
\author[H.\ Dietrich]{Heiko Dietrich}
\author[J.\ Maglione]{Joshua Maglione}
\author[E.\ A.\ O'Brien]{E.A.\ O'Brien}
\author[J.\ B.\ Wilson]{James B.\ Wilson}
\address[Brooksbank]{{\tt pbrooksb@bucknell.edu}, Bucknell University, USA}
\address[Dietrich]{{\tt heiko.dietrich@monash.edu}, Monash University, Australia}
\address[Maglione]{{\tt joshua.maglione@universityofgalway.ie}, University of Galway, Ireland}
\address[O'Brien]{{\tt e.obrien@auckland.ac.nz}, University of Auckland, New Zealand}
\address[Wilson]{{\tt James.Wilson@ColoState.Edu}, Colorado State University, USA}

\begin{document}

\begin{abstract}
   We develop a representation theory of categories as a means to 
   explore characteristic structures in algebra.
  Characteristic structures play a critical role in isomorphism testing of
  groups and algebras, and their construction and description often rely on
  specific knowledge of the parent object and its automorphisms. In many cases,
  questions of reproducibility and comparison arise. Here we present a 
  categorical framework that addresses these questions. We prove that every
  characteristic structure is the image of a functor equipped with a 
  natural transformation. 
  This  shifts the local description in the parent object to a
  global one in the ambient category.
  Through  constructions in representation theory, such as tensor products, 
  we can combine characteristic structure across multiple categories.
  Our results are constructive and stated in the language of a constructive type theory which facilitates their implementation in proof checkers.
\end{abstract} 

\maketitle

\newcommand{\customlabel}[2]{%
    \begingroup
    \renewcommand{\thetheorem}{#1}
    \refstepcounter{theorem}
    \label{#2}
    \endgroup
}

\section{Introduction}
\label{sec:intro}

The problem of deciding when two algebraic structures are isomorphic  
is fundamental to algebra and computer science. It encompasses issues of
decidability and complexity, and it tests the limits of our theories and
algorithms. An initial tactic in deciding isomorphism is to
identify substructures that are invariant under isomorphisms because doing so
reduces the search space. We first discuss groups, where the literature is 
most developed (see, for example, 
\citelist{\cite{ELGO2002}\cite{BOW}\cite{Maglione2021}\cite{Wilson:filters}}), 
but our results apply to monoids, loops, rings, and non-associative algebras.
 
A subgroup $H$ of a group $G$ is \emph{characteristic}
if  $\varphi(H)=H$ for every automorphism $\varphi:G\to G$; it is \emph{fully
invariant} if $\psi(H)\leq H$ for every homomorphism $\psi:G\to G$. We use
the language of categories, following  \cite{Riehl}, and a type of
natural transformation to describe our main results (details
are given in Section~\ref{sec:nat-trans-express}).
\begin{defn}\label{def:firstcounital}
  Let $\cat{A}$ be a category, and let $\cat{B}$ be a subcategory with inclusion
  functor $\fI:\cat{B}\to \cat{A}$. A \emph{counital} is a natural
  transformation $\iota:\fC\Rightarrow \fI$ for some functor $\fC:\cat{B}\to
  \cat{A}$. The class of all such counitals is denoted
  $\text{Counital}(\cat{B},\cat{A})$. For an object $X$ of $\acat{B}$, 
  the $X$-component of $\iota$ is the morphism $\iota_X: \fC(X) \to \fI(X)$ in $\acat{A}$.
\end{defn}

A special case of our results, for the category of groups, can be stated as follows.

\begin{mainthm}\label{thm:char-counital}
  For the category $\cat{Grp}$ of groups and subcategory $\lcore{Grp}$ 
  of groups and their isomorphisms, the following equalities of sets hold:
  \begin{center}
    \begin{tabular}{ccc} 
      $\{H \leq G ~|~ H \text{ characteristic in } G \} $
      & 
      $=$ 
      & 
      $\left\{
        \mathrm{Im}(\iota_G) ~\middle|~ \iota\in \text{\rm Counital}\left(\lcore{Grp},\cat{Grp}\right)
      \right\};$\\[8pt]
      $\{H \leq G ~|~ H\text{ fully invariant in }G \} $
      & 
      $=$ 
      & 
      $\left\{
        \mathrm{Im}(\iota_G) ~\middle|~ \iota\in \text{\rm Counital}(\cat{Grp},\cat{Grp})
      \right\}$. 
    \end{tabular}
  \end{center}
\end{mainthm}

Theorem~\ref{thm:char-counital} contrasts a ``recognizable'' description
of characteristic (fully invariant) subgroups with a ``constructive''
one. For a fixed group $G$, the sets on the left are of the form $\{H\mid P(G,H)\}$, where $P$
is the appropriate logical predicate that allows us to recognize when a subgroup $H$ 
belongs to the set; those on the right are of the form
$\{f(\iota) \mid \iota \in \text{Counital} (\ldots, \cat{Grp})\}$, where
$f(\iota)=\mathrm{Im}(\iota_G)$ allows us to  construct members of the
subset by applying a function. Also, the descriptions on the left are
``local" since they reference just a
single parent group, whereas those on the right are ``global" since they
apply to the
ambient categories.

The characterization of characteristic subgroups by natural
transformations allows one to recast the lattice theory of characteristic
subgroups into the globular compositions of natural transformations as
explored in \cites{Baez, Power}. 
We now explore 
other implications 
of Theorem~\ref{thm:char-counital}.

\subsection{Constraining isomorphism by characteristic subgroups}
\label{sec:intro-iso}
Characteristic subgroups 
constrain isomorphisms in the following sense:
\begin{fact}\label{fact:aut-iso}
  If $H$ is a characteristic subgroup of $G$, 
and $\alpha,\beta:G\to \tilde{G}$ are isomorphisms, 
  then $\alpha(H)=\beta(H)$.
\end{fact}

It is therefore useful for an isomorphism test to locate characteristic
subgroups of a group $G$: every hypothetical isomorphism from $G$ to $\tilde{G}$
must then assign such a subgroup $H$ to a unique corresponding subgroup
$\tilde{H}$ of $\tilde{G}$.  This raises at least two issues. First, if the task is
to construct isomorphisms, then we should assume that $\Aut(G)$ is not yet
known. How then do we verify that $H$ is characteristic? Is there an alternative
definition of the characteristic property that does not directly reference
$\Aut(G)$? A second issue is how to determine the possible $\tilde{H}\leq
\tilde{G}$ when we know only that $H$ is characteristic in $G$. For familiar
characteristic subgroups such as the center $\zeta(G)$ this is possible because
the definition is already global to all groups. Hence,  
a hypothetical isomorphism $\alpha:G\to \tilde{G}$ must satisfy
$\alpha(\zeta(G))=\zeta(\tilde{G})$, and typically $\zeta(G)$ and
$\zeta(\tilde{G})$ can be constructed without explicit knowledge of $\Aut(G)$ or
$\Aut(\tilde{G})$. However, the following family of examples, first explored by
Rottlaender \cite{Rottlander28}, exhibits groups whose characteristic subgroups
have no known global definition, so it is difficult to utilize
Fact~\ref{fact:aut-iso}.

\begin{ex}
  \label{ex:Rottlaender}
  Let $p$ be a prime and $m<p$ a positive integer. Let $q\equiv 1\bmod{p}$ be a
  prime and denote by $\mathbb{F}_q$ the field with $q$ elements. Let
  $\theta\in\GL_m(\mathbb{F}_q)$, with $\theta^p=1$, be diagonalizable with $m$
  eigenvalues $a_1,\ldots,a_m$, each different from $1$, 
satisfying the following
  property: if there exists $u\in \{1,\dots, p-1\}$ with $a_i^u=a_j$ for all $i\neq
  j$, then $p\nmid(u^k-1)$ for $k\in \{1,\dots, m\}$. 
For $m=2$, this requires $a_1\ne a_2^{\pm 1}$.
  
  The cyclic group $C_p$ of order $p$ acts on the vector space
  $V=\mathbb{F}_q^m$ via $\theta$. The condition on $\theta$ means that each eigenspace
  in $V$ is a characteristic subgroup of the semidirect product
$G_\theta=C_p\ltimes_\theta V$ determined by $\theta$, and
exactly $m$ of the $1+q+q^2+\cdots+q^{m-1}$ order $q$ subgroups of $G_\theta$ 
are characteristic. Two such groups $G_\theta$ and $G_\tau$ may be
  isomorphic even if the eigenvalues of $\theta$ and $\tau$ are different. For
  example, this occurs when $\tau=\theta^j$ for some $j$ coprime to $p$. 
Thus, the correspondence between characteristic subgroups of 
$G_\theta$ and $G_\tau$ is not {\em a priori} clear. 
\exqed
\end{ex}

One of the goals of this work is to reinterpret the definition of a
characteristic subgroup in a way that is independent of automorphisms and which
is unambiguously defined for all groups.  We do this by formulating the
characteristic condition on the entire category of groups, thereby providing a
categorification of the property of being characteristic. Moreover, our
formulation pairs well with---and indeed is motivated by---the necessities of
computation (see Section~\ref{sec:apps_type}). To address this, we employ 
methods from theorem-checking, specifically type-theoretic
techniques \cite{Hindley-Seldin,Pierce:types,HoTT}; these have recently become
accessible through systems such as Agda \cite{Agda}, Coq \cite{Coq}, and Lean
\cite{lean}.

\subsection{A local-to-global problem}
\label{sec:local-to-global}
Our approach is to transform the local characteristic property of subgroups into
an equivalent global property of the category of all groups and their
isomorphisms. Calculations now take place within the category instead of within
individual groups, which opens up new ways to search for characteristic
subgroups. Our approach also facilitates an {\it a priori} verification of the
global characteristic property, rather than the usual {\it a posteriori} check
that requires knowledge of automorphisms. The process is analogous to proving
that $\zeta(G)$ is characteristic without employing specific properties of $G$.
Our methods extend to \textit{every characteristic subgroup},
even those discovered via bespoke calculations.

The traditional model of a category $\cA$ involves both objects and morphisms. By
sometimes focusing only on morphisms, we work with categories as an algebraic 
structure with a partial binary associative product on $\cA$---given by composition 
of its morphisms---and with identities
$\one_{\cA}=\{\id_X\mid X$ an object in $\cA\}$. It
is partial because not every pair of morphisms is composable, in which case the
product is undefined. 
This perspective yields an algebraic framework for our computations.

The morphisms of a category can act on the morphisms of another category 
either on the left or the right. Although several interpretations of ``category
action" appear in the literature~\citelist{\cite{Bergner-Hackney}*{\S 2}
\cite{nlab:action} \cite{FS}*{1.271--274}}, there is no single established meaning.
Our formulation uses \textit{partial functions} that are purposefully undefined for some
inputs; see Section~\ref{sec:partial-functions} for a precise definition.
Let $\acat{A}$, $\acat{B}$, and $\acat{X}$ be categories. 
A left \emph{$\acat{A}$-action} on $\acat{X}$ is a partial function,
where $a\cdot x$ is defined for some morphisms $a$ of $\acat{A}$ and $x$ of $\acat{X}$,
that satisfies two conditions inspired by group actions. The first is that  
$(a\acute{a})\cdot x=a\cdot (\acute{a}\cdot x)$, whenever defined, for all morphisms 
$a,\acute{a}$ of $\acat{A}$ and $x$ of $\acat{X}$.
The second is that 
$\one_{\acat{A}}\cdot x=\{x\}$; 
to simplify notation we write $\one_{\acat{A}}\cdot x=x$.
As in the theory of bimodules of rings, an
\emph{$(\acat{A},\acat{B})$-biaction} on $\acat{X}$ is a left
$\acat{A}$-action and a right $\acat{B}$-action on $\acat{X}$ such 
that for every morphism $a$ in $\acat{A}$, $b$ in $\acat{B}$, and 
$x$ in $\acat{X}$,
\[ 
  a\cdot (x \cdot b) = (a \cdot x) \cdot b
\] 
whenever both sides of the equation are defined. 
For $(\acat{A},\acat{B})$-biactions on categories 
$\acat{X}$ and $\acat{Y}$, 
an \emph{$(\cat{A},\cat{B})$-morphism} 
is a partial function $\mathcal{M}:{\acat{Y}}\to \acat{X}$
such that 
\[ 
  \mathcal{M}(a\cdot y\cdot b)=a\cdot \mathcal{M}(y)\cdot b
\] 
whenever $a\cdot y\cdot b$ is defined for morphisms $a$ in $\cat{A}$, 
$b$ in $\cat{B}$, and $y$ in $\acat{Y}$.

We write $\acat{A}\leq \acat{B}$
to indicate that $\acat{A}$ is a subcategory of $\acat{B}$, and denote 
the identity functor of $\acat{A}$
by $\id_{\acat{A}} : \acat{A} \to \acat{A}$.
A \emph{counit} is a counital of the form $\eta: \mathcal{C}\Rightarrow \id_\acat{A}$. 
The following specialization of one of our principal results to groups 
describes how characteristic subgroups relate to counits and morphisms of category 
biactions.\enlargethispage{0.6cm}

\begin{mainthm} 
  \label{thm:char-repn}
   Let $G$ be a group and $H\leq G$ with inclusion $\iota_G:H\hookrightarrow G$.
   There exist categories $\cat{A}$ and $\cat{B}$, where
   $\lcore{Grp}\;\leq\acat{A}\leq\acat{Grp}$, such that the following are
   equivalent.
  \begin{ithm}
  \item[\rm (1)] $H$ is characteristic in $G$.
   \item[\rm(2)] There is a functor $\func{C} : \cat{A} \to \cat{A}$ and
     a counit $\eta:\func{C}\Rightarrow \id_{\cat{A}}$ such that 
     $H = \im(\eta_G)$.

     \item[\rm (3)] There is an $(\acat{A},\acat{B})$-morphism $\mathcal{M}:\acat{B}\to
    \acat{A}$ such that $\iota_G=\mathcal{M}(\id_G \cdot \one_{\acat{B}})$.
  \end{ithm} 
\end{mainthm}
\noindent 
We emphasize that the category $\cat{B}$ in Theorem~\ref{thm:char-repn} need not be a subcategory of $\cat{Grp}$; see   
Section \ref{sec:compose} for an example. Moreover, our results apply to characteristic substructures of
varieties of algebraic structures,
which include  monoids, loops, rings, and non-associative
algebras. This generalization (Theorem~\ref{thm:char-repn-eastern})
and its dual  version (Theorem~\ref{thm:general-dual}) are proved in  Section~\ref{sec:inv-cat}. We 
now illustrate how natural transformations arise from characteristic substructures.

\begin{ex}
\label{ex:three-chars_first}
 The derived subgroup $\gamma_2(G)$ of a group $G$ determines the inclusion
  homomorphism $\lambda_G:\gamma_2(G)\hookrightarrow G$ and a functor $\func{D}
  : \cat{Grp} \to \cat{Grp}$ mapping groups to their derived subgroup and
  mapping homomorphisms to their restriction onto the derived subgroups.
  For every group homomorphism $\varphi : G \to H$, observe that 
  $\lambda_H\func{D}(\varphi) = \id_{\cat{Grp}}(\varphi)\lambda_G$, so $\lambda
  : \func{D} \Rightarrow \id_{\cat{Grp}}$ is a natural transformation.

  The center $\zeta(G)$ of $G$ yields the
  inclusion homomorphism $\rho_G:\zeta(G)\hookrightarrow G$. 
To define a functor with
  object map $G\mapsto \zeta(G)$, we must restrict the type of homomorphisms
  between groups since homomorphisms need not map centers to centers. (Consider,
  for example, an embedding $\mathbb{Z}/2\hookrightarrow \text{Sym}(3)$.) Since
  every isomorphism  maps center to center, we restrict to $\lcore{Grp}$, defining a functor $\func{Z} : \;\lcore{Grp}\; \to
  \;\lcore{Grp}$ mapping $G\mapsto \zeta(G)$ and mapping each homomorphism to its
  restriction. If $\func{I} : \; \lcore{Grp}\; \to \cat{Grp}$ is the inclusion functor, then $\rho :
  \func{I}\func{Z}\Rightarrow \func{I}$ is a natural transformation.~\exqed
\end{ex}

\subsection{Applications to computation}
\label{sec:apps_type}
Part of the motivation for our work comes from computational challenges that
arise in contemporary isomorphism tests in algebra. 
One of these is to develop new ways to discover characteristic subgroups. 
Standard constructions---such as the commutator subgroup, 
the center, and the Fitting subgroup---can be applied to any group.  
However, these subgroups often contribute little to resolving isomorphism.
Many ideas have been introduced to search for   
new structures; see, for example, 
~\citelist{\cite{BOW}\cite{ELGO2002}\cite{Maglione2021}}. 
Often these involve detailed computations with 
individual groups, and 
their application 
is {\it ad hoc}. Indeed, a primary motivation for 
this study is to systematize the disparate techniques 
currently used to search for characteristic subgroups.

Theorem~\ref{thm:char-repn} provides the framework for a systematic search for
characteristic subgroups. An $(\acat{A},\acat{B})$-morphism generalizes
the familiar and much studied category theory notion of adjoint functor pairs.
We show in Section~\ref{sec:comp-nat-trans} that category actions offer a
flexible way to implement the behavior of natural transformations in a computer
algebra system. To exploit the full power of the categorical interpretation of
characteristic subgroups, we work in a suitably general algebraic framework that
allows a seamless transfer of information from one category to another. The
familiar examples from Sections~\ref{sec:examples} and \ref{sec:compose} demonstrate how to
identify characteristic structure in a category and transfer it back to groups.

A second challenge concerns reproducibility and comparison of characteristic
subgroups. 
  Algorithms to decide isomorphism between groups $G$ and $H$ often, as a first step, 
  generate lists of characteristic subgroups for $G$ and $H$, respectively. To exploit 
  these lists, any hypothetical isomorphism $G\to H$ must map the first list to the 
  second (Fact~\ref{fact:aut-iso}). We observed in Example~\ref{ex:Rottlaender} that 
  it is not always possible to determine a `canonical ordering'  of characteristic 
  subgroups such that this is guaranteed. In practice, some constructions employ randomization
or make labelling choices that vary from one run to the next. 
These variations can limit the utility of characteristic subgroups in
deciding isomorphism.

Our proposed solution is to develop algorithms that return the natural
transformation (or a morphism of biactions) from Theorem~\ref{thm:char-repn} 
instead of the characteristic subgroup
itself. This will allow us, in principle, to extend the reach of a specific
characteristic subgroup of a given group to an entire category, in much the same
way that the commutator subgroup and center behave. 
The natural transformation can then be applied to a group $\tilde{G}$ to
produce a characteristic subgroup $\tilde{H}$ that corresponds to $H$ in the
sense of Fact~\ref{fact:aut-iso}: every isomorphism $G\to\tilde{G}$ necessarily maps
$H$ to $\tilde{H}$, so allowing a meaningful comparison of characteristic
subgroups.

A third challenge is verifiability: in a computer algebra system, 
subgroups are often given by monomorphisms which are defined on a given 
generating set.
The construction of such a monomorphism usually invokes computations that 
\emph{prove} 
the claimed properties (such as homomorphism or characteristic image). 
We present our work in a framework that combines these computations, data, 
and proofs, by employing 
an intuitionistic Martin-L\"of type theory; such a model also allows machine verification of proofs.  
In this setting, if a computer algebra system returns a counital $\iota$, then this counital comes 
with a ``type" that \emph{certifies} that each morphism $\iota_G$ of $\iota$ yields a characteristic substructure. 

\subsection{Structure of this paper}
In Section~\ref{sec:type}, we discuss the required background for our
foundations (type theory). In Section~\ref{sec:Eastern}, we first review 
varieties of algebraic structures 
and then show how to model
categories as terms of the variety of \emph{abstract~categories}.

Section~\ref{sec:actions} studies category actions. In particular, we define
\emph{capsules} (category modules) and describe a computational model for
natural transformations  as category bimorphisms
(Proposition~\ref{prop:nat-trans-biact}). This also allows us to describe  
counitals (Theorem~\ref{thm:counit-capsules}) and adjoint functor pairs
(Theorem~\ref{thm:iso-biacts-adjoints}) in the language of bicapsules and bimorphisms.

In Section~\ref{sec:induced}, we explain how characteristic structures can be
described by counitals.  
The functors involved in this construction are defined on categories with one
object, but Theorem~\ref{thm:extension}---which we call the \emph{Extension
Theorem}---allows us to extend these functors to larger categories. This theorem
is the essential ingredient for proving  our main results. We  also generalize
Theorem~\ref{thm:char-counital} to varieties of algebras
(Theorem~\ref{thm:char-counital-eastern}).

In Section~\ref{sec:inv-cat}, we generalize  Theorem~\ref{thm:char-repn} to 
varieties of algebras (Theorem~\ref{thm:char-repn-eastern}). We show that characteristic
substructures can be described as certain counits, and as bimorphism actions on
capsules. We also prove the dual version of this result for characteristic
quotients (Theorem~\ref{thm:general-dual}).

In Section~\ref{sec:examples}, we use our framework to provide categorical
descriptions of common characteristic subgroups, including verbal and marginal
subgroups. 

In Section~\ref{sec:compose}, we describe a cross-category translation of
counitals and explain, in categorical terms, how a counital for a category of
groups can be constructed from a counital for a  category of algebras.

In Section~\ref{sec:imp}, we 
report on an implementation of some of the concepts introduced in this work.

Table~\ref{tab:notation-table} summarizes notation used throughout the
paper.

\begin{table}[ht]
\begin{center}
{
\begin{tabular}{l|l}
  Symbol & Description  \\ \hline \hline 
  $\cat{E}$ & variety of algebraic structures \\
  $\cA, \cB,\cC,\cD$ & Abstract categories or categories that act  \\
  $\aX, \aY, \aZ$ & Capsules  \\
  $\Delta,\Sigma$ & Bicapsules\\
  $\id_X$ & Identity morphism of type $X$  \\
  $\one_{\acat{A}}$ & Identity morphisms of $\acat{A}$  \\
  $\func{F},\func{G}$ & Morphisms between categories  \\
  $\func{I}, \func{J},\func{K},\func{L}$ & Inclusion functors \\ 
  $\func{M},\func{N},\fR,\fS$ & Capsule morphisms\\
  $A^X$ &  
    Functions $X\to A$  \\
  $A^n$ & 
    Functions $\{1,\dots, n\} \to A$  \\
  $\Omega$ & Signature  \\
  $\mathrm{Alge}_\Omega$  & Type of $\Omega$-algebras  \\
  $\mathrm{Alge}_{\Omega,\mathcal{L}}$ & Type of $\Omega$-algebras in the variety for the laws $\mathcal{L}$\\  
  $\bot$ & The void type  \\
  $B^?$ & The type $B\sqcup \{\bot\}$  \\
  $f(a)\venturi b$ & If $f(a)$ is defined, then $f(a)=b$  \\
    $\src{f},\tgt{f}$ & Source and target of a morphism\\
    $f\lhd , \lhd f$ & Guards for a category action\\
  $\core{A}$, $\epi{A}$, $\mono{A}$ & The iso-, epi-, and mono-morphisms of $\acat{A}$ (resp.)  \\  
\end{tabular}
}
\end{center}
\caption{A guide to notation}\label{tab:notation-table}
\end{table}

\section{Type theory and certifying characteristic structure}
\label{sec:type}
The emergence of randomized methods in computational algebra has 
elevated the significance of certification. Certificates are used to upgrade 
Monte Carlo algorithms
to Las Vegas algorithms, where 
every output (other than failure) is correct~\cite[\S 3.2.1]{handbook}.
Usually this certification occurs \textit{a posteriori}, 
which can present an intractable obstacle. 
Suppose an algorithm 
constructs a characteristic subgroup $H$ of a group $G$ with inclusion $\iota\colon H\hookrightarrow G$.
To certify that $H$ is characteristic, we must verify that 
\begin{equation}\label{eq:char-logic}
\begin{array}{llll}
   (\forall \varphi\in \Aut(G)) &(\forall h\in H) & (\exists k\in H)& \varphi(\iota(h))=\iota(k).
 \end{array}
\end{equation}
An obstacle to certification 
is that the algorithm may not know 
$\Aut(G)$ explicitly. 
The very construction 
of $\Aut(G)$ is often one of the key reasons to find characteristic 
subgroups in the first place. What is needed is \textit{a priori} certification.

Of course, certain constructions yield subgroups of $G$ which 
are guaranteed to be characteristic;
these include $\zeta (G)$ and $\gamma_2(G)$.
Their constructions, and the 
reasons they produce characteristic subgroups, apply to all groups. A 
careful examination of these reasons on the categorical level leads to the 
key insight of this paper:
there is a uniform 
categorical description of the characteristic property. 
As we shall see, this insight ultimately leads to the possibility 
of \textit{a priori} certification.

 To put this into practice, we develop a
constructive version of our main results using type theory language.
Specifically, we use an intuitionistic Martin-L\"of type theory (MLTT), a model
of computation capable of expressing aspects of proofs that can be machine
verified. 
An advantage of this approach is that
certificate data can be verified by practical  
type-checkers. 
An MLTT employs the ``propositions as types" paradigm 
(Curry--Howard Correspondence),
where types correspond to propositions and terms are programs that 
correspond to proofs. The remainder of this section is a concise treatment 
of type theory from
\citelist{
\cite{Hindley-Seldin}*{Chapters~10--13}
\cite{HoTT}*{Chapter~3}}.

\subsection{Types}
\label{sec:types}
Informally, \emph{types} annotate data by signalling which syntax rules apply
to the data. We write $a:A$ and 
say ``$a$ is a \emph{term} of type $A$'' or ``$a$
inhabits $A$''.  
For example, $a: \mathbb{N}$
signals that $a$ can only be used as a natural number. 
A type $A$ is \emph{inhabited} if there exists at
least one term $a:A$ and \emph{uninhabited} if no term of type $A$ exists.  
The \emph{void} type $\bot$ has no inhabitants by definition.  
Deciding whether a type is inhabited or not is computationally 
undecidable \cite{Hindley-Seldin}*{pp.~66--67}. Therefore, 
in computational settings, 
types are permitted to be neither inhabited nor uninhabited.
Type annotations enable us to use symbols according to their logical
purpose; for example, $a:A$ is analogous to $a\in A$,
but type theories do not have the axioms of set theory.

Types are introduced from two sources.  
Some are predefined by the \emph{context}: they are 
given \emph{a priori}, such as the type of natural numbers $\mathbb{N}$. 
Others are created using \emph{type-builders}: these construct new types 
from existing ones.
We use both $\type{A}\to \type{B}$ and $\type{B}^{\type{A}}$ to
denote the type of functions, and set $\Dom (A\to B) = A$ and 
$\Codom (A\to B) = B$. If $n$ is a natural number, then an inhabitant of
type $\type{A}^n$ can be interpreted as an
$n$-tuple $(a_1,\ldots,a_n)$ with each $a_i: \type{A}$, or alternatively 
as a function $\{1,\ldots,n\}\to \type{A}$.  There is a unique function 
$\bot\to A$ (akin to the uniqueness of a function $\varnothing \to A$), 
so $A^0$ is a type with a single inhabitant---it is \emph{not} void.

The notation $\prod_{i\tin I}\type{A}_i$ together with projection maps $\pi_i:
\left(\prod_{i\tin I}\type{A}_i\right) \to A_i$ is used for Cartesian products,
and $\bigsqcup_{i\tin I} \type{A}_i$ together with inclusion maps $\iota_i : A_i
\to \bigsqcup_{i\tin I} \type{A}_i$ is used for disjoint unions. (The tradition
in type theory is to use $\sum_{i\tin I}\type{A}_i$ instead of $\bigsqcup_{i\tin
I} \type{A}_i$, but this conflicts with algebraic uses of $\Sigma$.)

\subsection{Propositions as types}\label{sec-propositions-as-types}
In set theory, propositions are part of the existing foundations. 
In type theory, propositions co-evolve with the theory as special types. A
proposition $P$ in logic is associated to a type $\hat{P}:\mathrm{Type}$. 
(Only in this section do we distinguish propositions $P$ in logic from 
propositions as types with the notation~$\hat{P}$.) 
If the type $\hat{P}$ is
inhabited by data $\pf:\hat{P}$, then the term $\pf$ is regarded as a proof that
$P$ is true. For example, an implication $P \Rightarrow Q$ (here
$\Rightarrow$ means ``implies" with weakening and contraction laws) can be proved 
by means of a function $f:\hat{P}\to \hat{Q}$, where $\hat{P}$ and $\hat{Q}$ are 
the respective types associated with $P$ and $Q$, because it suffices to assume $P$ 
and derive a proof of $Q$.  Likewise, if we assume that there is a term $\pf:\hat{P}$ 
and apply the function $f$, then it produces a term $f(\pf):\hat{Q}$. 

In classical logic, it is only the existence of a proof for a proposition that
is relevant. Analogously, in type theory,
$\hat{P}:\mathrm{Type}$ is a \emph{mere proposition}, 
written $\hat{P}:\mathrm{Prop}$, if it has at most one inhabitant.  

Consider the function $\hat{P}:A\to \mathrm{Prop}$. 
Now $(\forall a\in A)(P(a))$ and $(\exists a\in A)(P(a))$ are expressed by terms of type 
$\prod_{a:A}\hat{P}_a:\mathrm{Prop}$ and
$\|\bigsqcup_{a:A} \hat{P}_a\|:\mathrm{Prop}$, respectively, where $\|A\|$ {\it truncates} a type to a 
single term if it has any terms~\cite{HoTT}*{\S 3.7}. The negation of a proposition $P$ is
$P\Rightarrow \textsc{False}$, which accords with functions of type $\hat{P}\to
\bot$. For additional details, see
\citelist{\cite{Hindley-Seldin}*{Chapters~12--13}\cite{HoTT}*{Chapter~3}}.

\subsection{Equality}
\label{sec:equality}
In Zermelo set theories, all data are sets and there is a single notion of
equality afforded by the \textit{Axiom of extensionality}: two sets are equal
if, and only if, they have the same elements. In type theory, terms and types 
are separate entities, and this single axiom
is replaced by several distinct notions of equality more representative of
computational behavior.
Each type theory is built on a 
rewriting system (such as a $\lambda$-calculus or combinatory logic). 
Employing the language of \cite{Hindley-Seldin}*{\S 1D and \S 2D}, 
we judge data as equal if their normal forms in this rewriting system
coincide; and 
$s\defeq t$ followed by some sentences $M$ means that
``within the given scope $M$, the variable $s$ should be substituted by the data $t$".

Type theories include axioms that  allow equality after taking normal forms to
count as \emph{$($definitional$)$ equality}--the type theory sees no difference
between the data \cite{Hindley-Seldin}*{p.\ 193}. For example, if we build
the type $\mathbb{Z}/n$ which depends on a term $n:\mathbb{N}$, 
then some type systems judge that 
$\mathbb{Z}/(m+m)$ is equal to
$\mathbb{Z}/2m$ because $m+m$ and $2m$ have the same normal form.
But the function $\gcd(m,2m)$ is more complicated and its normal 
form may differ from $m$.  Hence, the type
system does not judge $\mathbb{Z}/\gcd(m,2m)$ as equal to $\mathbb{Z}/m$; 
neither does it assert they are not equal; instead it withholds judgment.

To construct an equality that mimics set theory,
Per Martin-L\"of developed a notion of 
propositional equality that imitates the \emph{Leibniz Law}~\cite{Feldman}:  
\begin{equation*}
    (s = t) \iff \left[(\forall P(x))\; P(s)\Longleftrightarrow P(t)\right],
\end{equation*}
where $P(x)$ runs over all predicates of a single variable $x$. 
For every type $A$ and terms $s,t:A$, we define 
an auxiliary type $s=_A t$, where terms are proofs that $s$ equals $t$, 
with the rule that, given a function $f:A\to B$, there is a function 
\begin{align}  \label{eqn:path-proof}
  \text{path}(f): (s =_A t)\to (f(s) =_B f(t)). 
\end{align}
For example, a proof $p: (\gcd(m,2m) =_{\mathbb{N}} m)$ can be transported along a path 
to $q: (\mathbb{Z}/\gcd(m,2m)$ $=_{\mathbb{Z}/m} \mathbb{Z}/m)$ allowing programs
to treat these types as equal. Thus, computational 
evidence enhances the reach of equality,  
see \citelist{\cite{Hindley-Seldin}*{\S 13D}\cite{HoTT}}.

For readability we often omit the subscript $A$ in $s=_A t$. 
By slight abuse of notation, writing ``$s=t$'' as a 
logical statement in text should be interpreted as  
``the type $s=_At$ is inhabited''. 

\subsection{Subtypes and inclusion functions}
\label{sec:subtypes}
Sets are a special case of types: 
we write $S:\mathrm{Set}$ for a type $S$ 
if the type $s=_S t$ is a mere
proposition for all $s,t:S$.
Let $A$ be a type. 
If $P : A \to \mathrm{Prop}$, then 
\[ 
  \{ a : \type{A} \,|\, P(a) \} \defeq \bigsqcup_{a: \type{A}} P(a),
\] 
is the \emph{subtype} of $\type{A}$ defined by $P$. We also write this as $B \defeq \{ a
: \type{A} \,|\, P(a) \}\subset A$. Terms of type $B$ have the form $\langle a,
p\rangle$ for $a:A$ and $p:\mathrm{Prop}$, where $p$ is a proof that $P(a)$ is
inhabited. We sometimes use set theory notation to improve readability when
describing a subtype. For more details, see \cite{HoTT}*{\S 3.5}.
For a typed function $f:A\to B$, the image $\{f(a)\mid a:A\}$ is shorthand 
for $\{b:B\mid (\exists a:A)(f(a)=b)\}$.

Subtypes have an associated inclusion function 
$\alpha:B\to A$ where $\alpha(\langle
a,p\rangle)\defeq a$. 
A subtlety is that if $C\subset B$ with inclusion map
$\beta:C\to B$, then the composition $\alpha\beta:C\to A$ is injective 
but does not show directly that $C\subset A$.  A term of type
$C\defeq\bigsqcup_{b:B}Q(b)$, with $Q : B\to \mathrm{Prop}$, has the 
form $\langle \langle a,p\rangle,q\rangle$, which differs from terms 
of type $B$.   A small
modification addresses the fact that the relation $\subset$ is not strictly
transitive. Define a subtype $C' \defeq \bigsqcup_{a:A} R(a)$, where
$R(a)\defeq\bigsqcup_{p:P(a)}Q(a)$, and inclusion $\gamma : C' \to A$. 
Now construct a map $\sigma: C\to C'$ given by 
\begin{align*} 
  \langle \langle a,p\rangle,q\rangle \mapsto \langle a,\langle p,q\rangle\rangle,
\end{align*}
where $a:A$ and $\langle p,q\rangle \tin R(a)$. Thus,
$\alpha\beta=\gamma\sigma$, and the composition $\alpha\beta$ is
equivalent to $\gamma$.  Hence, $\subset$ is transitive up to this equivalence. 

\subsection{Partial functions}
\label{sec:partial-functions}
In type theory, functions are ultimately programs so they may fail to halt.
Since our concern lies with algebraic
obstacles rather than decidability, we confine our model to algebras that 
have decidable operations, such as polynomial and integer operations,
 look-up tables, and strongly normalizing rewriting systems.
Thus, all functions are total: given an input, they produce an 
output. However, it is helpful to identify inputs we regard 
as ``undefined", or ``leading to errors". For example, 
a division operator may allow $0$ as an input and return an error token as output.  
We call such functions \emph{partial functions} and regard them as 
``undefined'' at such inputs.

To accommodate such partial operations, we extend types by adjoining the symbol
$\bot$ (the void type) to represent ``undefined''. For a type $A$, we define \begin{align} 
  \label{eq:A-lift}
  A^? \defeq A\sqcup \{\bot\}
\end{align}
with inclusion $\iota_A:A\hookrightarrow A^?$. For $a:A^?$, we write $a:A$ in
this setting as shorthand for ``there exists $a':A$ such that $a=\iota_A(a')$".
This allows us to define an endofunctor $Q$ on the category of types that maps a
morphism $f:A\to B$ to $f^?:A^?\to B^?$, such that $f^?(a)=f(a)$ for $a:A$ and
$f^?(\bot)=\bot$. The canonical projection $\mu : QQ \Rightarrow Q$, with
components $\mu_A: (A^?)^? \to A?$, is a natural isomorphism. Hence, it suffices
to work with single applications of $Q$, and $\bot$ will serve as the designated
symbol for undefined elements throughout.
 
This discussion also motivates a notion of ``directional equality'' similar to
that in~\cite{FS}*{1.12}. For $a,b: A^?$, define  
\begin{align}
  \label{eq:venturi}
  (a\venturi b) &\defeq [(a:A) \to (a=b)] \defeq [((\exists a':A)(a=\iota_A(a'))) \to (\iota_A(a')=b)]. 
\end{align}

By slight abuse of notation, for function terms  $f,g : A^?\to B^?$ we denote
function extensionality also by $f=g$, that is, we define 
\[(f=g) \defeq \big[(\forall a:A^?)\; (f(a)=g(a))\big].\]

\subsection{Certifying that the trivial group is characteristic}
As an illustration, we present a type verifying the characteristic 
property of the trivial subgroup. Let $G : \mathrm{Group}$ be a group 
with identity $1:G$. Let $H\defeq\{x: G \mid x=1\}$ be the 
subtype of $G$ representing the trivial subgroup. Recall that terms of 
$H$ have the form $\langle x, p\rangle$, where $x:G$ and 
$p$ is a term of type $x=1$, 
and there is a map $\iota:H\to G$,  $\langle x,p\rangle\mapsto x$. 
If $h,k:H$, then  $\iota(h)=\iota(k)=1$, and, by \eqref{eqn:path-proof}, 
for every $\varphi:\Aut(G)$ there is an invertible function of type
\begin{align} \label{eq:my-invert}
  \varphi(1)=_G 1\; \longleftrightarrow\; \varphi(\iota(h))=_G \iota(k).
\end{align}
The latter function depends on $h$ and $k$, but  we suppress this dependency  
to simplify the exposition. Let $\mathrm{idLaw}(\varphi) : \varphi(1)=_{G} 1$
be a proof that $\varphi:\Aut(G)$ fixes $1:G$. Using \eqref{eq:my-invert}, we 
define the term 
\[
    \mathrm{idMap}(\varphi) : \prod_{h:H}\left\|\bigsqcup_{k:H} \varphi(\iota(h))
    =_G \iota(k)\right\|
\]
that takes as input $h:H$ and produces $\langle
1,\mathrm{idLaw}(\varphi)\rangle: \left\|\bigsqcup_{k:H} \varphi(\iota(h))=_G \iota(k)\right\|$.
Therefore we obtain the term 
\begin{align*}
  \mathrm{idMap} & : \prod_{\varphi:\Aut(G)}\prod_{h:H}\left\|\bigsqcup_{k:H}
\varphi(\iota(h))=_G \iota(k)\right\|, 
\end{align*}
which certifies that $H$ is
characteristic in $G$; compare to \eqref{eq:char-logic}. Recall that in MLTT,
types correspond to propositions, and terms are programs that correspond to
proofs. Thus, the term $\mathrm{idMap}$  is not an  exhaustive tuple listing
$\Aut(G)$, but a program (function) that takes as input $\varphi:\Aut(G)$ and
$h:H$, and produces $k:H$ and $p: \varphi(\iota(h)) =_G \iota(k)$.

\section{Algebraic structures and varieties}
\label{sec:Eastern} 
To interpret characteristic structure as computable categorical information, we
treat categories as algebraic structures.  (Computational
categories should not be confused with categorical semantics of computation.) 
For our purpose, it suffices 
to use operations that may only be partially defined, so categories are important examples, as are
monoids, groups, groupoids, rings, and non-associative algebras. We
give an abridged account and refer to  
\citelist{\cite{Cohn}*{\S II.2} \cite{AR1994:categories}*{Chapter~3}}
for details.

\subsection{Intentional and extensional formulations of algebra}
\label{sec:int-ext}
It is natural to ask if algebraic structures such as groups and rings, 
that are introduced in standard texts such as~\cite{Hungerford} using 
\textit{extensional}
set theory, have logically consistent \textit{intentional}
formulations in foundations such as MLTT.
While it is not within our purview 
to consider alternative 
foundations of algebra, we briefly compare type-theoretic 
formulations of 
groups with their long-standing 
and rigorous 
treatment in \textit{computational} algebra.

Although systems such as {\sf GAP}~\cite{GAP4}, Macauley~\cite{Macauley},
{\sc Magma}~\cite{magma}, and SageMath~\cite{sagemath} are not designed 
to use types robustly, 
they nevertheless facilitate a treatment of 
groups that is more intentional than extensional. In these systems, 
groups can be represented in many different ways, but for practical reasons 
\textit{they are generally not treated as sets of elements}. 
For instance, a group $G$ may be specified by a generating set $Y$ of 
permutations. 
Algorithms such as the \textit{product replacement}~\cite[\S 3.2.2]{handbook} 
can then be used to select ``random" elements of $G$ as words in $Y$.
However, basic questions such as 
membership---\textit{does a permutation belong to $G$?}---often require 
clever algorithms to answer~\cite[Chapter 4]{handbook}.

Even the question of whether two elements in a group are equal is often not immediate.
(There are models, such as finitely-presented groups, where this question is 
not decidable.) In standard models for computation with groups,
effective equality testing is usually available, but it is not always
done by simply asking if two pieces of data are identical. For example, 
do elements $a$ and $b$ of $G=\langle Y\rangle$ coincide in $G / \zeta(G)$? 
Equivalently, is $ab^{-1} \in \zeta(G)$?  Even if we do not know generators for $\zeta(G)$, 
we can answer the latter question efficiently by deciding whether 
$ab^{-1}$ commutes with every element of $Y$. 
In this sense, asking whether $a=b$ in computational algebra (where a program 
settles the question) is closer to writing $a=b$ in type theory (where evidence 
is provided by a proof) than it is to the (trivial) 
question in set theory (cf.~Section~\ref{sec:equality}).

\subsection{Operators, grammars, and signatures}
\label{sec:ops-grammar}
Informally, a grammar is a description of rules for formulas. 

\begin{defn}
  An \emph{operator} is a symbol with a grammar, which we describe 
  using the Backus--Naur Form (BNF)~\cite{Pierce:types}*{p.~24}.  
  The \textit{valence} of an operator $\omega$,
  written $|\omega|$, is the number of parameters in its grammar.
  A set $\Omega$ of operators is a \emph{signature}.
\end{defn}

\begin{ex}\label{ex:additive-grammar}
  A signature for additive formulas specifies three operators:
  \begin{equation*}
    \label{eq:abelian-grammar}
    \texttt{<Add> ::= (<Add> + <Add>)~|~0~|~(-<Add>)}
  \end{equation*}
  The bivalent addition $(+)$ depends on terms to the left and right;
  zero ($0$) depends on nothing; and univalent negation ($-$) is followed by a term. \exqed
\end{ex}
It is easy to reject $+-+\,2\,3\,7$ since it is not  meaningful. However, we might write
$2+3-7$ intending $(2+3)+(-7)$; the BNF grammar  
\texttt{<Add>} accepts only the latter.

The purpose of the signature is to formulate 
important algebraic concepts such as  
homomorphisms.  
To declare that a 
function $f:A\to B$ is a homomorphism between additive groups, 
we use the signature of Example~\ref{ex:additive-grammar} as follows:
\begin{align*} 
  f((x+y)) & = (f(x)+f(y)), & f(0) & =0, & f((-x)) & = (- f(x)).
\end{align*}

\subsection{Algebraic structures}
An algebra is a single type with a signature \cite{Cohn}*{\S II.2}.

\begin{defn}
  \label{def:alg-structure}
  An \emph{algebraic structure} with signature $\Omega$ 
is a type $A$ and a function $\omega\mapsto \omega_A$, 
where $\omega:\Omega$ and $\omega_A:A^{|\omega|}\to A$.
  A \emph{homomorphism} between algebraic structures $A$ and $B$, 
each having signature $\Omega$, 
  is a function $f:A\to B$ such that, for every $\omega:\Omega$ and 
  $a_1,\ldots,a_{|\omega|}:A$,
  \begin{align*}
    f(\omega_A(a_1,\ldots, a_{|\omega|})) & = \omega_B(f(a_1),\ldots, f(a_{|\omega|})).
  \end{align*}
As in Section~\ref{sec-propositions-as-types}, we extend these propositions to 
  types as follows:
  \begin{align*}
    \mathrm{Alge}_{\Omega} & \defeq \bigsqcup_{A:\mathrm{Type}}~\prod_{\omega:\Omega} (A^{|\omega|}\to A)
    \\
    \mathrm{Hom}_{\Omega}(A,B) & \defeq \!\bigsqcup_{f:A\to B}~\prod_{\omega:\Omega}~\prod_{a:A^{|\omega|}}\!
  f(\omega_A(a_1,\ldots,a_{|\omega|})) =_B \omega_B(f(a_1),\ldots,f(a_{|\omega|})).
  \end{align*}
  Terms of type $\mathrm{Alge}_{\Omega}$ are \emph{$\Omega$-algebras}.
\end{defn}

For example, consider the additive group signature from
Example~\ref{ex:additive-grammar}. The underlying structure of an additive group
can be described by a type (set) $A$ together with assignments of the operators
in \texttt{Add} such as \texttt{(<Add> + <Add>)} to $+_A : A\times A\to A$. The
nullary operator \texttt{0} is then identified with a term $0:A$.

\subsection{Free algebras and formulas}\label{sec:free}

We now extend signatures to include variables that allow us to work with formulas.

\begin{defn}
  Let $\Omega$ be a signature and 
let $X$ be a type whose terms are \emph{variables}. The
  \emph{free} $\Omega$-algebra in variables $X$, denoted by $\Free{\Omega}{X}$,
  is the type of every formula in $X$ constructed using the operators in
  $\Omega$. 
\end{defn}

\begin{ex}
  To describe formulas in variables $x,y$ and $z$, 
we extend the additive signature
  $\Omega =\texttt{Add}$ of Example~\ref{ex:additive-grammar} as follows: 
   \begin{equation*}
     \label{eq:free-grammar}
     \begin{split}
     \texttt{<Add<X>>} \texttt{ ::= (<Add<X>> + <Add<X>>)  | 0 | (-<Add<X>>) |  x | y | z}.
     \end{split}
   \end{equation*}
   Here, $x+y$ and $(-x)+(0+z)$ have type $\texttt{Add}\langle X\rangle$,
   but $x-$ and $x+7$ do not. 
   The operations on the formulas $\Phi_1(X),\Phi_2(X):\Free{\texttt{Add}}{X}$ are:
   \begin{align*}
      \Phi_1(X) +_{\Free{\texttt{Add}}{X}} \Phi_2(X) 
      & \defeq ( \Phi_1(X) + \Phi_2(X) )\\
      0_{\Free{\texttt{Add}}{X}} 
      & \defeq 0\\
      -_{\Free{\texttt{Add}}{X}} \Phi_1(X) 
      & \defeq (-\Phi_1(X)).
   \end{align*}
   Thus, $\texttt{Add}\langle X\rangle$ is the free additive algebra,  but it
   lacks laws such as $x+y = y+x$ and $x + (-x) = 0$. We explain how to impose
   these laws in  Section~\ref{sec_laws}. 
   \exqed
 \end{ex}

 \begin{fact}
  \label{fact:ump}
    Let $A$ be an $\Omega$-algebra and $a:A^X$, where 
    $X$ is a type whose terms are variables. 
    There is a unique homomorphism $\mathrm{eval}_a:\Omega\langle
    X\rangle \to A$ that satisfies 
    $\mathrm{eval}_a(x)  = a_x $.
\end{fact}
\noindent
Consequently, we write $\Phi(a)\defeq \eval_a(\Phi)$ for formulas
$\Phi:\Free{\Omega}{X}$ and $a:A^X$.

\begin{rem}
  The construction in~Fact~\ref{fact:ump} is categorical in nature, and we use
  it in Section~\ref{sec:examples} to construct characteristic subgroups. The
  category of $\Omega$-algebras has objects of type $\mathrm{Alge}_{\Omega}$
  together with homomorphisms. The pair of functors (given only by their object
  maps)
  \[
    X\mapsto \Free{\Omega}{X}\quad\text{ and }\quad
    \langle A:\mathrm{Type},\ (\omega:\Omega)\mapsto (\omega_A:A^{|\omega|}\to A)\rangle\mapsto A
  \]
  forms an \emph{adjoint functor pair} between the categories of types and
  $\Omega$-algebras; see Section~\ref{sec:biacts-adjoints} for related discussion. 
\end{rem}

\subsection{Laws and varieties}\label{sec_laws}
Let $\Omega$ be a signature. We
now describe the variety of $\Omega$-algebras whose operators satisfy a list
of (equational) laws such as the axioms of a group.  Let $X$ be a type for variables. A \emph{law} is a term of
type $\Free{\Omega}{X}^2$. We index laws by a type $L$, so they
are terms $\mathcal{L}: L \to \Omega\langle X\rangle^2$ and are written $\ell \mapsto
(\Lambda_{1,\ell}, \Lambda_{2,\ell})$. 

An $\Omega$-algebra ${\type{A}}$ is
in the \emph{variety} for the laws $\mathcal{L}: L \to \Omega\langle
X\rangle^2$ if 
\begin{align*}
\begin{array}{lll}
  (\forall a\tin {\type{A}}^X)& (\forall \ell\tin {L})& \Lambda_{1,\ell}(a) = \Lambda_{2,\ell}(a).
\end{array}
\end{align*}
We write $\mathrm{Alge}_{\Omega,\mathcal{L}}$ for the type of all $\Omega$-algebras in the variety for the laws $\mathcal{L}$; this is a subtype of  $\mathrm{Alge}_{\Omega}$. The category of  $\Omega$-algebras in the variety for $\mathcal{L}$ has object type  $\mathrm{Alge}_{\Omega,\mathcal{L}}$ and morphism type
  \[ 
  \mathrm{Hom}_{\Omega,\mathcal{L}} \defeq \bigsqcup_{A : \mathrm{Alge}_{\Omega,\mathcal{L}}}~\bigsqcup_{B : \mathrm{Alge}_{\Omega,\mathcal{L}}} \mathrm{Hom}_{\Omega}(A,B),
  \]
with $\mathrm{Hom}_{\Omega}(A,B)$ as in Definition \ref{def:alg-structure}.

\begin{ex}\label{ex:group-grammar}
  The signature $\Omega$ for groups is the following:
  \begin{align*}
    \texttt{<G> ::= (<G><G>) | 1 | (<G>)$^{-1}$}  .
  \end{align*}
  The variety of groups uses three laws, indexed by 
  $L\defeq\{{\tt asc}, {\tt id}, {\tt inv}\}$ with variables 
  $X\defeq\{x,y,z\}$, where for example
  \[ 
    \Lambda_{1,\texttt{asc}} \defeq x(yz) \quad \text{and} \quad \Lambda_{2, \texttt{asc}} \defeq (xy)z.
  \] 
  Thus, 
  $\Lambda_{1,\texttt{asc}}(g,h,k)=g(hk)$ and $\Lambda_{2,\texttt{asc}}(g,h,k)=(gh)k$, 
  and associativity is imposed on the $\Omega$-algebra $G$ by requiring a term (``proof'') 
  of type
  \[ 
    \prod_{g:G}\prod_{h:G}\prod_{k:G} 
    g(hk)=_G (gh)k.
  \] 
  Encoding $1x=x$ and $x^{-1} x=1$ as additional laws 
  gives a complete description of the variety of groups.  
  Laws need not be algebraically independent:  for example, 
  $x1=x$ and $xx^{-1}=1$ are often also encoded. 
  \exqed 
\end{ex}
For clarity, henceforth we write laws as propositions.  For example, we write
$g(hk)=(gh)k$ rather than  terms of a mere proposition type.

\subsection{Categories as algebraic structures}\label{sec:cats-algestruct} 
We cannot always
compose a pair of morphisms in a category: composition may be a
partial function. Hence, the morphisms need not form an algebraic structure
under composition. We address this limitation by identifying precisely when the
operators yield partial functions. 

\begin{ex}
The type of each function is given as 
\[
  \mathrm{Fun}\defeq\bigsqcup_{\type{A}:\mathrm{Type}}~\bigsqcup_{\type{B}:\mathrm{Type}}(\type{A}\to \type{B}).
\]  
Technically, to quantify over all types, we shift  to a larger 
universe $\text{Type}_1$; see Remark~\ref{rem:[paradox]}.
For $f:\type{A}\to\type{B}$ and $g:\mathrm{Fun}$, define
\begin{align}
  \label{eq:def-comp}
  \src{f} & \defeq \id_{A},
  &
  \tgt{f} & \defeq \id_{B},
  &
  fg &\defeq \begin{cases} 
      f\circ g & \src{f}=\tgt{g},\\
      \bot & \text{otherwise}.
  \end{cases}
\end{align}
where $f\circ g$ is the usual composition of functions.    
The condition $\src{f}=\tgt{g}$ guards against composing non-composable
functions (one can think of $\src{f}=\tgt{g}$ as saying ``what enters $f$ must
match what exits $g$''). Note that
$\tgt{(\src{f})}=\tgt{\id_{A}}=\id_{A}=\src{f}$, and similarly
$\src{(\tgt{f})}=\tgt{f}$. 
\exqed
\end{ex}

The definitions in~\eqref{eq:def-comp} motivate an algebraic structure on $\mathrm{Fun}^?$.
We define the \emph{composition signature}:
\begin{equation}\label{eqn:comp-sig}
  \texttt{<Comp> ::= (<Comp><Comp>) | }(\tgt{\texttt{<Comp>}})\texttt{ | }(\src{\texttt{<Comp>}})\texttt{ | } \bot
\end{equation}

\begin{defn} 
  \label{def:abs-cat}
    Let $\Omega$ be the composition signature of~\eqref{eqn:comp-sig}. An
    \emph{abstract category} 
    $\cat{A}$ is an $\Omega$-algebra on a type $C$ satisfying the 
    law
    \[f(gh)   = (fg)h\]
    in variables $f,g,h$,
    together with the following \emph{source--target} laws and \emph{$\bot$-sink} laws:
    \begin{align*}
      \tgt{(\src{f})} & = \src{f} & (\tgt{f}) f & = f & \tgt{(fg)} & = \tgt{(f (\tgt{g}))}\\
      \src{(\tgt{f})} & = \tgt{f}      
      & 
      f (\src{f})  & = f
      & 
      \src{(fg)} & = \src{((\src{f})g)}
    \end{align*}
    \begin{align*}
      \src{\bot} &= \bot & \tgt{\bot} &= \bot & f\bot &= \bot & \bot f &= \bot.
    \end{align*}
\end{defn}

We  refer to the operators
    $\src{(-)}$ and $\tgt{(-)}$ in Definition~\ref{def:abs-cat} as \emph{guards}.  Note that $\tgt{f}=\bot$ or $\src{f}=\bot$ if, and only if, 
$f=\bot$; this follows from the laws $ (\tgt{f}) f  = f$ and 
$f(\src{f})=f$.

Conventional categories can be treated as abstract categories.
First, the morphisms of the category can be packaged as a disjoint union 
into a common type $A$, which possibly requires an enlarged universe. 
Then we use $A^?$ as the carrier type for the abstract category $\acat{A}$, 
where the nullary operator 
$\bot:\Omega$ is identified with the term $\bot$ in $A^?$; see \eqref{eq:A-lift}. 
 We write $a:\acat{A}$ to indicate that $a$ is a 
term of the carrier type $A^?$. 
Henceforth, 
we assume that all abstract categories have carrier types of the form $A^?$.

A useful subtype of an abstract category $\acat{A}$ is the 
type of \emph{identities}:
\[
    \one_{\acat{A}}\defeq\{\src{a}\mid a\tin\acat{A},\,a\neq \bot\}.
\] 
Since $\tgt{(\src{a})}=\src{a}$ and $\src{(\tgt{a})}=\tgt{a}$, we also have
$\one_{\acat{A}}=\{\tgt{a}\mid a\tin \acat{A},\,a\neq \bot\}$.

\begin{lem}\label{lem:idempotent-guards} 
  The following hold in every abstract category. 
  \begin{ithm}
    \item\label{lempart:idem} The guards are idempotents, namely
    \begin{align*}
      \src{(\src{(-)})} &= \src{(-)}, & \tgt{(\tgt{(-)})} &= \tgt{(-)}.
    \end{align*}
    \item\label{lempart:guard-reduc} Terms $f$ and $g$ satisfy
    \begin{align*}
      \tgt{(fg)} &\venturi \tgt{f}, & \src{(fg)} \venturi \src{g}.
    \end{align*}
  \end{ithm}
\end{lem}

\begin{proof}
  For a term $f$ in an abstract category,
  \[ 
    \tgt{(\tgt{f})}=\tgt{(\src{(\tgt{f})})}=\src{(\tgt{f})}=\tgt{f}.
  \] 
  A similar argument shows $\src{(\src{f})} =\src{f}$, so (a) holds. For
  (b), it remains to consider terms $f,g$ such that $\tgt{(fg)}$ is not $\bot$.
  This means that $fg$ is not $\bot$, and hence $\src{f}=\tgt{g}$ with both
  $\src{f}$ and $\tgt{g}$ not $\bot$. Now  \[\tgt{(fg)} =
  \tgt{(f(\tgt{g}))}=\tgt{(f(\src{f}))}=\tgt{f},\]as claimed. The other formula
  follows similarly.
\end{proof}

Let $\cat{C}$ be a category with object type $\cat{C}_0$.
Form the type of all morphisms of $\cat{C}$:
\begin{align}\label{eqn:all-morphisms}
    \cat{C}_1 &\defeq \bigsqcup_{U:\cat{C}_0}\bigsqcup_{V:\cat{C}_0} \cat{C}_1(U,V).
\end{align} 
For objects $U,V:\cat{C}_0$, there is an
inclusion map (see Section~\ref{sec:types})
\begin{align*} 
    \iota_{UV} & : \cat{C}_1(U,V)\hookrightarrow \acat{C}_1.
\end{align*} 
Thus, for each $\varphi:\acat{C}_1$, there exist unique $U,V:\cat{C}_0$ 
and $f:\cat{C}_1(U,V)$ such that  $\varphi=\iota_{UV}(f)$.

\begin{prop}
  \label{prop:morph_laws}
Let $\cat{C}$ be a category. 
The type $\acat{C}_1^?$ from~\eqref{eqn:all-morphisms} of all
morphisms of $\cat{C}$ with the composition signature 
from~$\eqref{eqn:comp-sig}$ forms an abstract category.
\end{prop}
\begin{proof}
Let $f,g:\acat{C}_1^?$. If $f=\bot$ or $g=\bot$, then all the equations in
Definition~\ref{def:abs-cat} become $\bot=\bot$. It remains to consider
the case that $f,g,h:\acat{C}_1$.  If $f:U \to V$ in $\acat{C}$, then 
  \(
    \tgt{(\src{f})}
      = \id_{\text{Codom} \id_{U}}
      = \id_{U}
      = \src{f}.
  \)
  Similarly, $\src{(\tgt{f})}=\tgt{f}$, and $\tgt{(\src{f})} =
  \src{f}$ and $\src{(\tgt{f})}=\tgt{f}$.
 
  Observe that $(\tgt{f})f$ is defined and equals $\id_{V}f=f$; also
  $f(\src{f})$ is defined and equals $f\id_{U}=f$. For $g:\acat{C}_1(U',V')$,
  the expression $\tgt{(fg)}$ is defined whenever $\src{f}=\tgt{g}$, and
  $f(\tgt{g})$ is defined whenever $\src{f} = \tgt{(\tgt{g})}$. Since
  $\tgt{(-)}$ is idempotent by
  Lemma~\ref{lem:idempotent-guards}\ref{lempart:idem}, both $\tgt{(fg)}$ and
  $f(\tgt{g})$ are defined when $\src{f}=\tgt{g}$. Thus, $\src{f}=\tgt{g}$
  implies  
  \[ 
    \tgt{(fg)} = \id_{V} 
     = \tgt{(f(\src{f}))} = \tgt{(f(\tgt{g}))} ,
  \] 
  so $\tgt{(fg)} = \tgt{(f (\tgt{g}))}$. Similar arguments hold for
  $\src{(fg)} = \src{((\src{f})g)}$ and for  $f(gh) = (fg)h$. 
\end{proof}

\begin{ex}\label{ex:abscat}
  Let $\acat{A}$ be an abstract category with
  $\one_{\acat{A}}\defeq\{e_1,\ldots,e_6\}$ and additional morphisms
  $a_{12},a_{23},a_{13},\acute{a}_{13},b_{45},b_{54}$, where $\src{x_{ij}} =
  e_j$ and $\tgt{x_{ij}}=e_i$.  Using the composition signature from~\eqref{eqn:comp-sig},
  $\acat{A}$ is an algebraic structure with multiplication defined in
  Table~\ref{tab:comp-grammer-table}, where each instance of $\bot$ is omitted. 
It is not easy
  to discern structure from this table, so two additional visualizations of
  $\acat{A}$ are given in Figure~\ref{fig:abscat-prod}, again with $\bot$ omitted. 
  The first is the Cayley graph
  of the multiplication with undefined products omitted. The second is the
  Peirce decomposition, which we now discuss.\exqed

\begin{table}[h]
  \begin{align*}
    \begin{array}{|c||ccccc|c|cccc|cc|}\hline
     x& e_1 & e_2 & e_3 & e_4 & e_5 & e_6 & a_{12} & a_{23} & a_{13} & \acute{a}_{13} & b_{45} & b_{54}\\\hline\hline
       \src{x}   & e_1 & e_2 & e_3 & e_4 & e_5 & e_6 & e_2 & e_3 & e_3 & e_3 & e_5& e_4 \\\hline
       \tgt{x}  & e_1&e_2&e_3&e_4&e_5& e_6 & e_1 & e_2 & e_1 & e_1 & e_4 & e_5 \\\hline
       \multicolumn{13}{c}{}\\\hline
        \cdot & e_1 & e_2 & e_3 & e_4 & e_5 & e_6 & a_{12} & a_{23} & a_{13} & \acute{a}_{13} & b_{45} & b_{54}\\
      \hline
      \hline 
         e_1 & e_1 &     &     &     &     &     & a_{12}&   & a_{13}& \acute{a}_{13} &   & \\
         e_2 &     & e_2 &     &     &     &     &     & a_{23}&   &   &   & \\
         e_3 &     &     & e_3 &     &     &     &     &     &     &   &   &\\
         e_4 &     &     &     & e_4 &     &     &     &     &     &   & b_{45} &    \\
         e_5 &     &     &     &     & e_5 &     &     &     &     &   &   & b_{54} \\
        \hline
         e_6 &     &     &     &     &     & e_6 &     &     &     &   &   & \\
        \hline
         a_{12}&   & a_{12}&   &     &     &     &     & a_{13}&   &   &   & \\
         a_{23}&   &     & a_{23}&   &     &     &     &     &     &   &   & \\
         a_{13}&   &     & a_{13}&   &     &     &     &     &     &   &   & \\
         \acute{a}_{13}&&&\acute{a}_{13}&& &     &     &     &     &   &   & \\
        \hline        
         b_{45}&   &     &     &     & b_{45}&   &     &     &     &   &   & e_4 \\
         b_{54}&   &     &     & b_{54}&   &     &     &     &     &   & e_5& \\
        \hline
    \end{array}
  \end{align*} 
\caption{The multiplication table for $\acat{A}$}\label{tab:comp-grammer-table}
\end{table}
\end{ex}

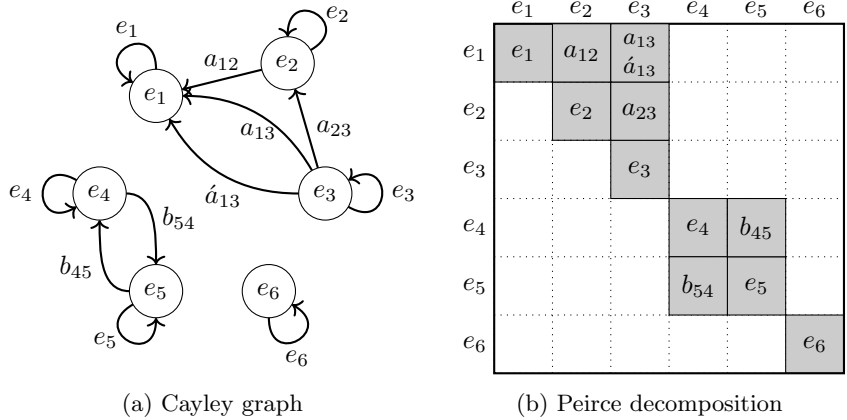
\begin{figure}[!htbp]
  \centering 
  \begin{subfigure}[b]{0.45\textwidth}
    \centering 
    \begin{tikzpicture}
      \node[draw,  circle, outer sep=0pt] (a) at (0:1.5) {$e_3$};
      \node[draw,  circle, outer sep=0pt] (b) at (60:2) {$e_2$};
      \node[draw,  circle, outer sep=0pt] (c) at (120:1.5) {$e_1$};
      \node[draw,  circle, outer sep=0pt] (d) at (180:1.5) {$e_4$};
      \node[draw,  circle, outer sep=0pt] (e) at (240:1.5) {$e_5$};
      \node[draw,  circle, outer sep=0pt] (f) at (300:1.5) {$e_6$};
  
      \draw (a) edge[thick,->,out=-30,in=30, looseness=5, "$e_3$"{right}] (a);
      \draw (b) edge[thick,->,out= 30,in=90, looseness=5, "$e_2$"{right}] (b);    
      \draw (c) edge[thick,->,out= 90,in=150, looseness=5, "$e_1$"{above}] (c);    
      \draw (d) edge[thick,->,out=150,in=210, looseness=5, "$e_4$"{left}] (d);    
      \draw (e) edge[thick,->,out=210,in=270, looseness=5, "$e_5$"{left}] (e);    
      \draw (f) edge[thick,->,out=270,in=330, looseness=5, "$e_6$"{below}] (f); 
  
      \draw (a) edge[thick,->, "$a_{23}$"{right}] (b);
      \draw (b) edge[thick,->, "$a_{12}$"{above}] (c);
      \draw (a) edge[thick,->, bend right, "$a_{13}$"{below}] (c);
      \draw (a) edge[thick,->, bend left, "$\acute{a}_{13}$"{below}] (c);

      \draw (d) edge[thick,->,out=0,in=90, looseness=1, "$b_{54}$"{right}] (e);    
      \draw (e) edge[thick,->,out=180,in=-90, looseness=1, "$b_{45}$"{left}] (d);    

    \end{tikzpicture}
    \caption{Cayley graph}
    \label{fig:dots-arrows}
  \end{subfigure}
  \begin{subfigure}[b]{0.45\textwidth}
    \centering
    \begin{tikzpicture}[scale=1.25]
      \pgfmathsetmacro{\s}{0.62}
      \foreach \t in {0,...,5} {
        \draw[dotted] (0,\s*\t) -- ++(6*\s,0);
        \draw[dotted] (\s*\t,0) -- ++(0,6*\s);
        \draw[fill=black!20] (\s*\t,5*\s-\s*\t) rectangle node{$e_{\the\numexpr \t + 1\relax}$} ++(\s,\s);
      };
      \draw[thick] (0,0)  rectangle (6*\s,6*\s);      
      \draw[fill=black!20] (2*\s,4*\s) rectangle node{$a_{23}$} ++(\s,\s);
      \draw[fill=black!20] (1*\s,5*\s) rectangle node{$a_{12}$} ++(\s,\s);
      \draw[fill=black!20] (2*\s,5*\s) rectangle node{{\small $\begin{array}{c}a_{13} \\ \acute{a}_{13}\end{array}$}} ++(\s,\s);
      \draw[fill=black!20] (4*\s,2*\s) rectangle node{$b_{45}$} ++(\s,\s);
      \draw[fill=black!20] (3*\s,1*\s) rectangle node{$b_{54}$} ++(\s,\s);
      \node at (0.5*\s, 6.25*\s) {$e_1$};
      \node at (1.5*\s, 6.25*\s) {$e_2$};
      \node at (2.5*\s, 6.25*\s) {$e_3$};
      \node at (3.5*\s, 6.25*\s) {$e_4$};
      \node at (4.5*\s, 6.25*\s) {$e_5$};
      \node at (5.5*\s, 6.25*\s) {$e_6$};
      \node at (-0.33*\s, 5.5*\s) {$e_1$};
      \node at (-0.33*\s, 4.5*\s) {$e_2$};
      \node at (-0.33*\s, 3.5*\s) {$e_3$};
      \node at (-0.33*\s, 2.5*\s) {$e_4$};
      \node at (-0.33*\s, 1.5*\s) {$e_5$};
      \node at (-0.33*\s, 0.5*\s) {$e_6$};
    \end{tikzpicture}
    \caption{Peirce decomposition}
    \label{fig:Peirce}
  \end{subfigure}
  \caption{Visualizing the abstract category $\acat{A}$ in Example~\ref{ex:abscat}}
  \label{fig:abscat-prod}
\end{figure}

\subsection{Peirce decomposition of abstract categories}
\label{sec:Peirce}

Treating categories as algebraic structures allows us to frame aspects of
category theory in algebraic terms.
Our goal is an elementary representation theory of categories. In
particular, we seek matrix-like structures---known as Peirce
decompositions in ring theory---for abstract categories.

One can recover from an abstract category $\acat{A}$ notions of objects and
morphisms by considering the identities $\one_{\cat{A}}$. Using the laws in
Definition~\ref{def:abs-cat}, if $e\tin \one_{\cat{A}}$, then $e=\src{e}$ 
and so $ee=e(\src{e})=e$; more generally,
\[
 (\forall e\tin \one_{\cat{A}})\;\;(\forall f\tin \one_{\cat{A}}) \;\;\;
    ef = 
    \begin{cases}
      e  & \text{ if } f = e, \\ 
      \bot & \text{ otherwise}. 
    \end{cases}
\]
In algebraic terms, the subtype $\one_{\cat{A}}$ is a type of pairwise 
orthogonal idempotents. For subtypes $X$ and $Y$ of $\acat{A}$, define
\[ 
 XY\defeq\{xy\mid x\tin X, y\tin Y,\ \src{x}=\tgt{y}\}. 
\]

\begin{fact}\label{fact:capsule}
  If $a\tin\acat{A}$, then $\one_{\cat{A}} \{a\}=\{a\}=\{a\} \one_{\cat{A}}$;
  we write simply $\one_{\cat{A}} a=a=a\one_{\cat{A}}$.
\end{fact}

\noindent Given $e,f\tin\one_{\cat{A}}$, we define three subtypes:
\begin{align*}
  & (\text{left slice}) & e\acat{A} & \defeq \{a\tin \acat{A} \mid e=\tgt{a}\}; \\
  & (\text{right slice}) & \acat{A}f & \defeq \{ a\tin \acat{A} \mid \src{a}=f\};
  \\
  & (\text{hom-set}) & e\acat{A}f & \defeq \{ a\tin \acat{A} \mid  e=\tgt{a},\ \src{a}=f\}.
\end{align*}
These subtypes appear in Figure~\ref{fig:peirce} in the left, middle, and right
images, respectively. If $\src{e}=\tgt{a}$ for $a:\acat{A}$, then
$ea=(\src{e})a=(\tgt{a})a=a$.

\begin{figure}[!htbp]
  \centering
  \begin{tikzpicture}
    \pgfmathsetmacro{\mylen}{3}
    \node (leftcoset) at (-4,0) {\begin{tikzpicture}
      \foreach \t in {0,...,6} {
        \draw[dotted] (0,0.5*\t) -- ++(\mylen,0);
        \draw[dotted] (0.5*\t,0) -- ++(0,\mylen);
      };
      \draw[thick] (0,0)  rectangle (\mylen,\mylen);      
      \draw[fill=black!20] (0,0.5) rectangle (\mylen,1.0);
      \node[scale=0.75] (A) at (1.25,0.75) {$e\acat{A}$};
      \node at (-0.25, 0.75) {$e$};
      \node at (1.25, 3.25) {$\phantom{f}$};
    \end{tikzpicture}};

    \node (rightcoset) at (0,0) {\begin{tikzpicture}
      \foreach \t in {0,...,6} {
        \draw[dotted] (0,0.5*\t) -- ++(\mylen,0);
        \draw[dotted] (0.5*\t,0) -- ++(0,\mylen);
      };
      \draw[thick] (0,0)  rectangle (\mylen,\mylen);      
      \draw[fill=black!20] (1.0,0) rectangle (1.5,\mylen);
      \node[scale=0.75] (A) at (1.25,0.75) {$\acat{A}f$};
      \node at (1.25, 3.25) {$f$};
    \end{tikzpicture}};

    \node (innercoset) at (3.75,0) {\begin{tikzpicture}
      \foreach \t in {0,...,6} {
        \draw[dotted] (0,0.5*\t) -- ++(\mylen,0);
        \draw[dotted] (0.5*\t,0) -- ++(0,\mylen);
      };
      \draw[thick] (0,0)  rectangle (\mylen,\mylen);      
      \draw[fill=black!20] (1.0,0.5) rectangle (1.5,1.0);
      \node[scale=0.75] (A) at (1.25,0.75) {$e\acat{A}f$};
      \node at (-0.25, 0.75) {$e$};
      \node at (1.25, 3.25) {$f$};
    \end{tikzpicture}};
  \end{tikzpicture}
  \caption{Visualizing the Peirce decomposition of $\acat{A}$}
  \label{fig:peirce}
\end{figure}
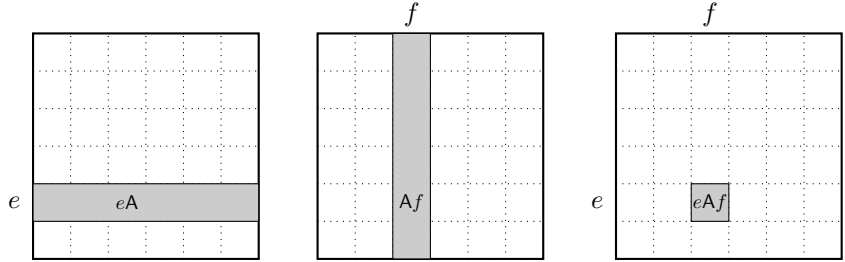

If $a\tin \acat{A}$, then $a\tin \acat{A} (\src{a})$, from which we
deduce the following. 

\begin{prop}
  \label{prop:Peirce-decomposition}
  If $\acat{A}$ is an abstract category, then $a\mapsto (\tgt{a})a$, 
  $a\mapsto a(\src{a})$, and $a\mapsto (\tgt{a})a(\src{a})$ induce  
  invertible functions $($denoted by ``$\leftrightarrow$"$)$ of the following types:
    \begin{align*}
      \acat{A} &\longleftrightarrow \bigsqcup_{e\tin \one_{\cat{A}}} e\acat{A}, 
      &
      \acat{A} &\longleftrightarrow \bigsqcup_{f\tin \one_{\cat{A}}}\acat{A}f 
      &
      \acat{A} &\longleftrightarrow \bigsqcup_{e\tin \one_{\cat{A}}} \bigsqcup_{f\tin \one_{\cat{A}}} e\acat{A}f.
    \end{align*}
\end{prop}

Proposition~\ref{prop:Peirce-decomposition}, which we use to prove Theorem~\ref{thm:extension},
allows us to draw upon intuition from matrix algebras.  
The morphisms of a category appear in its multiplication table, as in 
Table~\ref{tab:comp-grammer-table}. Products of morphisms and slices are defined, as with
matrix products, only when the inner indices agree. In this model, $\one_{\cat{A}}$ can be
visualized as the identity matrix, where the entries on the diagonal are the
individual identities $e\tin \one_{\cat{A}}$. In Figure~\ref{fig:Peirce}, 
that product is
represented in a matrix-like form respecting the conditions of the Peirce
decomposition.

\begin{rem}
  \label{rem:bi-inter} 
  While
types for categories and 
abstract categories differ, every theorem stated in one
  setting translates to a corresponding theorem in the other.  More
  precisely, the translation is a model-theoretic \emph{definable
  interpretation}~\cite{Marker:models}*{\S 1.4}: there is a prescribed
  formula that translates every theorem and its proof between the two theories.
  Example~\ref{ex:abscat} shows how the model of categories with both objects and
  morphisms may be interpreted as definable types in the
  theory of categories with only morphisms (abstract categories). Conversely, if
  $\acat{A}$ is an abstract category, then we obtain a category
  $\cat{C}$ with object type $\cat{C}_0\defeq \one_{\acat{A}}$ as follows. For objects 
  $e,f\tin \cat{C}_0$, we define 
  \[ 
    \cat{C}_1(e,f)  \defeq  f \cat{A} e,
  \]
  where the identity morphisms of $\acat{C}$ are $e:\cat{C}_1(e,e)$.
  To compose morphisms $fae : \cat{C}_1(e,f)$ with $gbf:\cat{C}_1(f,g)$ for objects $e,f,g:\cat{C}_0$, we define 
  \[
    (gbf)(fae) \defeq gbae : \cat{C}_1(e, g).
  \]  
  Hence, we no longer distinguish between categories and
  abstract categories.
\end{rem}

\subsection{Varieties as categories}

Proposition~\ref{prop:morph_laws} shows that a category is an 
algebraic structure with the composition signature. Conversely, 
for every signature $\Omega$, the type  
$\mathrm{Alge}_{\Omega,\mathcal{L}}$ of $\Omega$-algebras in the 
variety for the laws $\mathcal{L}$ forms a category with morphism 
type $\mathrm{Hom}_{\Omega,\mathcal{L}}$ (Section \ref{sec_laws}). 
Indeed, as in Proposition~\ref{prop:morph_laws}, 
$\mathrm{Alge}_{\Omega,\mathcal{L}}$ is an abstract category on  
$(\mathrm{Hom}_{\Omega,\mathcal{L}})^?$, and is therefore 
an algebraic structure with the composition signature. 

 Freyd originally explored the concept of 
\emph{essentially algebraic structures} 
using partial functions, but did not include 
$\bot$ as an operator (see ~\cite{FS}*{\S 1.2}). 
This required dealing with implications such as
``if $\src{f}=\tgt{g}$ then $\src{(fg)}=\src{g}$",  
which in turn entails working in a \emph{quasi-variety}.
But a quasi-variety is not closed under homomorphic images, 
so no analogue of Noether's Isomorphism Theorem (see Theorem \ref{thm:Noether}) exists.
An earlier version of this paper used this approach, but 
implementing our methods revealed the simpler approach of
transforming categories into varieties (see Section \ref{sec:imp}).

We reserve $\cat{E}$ to denote a variety treated as category.

\begin{rem}
  \label{rem:[paradox]} Regarding categories as algebras could lead to a paradox
  of Russell type. The paradox is avoided either by limiting $\Pi$-types to
  forbid some quantifications~\cite{Tucker} or by creating an increasing tower
  of universe types and pushing the larger categories into the next
  universe~\cite{HoTT}*{\S 9.9}.  Both resolutions allow us to define categories
  and algebras computationally.
\end{rem}

Under the correspondence of Remark~\ref{rem:bi-inter}, morphisms between
abstract categories yield functors between categories, but the converse need not follow. A functor $\func{F} : \acat{C}\to\acat{D}$ between categories retains the homomorphism properties of most of the operators: $\func{F}(\src{c}) = \src{\func{F}(c)}$, $\func{F}(\tgt{c}) = \tgt{\func{F}(c)}$, and $\func{F}(\bot) = \bot$. But it relaxes composition to a directional equality:
\begin{align*}
  \func{F}(cc') &\venturi \func{F}(c)\func{F}(c').
\end{align*}
This translation
serves two of our goals.  The first is an elementary representation theory for
categories: by regarding categories as ``monoids with partial operators", we
mimic monoid actions.  The second is to treat a category as a single data type
with operations defined on it.  This is considerably easier to implement as a
computer program. Both {\sf GAP} and {\sc Magma}
are designed for such algebras. 
While there are advantages to the usual
description of categories, the translation to abstract categories is
essential for our approach to computing with and within categories.

We conclude this section with Noether's Isomorphism 
Theorem, see \cite{Cohn}*{Theorem~II.3.7}; it guarantees that images and coimages
exist in a variety.

\begin{thm}[Noether's Isomorphism Theorem]
\label{thm:Noether}
  Let $\varphi: E_1 \to E_2$ be a morphism of $\Omega$-algebras. There exists
  an $\Omega$-algebra $\mathrm{Coim}(\varphi)$ and epimorphism
  $\mathrm{coim}(\varphi) : E_1 \twoheadrightarrow \mathrm{Coim}(\varphi)$, an
  $\Omega$-algebra $\mathrm{Im}(\varphi)$ and monomorphism $\mathrm{im}(\varphi)
  : \mathrm{Im}(\varphi) \hookrightarrow E_2$, and an isomorphism $\psi
  :\mathrm{Coim}(\varphi) \to \mathrm{Im}(\varphi)$ such that the  
  following diagram commutes:
  \begin{center}
    \begin{tikzcd}
      E_1 \arrow[r, "\varphi"] \arrow[d, twoheadrightarrow, "\mathrm{coim}(\varphi)", swap] & E_2 \\
      \mathrm{Coim}(\varphi) \arrow[r, "\psi"] & \mathrm{Im}(\varphi) \arrow[u, hook, "\mathrm{im}(\varphi)", swap]
    \end{tikzcd}
  \end{center}
\end{thm}

The morphism $\mathrm{im}(\varphi)$ from Theorem~\ref{thm:Noether} is the
\emph{image} of $\varphi$, and the morphism $\mathrm{coim}(\varphi)$ is the
\emph{coimage} of $\varphi$. These maps possess universal
properties~\cite{Riehl}*{\S E.5}.

\subsection{Subobjects and images}
\label{sec:subobjects-images}

We close with a list of facts about varieties, which we use heavily in
Section~\ref{sec:induced}. We first define a pre-order that enables abbreviation
of compositions of multiple homomorphisms. To motivate this, assume $\varphi :
E_1\to E_2$ is a homomorphism of algebras. Theorem~\ref{thm:Noether} states
there exists $\theta : E_1 \to \mathrm{Im}(\varphi)$ such that $\varphi =
\mathrm{im}(\varphi) \theta$. We denote this by $\varphi\ll
\mathrm{im}(\varphi)$ and make the following more general definition. For
morphisms $a,b\tin \acat{E}$, 
\begin{align}
  \label{eq:pre-order}
  a &\ll b \iff  \left[(\exists c:\acat{E})\;\; 
  a=bc\right], \\
  \label{eq:dual-ll}
  a &\gg b \iff  \left[(\exists d:\acat{E})\;\; 
  a=db\right].
\end{align}
Two monomorphisms $a,b:\acat{E}$ are \emph{equivalent} if $a\ll b$ and
$b\ll a$. Similarly, epimorphisms $c,d:\acat{E}$ are \emph{equivalent} if $c\gg
d$ and $d\gg c$.

\begin{lem}\label{lem:im}
  Let $\acat{E}$ be a variety. 
For morphisms $a,b\tin \acat{E}$, if 
  $\src{a}=\tgt{b}$, then $a\,\mathrm{im}(b) \ll \mathrm{im}(ab)$.
\end{lem}

\begin{proof}
  By Theorem~\ref{thm:Noether}, there exist isomorphisms $\psi_b, \psi_{ab} : \acat{E}$
  such that 
  \begin{align*}
    b &= \mathrm{im}(b)\psi_b\mathrm{coim}(b), & ab &= \mathrm{im}(ab)\psi_{ab}\mathrm{coim}(ab) .
  \end{align*}
  By the universal property of coimages, there exists a unique morphism
  $\pi:\acat{E}$ such that $\mathrm{coim}(ab) = \pi\,\mathrm{coim}(b)$.
  Therefore
  \begin{align*}
    a\,\mathrm{im}(b)\psi_b\mathrm{coim}(b) = ab = \mathrm{im}(ab) \psi_{ab} \mathrm{coim}(ab) = \mathrm{im}(ab) \psi_{ab} \pi\, \mathrm{coim}(b).
  \end{align*}
  Since $\mathrm{coim}(b)$ is an epimorphism, $a\,\mathrm{im}(b) =
  \mathrm{im}(ab) \psi_{ab} \pi \psi_b^{-1} \ll \mathrm{im}(ab)$.
\end{proof}

\begin{lem}\label{lem:im-monic}
  Let $\acat{E}$ be a variety. For morphisms $a,b\tin \acat{E}$, if
  $\src{a}=\tgt{b}$, then 
  $\mathrm{im}(ab)\ll \mathrm{im}(a)$. If $a$ is also monic, then 
  $\mathrm{im}(ab) \ll a\,\mathrm{im}(b)$.
\end{lem}

\begin{proof}
  The first claim follows from the universal property of images, so we assume $a$ is monic.
  By Theorem~\ref{thm:Noether}, there exists an isomorphism $\psi_b : \acat{E}$ such
  that 
  \begin{align*}
    b &= \mathrm{im}(b)\psi_b\mathrm{coim}(b) .
  \end{align*}
  Since $a\,\mathrm{im}(b)$ is monic and
  $ab=(a\,\mathrm{im}(b))(\psi_b\mathrm{coim}(b))$, by the universal property of
  images, there exists a morphism $\iota : \acat{E}$ such that $\mathrm{im}(ab)
  = a\,\mathrm{im}(b)\iota\ll a\,\mathrm{im}(b)$.
\end{proof}

Varieties have a \emph{coproduct}~\cite{Riehl}*{p.~81} given by
the \emph{free product}~\cite{Riehl}*{p.~183}.  
An example concerning groups is given in \cite{Riehl}*{Corollary 4.5.7}. We list
some facts concerning coproducts in varieties.

\begin{fact}
  \label{fact:coprod}
  Let $I$ be a type. In a variety $\acat{E}$,
  the following hold for all $e:\one_{\acat{E}}$ and $a:I\to e\acat{E}$.
  \begin{ithm}
    \item\label{factpart:coprods} There exists a coproduct morphism
    $\coprod_{i:I}a_i$ and morphisms $\iota:I\to
    \big(\coprod_{i:I}a_i\big)\acat{E}$ satisfying
    $\big(\coprod_{i:I}a_i\big)\iota_j=a_j$ for each $j:I$.

    \item\label{factpart:empty-coprod} If $I$ is uninhabited, then
    $f\defeq\src{\left(\coprod_{i:I}a_i\right)}$ is the identity on the free algebra
    on the empty set. In particular, $\coprod_{i:I}a_i$ is the unique morphism
    inhabiting $e\acat{E}f$.

    \item\label{factpart:factor-out} If $b\tin \acat{E}$ such that
    $\src{b}=\tgt{a_i}$ for all $i:I$, then $\coprod_{i:I}(b a_i)=b
    \coprod_{i:I}a_i$. 

    \item\label{factpart:ignore-inside} If $b:I\to \acat{E}$ with
    $\src{a_i}=\tgt{b_i}$ for all $i:I$, then $\coprod_{i:I} (a_i b_i)\ll
    \coprod_{i: I}a_i$.

    \item\label{factpart:smaller-coprod} If $J\subset I$, then $\coprod_{j:J}a_j
    \ll \coprod_{i:I}a_i$.
  \end{ithm}
\end{fact}

Finally, if $a$ is monomorphism satisfying 
$\tgt{a}=e$, for some identity $e$, then $\src{a}$ can be regarded as 
a subobject of the object associated to $e$. 
Given a collection $\{a_i\mid i:I\}$ of such monomorphisms, consider the smallest subobject 
containing all set-wise images of the $\src{a_i}$.  
The coproduct allows us to effectively ``glue" together 
all of the monomorphisms, but the result is not a monomorphism.  
To obtain a monomorphism, we take the image of the coproduct, 
namely
\begin{align}\label{eq:coprod_im_comment}
  \text{im}\left(
    \coprod_{i:I} a_i
  \right).
\end{align}

\section{Category actions, capsules, and counits}
\label{sec:actions}

Theorem~\ref{thm:char-repn} asserts that characteristic subgroups arise from
categories acting on other categories.  In this section we define category
actions and introduce the notion of a \textit{capsule}. We also elucidate the
connection between capsules and the more familiar category notions of units,
counits, and adjoint functor pairs. Recall from our discussion following
Definition \ref{def:abs-cat} that an (abstract) category $\acat{A}$ is on an
underlying type 
$A^?$, hence $\bot:\acat{A}$.

\subsection{Category actions}
\label{sec:cat-acts-biacts}
Our formulation of category actions generalizes the familiar notion for groups
and also actions of monoids and groupoids ~\cite{MonoidsAC}*{\S I.4}. The
technical aspects of the definition concern the additional guards, denoted
$\lhd$, needed to express where products are defined. Their use is similar to
the guards $\blacktriangleleft$ introduced for abstract categories in
Definition~\ref{def:abs-cat}.

\newcommand{\actX}{\aX} 
\newcommand{\actY}{\aY}
\newcommand{\biactionpair}{biaction pair}
\newcommand{\myopp}{\lhd X}

\begin{defn}
\label{def:cat-act}
  Let $\cA$ be an abstract category with guards $\src{(-)}$ and
  $\tgt{(-)}$. Let $X$ be a type. A \emph{$($left$)$ category action} of
  $\cA$ on $X$ consists of a type $\myopp$, functions 
  $(-)\lhd:\cA\to (\myopp)^?$ and $\lhd(-):X^?\to (\myopp)^?$ that output $\bot$ if, and only if, the input is $\bot$, and a 
  function $\cdot: \cA\times X^?\to X^?$ that satisfies the following
  rules: \\[1ex]
  \hspace*{0.5cm}\begin{tabular}{llll}
    (1) & ($\forall a:\cA$) & ($\forall x:X^?$) &  $\left[(a\lhd=\lhd x)\iff ((\exists y:X)\; a\cdot x=y)\right]$;\\[0.5ex]
    (2) & ($\forall a:\cA)$  & ($\forall x:X^?$) & $\left[(\src{a})\lhd=a\lhd\text{ and } ((\src{a})\cdot x) \venturi x\right]$; and\\[0.5ex]
    (3) & ($\forall a,b:\cA$) & ($\forall x:X^?$) &  $((ab)\cdot x) \venturi (a\cdot (b\cdot x))$.
  \end{tabular}

  Given a left action of $\cA$ on a type $Y$, a function $\fM:X^?\to Y^?$
  is an \emph{$\cA$-morphism} if $\fM(a\cdot x) = a \cdot \fM(x)$ whenever
  $a:\cA$ and $x:X$ with $a\lhd=\lhd x$; that is, $\fM(a\cdot x) \venturi a \cdot \fM(x)$.
\end{defn}

Right category actions are similarly defined. 
We unpack the symbolic expressions in Definition~\ref{def:cat-act}. Condition
(1) states that the functions $(-)\lhd$ and $\lhd (-)$ serve as guards for the
function $\cdot : \cA\times X^?\to X^?$: namely, (1)
characterizes precisely when $\cdot$ is defined. The first part of condition (2)
asserts that $(-)\lhd$ respects the $\src{(-)}$ identity of $\acat{A}$; the
second part states that identity morphisms of $\acat{A}$ act as identities.
Condition (3) is the familiar group action axiom in the setting of
partial functions.

For subtypes $S\subset \acat{A}$ and $Y\subset X$, we write
\begin{align*}
  S\cdot Y & \defeq \{ s\cdot y\mid s\tin S,\; y\tin Y,\; s\lhd =\lhd y\}.
\end{align*}
From Definition~\ref{def:cat-act}, an $\acat{A}$-morphism $\fM:X^?\to Y^?$ maps a term
$b:(\acat{A}\cdot X)$ to a term of $Y$;
we say that $\fM$ is \emph{defined} on $\acat{A}\cdot X$.

\begin{defn}
  The category action of $\acat{A}$ on $X$ is \emph{full} if $e\cdot x \mapsto
  \lhd(e\cdot x)$ defines a bijection from $\one_{\acat{A}}\cdot X$ to
  $\cA\lhd = \{a \lhd \mid a :\acat{A}\}$.
  Thus, the action is full if, and only if, for every 
$a:\cA$ there exists $x:X$ such that $a\lhd =\lhd x$. 
\end{defn}
\noindent

Recall from Remark~\ref{rem:bi-inter}
that we identify categories and abstract 
categories. We say that a category $\cC$ 
acts on a type $X$ if its morphism type 
$\cC_1^?$ acts on $X$ (cf.\ Proposition~\ref{prop:morph_laws}).

\begin{ex}
  Let $\cat{C}$ be a category with object type $\cat{C}_0$ and morphism type   
  $\cat{C}_1^?$. Set $X = \myopp = \cat{C}_0$. Define $(-)\lhd : \cat{C}_1^? \to 
  \cat{C}_0^?$ via $f\lhd \defeq \Dom f$ and define $\lhd (-) : \cat{C}_0^? \to
  \cat{C}_0^?$ via $\lhd \,U \defeq U$. Let $\cdot : \cat{C}_1^? \times \cat{C}_0^?\to
  \cat{C}_0^?$ be defined by 
  \begin{align*}
    f\cdot U \defeq \begin{cases}
      \Codom f & \text{if } f\lhd = \lhd \,U, \\
      \bot & \text{otherwise}.
    \end{cases}
  \end{align*}
  This defines a full left action of $\cat{C}$ on $\acat{C}_0$. A full right
  action is defined similarly.~\exqed
\end{ex}

\begin{rem}
  Let $\acat{C}$ be a category and let $X=\acat{C}_1$. The
  definition of category action 
  in~\cite{FS}*{1.271--1.274} is similar to ours, but it requires
  $\myopp=\one_{\acat{C}}=\{\src{f} \mid f\tin\acat{C}_1\}$ and $\lhd
  x=\tgt{x}$ and $f\lhd =\src{f}$ for every $x\tin X$ and $f\tin
  \acat{C}_1$. Thus, for $f,g\tin \acat{C}_1$ and $x\tin X$, both $f\cdot
  x$ and $g\cdot x$ are defined (neither is $\bot$) only when $\src{f}=\lhd x=\src{g}$; this is too
  restrictive for our purposes.
\end{rem}

\subsection{Capsules}
As identified in Section~\ref{sec:local-to-global}, we focus on the 
action of one category $\cat{A}$ on another category $\cat{X}$; 
we call these ``category modules'' \emph{capsules}. 
Note the subtle change in notation from $X$ to $\cat{X}$ 
to emphasize this setting.
In this case, $\cat{X}$ already has a candidate type
for $\lhd \cat{X}$, namely $\tgt{\cat{X}}=\one_{\cat{X}}$. 
Since 
a category has its own operation of composition, the action by $\cat{A}$
respects composition.  For example, given a group homomorphism
$\varphi:G\to H$, we get an action $g\cdot h\defeq \varphi(g)h$ that satisfies
$g\cdot (hh')=(g\cdot h)h'$. 

\begin{defn}\label{def:vanilla} 
  A category $\cat{X}$ is a \emph{left $\acat{A}$-capsule} if there is a full left
  $\acat{A}$-action on $\acat{X}$ with $\lhd \acat{X} = \one_{\acat{X}}$ such
  that the following hold:\\[1ex]
  \hspace*{0.5cm}\begin{tabular}{llll}
    (a) & $(\forall x:\acat{X})$& $(\lhd x = \tgt{x})$;\\
    (b) & $(\forall a:\acat{A})$ & $(\forall x,y: \acat{X})$ & $(a\cdot (xy)=
    (a\cdot x)y)$.
  \end{tabular}
\end{defn}
A \emph{right $\acat{A}$-capsule} is similarly defined.  
We present our results below 
for left $\acat{A}$-capsules, but they can be formulated for both.

Much of our intuition on actions draws on familiar themes in representation
theory. A reader may be assisted by translating ``$\cA$-capsule'' to
``$A$-module'' and considering the matching statement for modules.  We write
${_{\cA} {\acat{X}}}$ to indicate the presence of a left $\acat{A}$-capsule action
on $\aX$.

From now on, if a category $\acat{A}$ acts on
itself, then we assume it is by the (left) \emph{regular action}, 
where $\cdot : \acat{A}\times \acat{A} \to \acat{A}$ is given 
by composition in $\acat{A}$. Moreover, a
category action on another category is implicitly understood to be on the
morphisms. We now show that capsules arise from morphisms between
categories.

\begin{prop}
  \label{prop:functors-are}
    A category $\cX$ is a left $\cA$-capsule of a category $\cA$ if, and only
    if, there is a morphism $\fF:\cA\to \cX$ such that $a\cdot x=\fF(a)x$
    for all $a:\cA$ and $x:\cX$. Furthermore, the morphism $\fF$ is
    unique.
\end{prop}

The following lemma proves one direction of Proposition~\ref{prop:functors-are}.
  
\begin{lem}
  \label{lem:induced-act}
    Every morphism $\fF:\cA\to\cX$ of categories makes $\cX$ a left 
$\acat{A}$-capsule,
    where for each $a:\cA$ and $x:\cX$, the guard is defined by
    $\text{$a\lhd\defeq\src{\fF(a)}$}$ and the action is defined by $a\cdot
    x\defeq \fF(a)x$. 
\end{lem}

\begin{proof}
Condition (1) of Definition~\ref{def:cat-act} is satisfied by the defined action.

For the first part of condition (2), let $a\tin\acat{A}$. Since $\func{F}$ is
a morphism and $\src{(-)}$ is everywhere defined,
$\func{F}(\src{a})=\src{\func{F}(a)}$. Hence, by
  Lemma~\ref{lem:idempotent-guards}\ref{lempart:idem}, 
  \begin{align*}
    (\src{a})\lhd & = \src{\func{F}(\src{a})} = \src{(\src{\func{F}(a)})} = \src{\func{F}(a)} = a\!\lhd.
  \end{align*}
For the second part of condition (2), let $a\tin\acat{A}$ and $x\tin\acat{X}$
  with $a\lhd=\lhd x$, so $\src{\func{F}(a)} =\tgt{x}$ by definition. Thus,
  $$(\src{a}) \cdot x = \func{F}(\src{a})x = (\src{\func{F}(a)})x = (\tgt{x}) x =
  x,$$ so $(\src{a})\cdot x\venturi x$ for every $a:\acat{A}$ and $x:\acat{X}$. 

  For condition (3), let $a,b\tin\acat{A}$ and $x\tin\acat{X}$ with $\src{a} =
  \tgt{b}$ and $(ab)\lhd=\lhd x$, so $(ab)\cdot x$ is defined and
  $\src{(ab)}=\src{b}$. We need to show that $(ab)\cdot x=a\cdot (b\cdot x)$.
  Since $\func{F}$ is a morphism, 
  $$
    (ab)\lhd=\src{(\func{F}(ab))}=\func{F}(\src{(ab)})=\func{F}(\src{b})=\src{\func{F}(b)}=b\lhd.
  $$ 
Hence, $(ab)\lhd=\lhd x$ implies $b\lhd =\lhd x$. Thus, $\func{F}(b)x$ is
  defined. Also,  $\src{a} = \tgt{b}$ implies $\src{\func{F}(a)} =
  \tgt{(\func{F}(b))}$, so
  \[
    a \lhd = \src{\func{F}(a)} = \tgt{(\func{F}(b))} =\tgt{(\func{F}(b)x)}=\tgt{(b\cdot x)}=\lhd (b\cdot x).
  \]
It follows that $a\cdot (b\cdot x)$ is defined. Since $\func{F}$ is a
morphism, 
  $$
    a\cdot (b\cdot x) = \func{F}(a)(\func{F}(b)x) = \func{F}(ab)x = (ab)\cdot x,
  $$ 
and therefore $(ab)\cdot x \venturi a\cdot (b\cdot x)$ for every
$a,b\tin\acat{A}$ and $x\tin\acat{X}$.

  To see that the action is full, consider  $a\tin \acat{A}$ and define
  $x=\src{\func{F}(a)}$. By the laws of an abstract category,  
  $a \lhd = \src{\func{F}(a)} = \tgt{({\src{\func{F}(a)}})} = \tgt{x} = \lhd x$.
  Finally, $(a\cdot x)y=\fF(a)xy = a\cdot (xy)$, so $\cX$ is a left
  $\cat{A}$-capsule.
\end{proof}

Our proof of the reverse direction of  Proposition~\ref{prop:functors-are} uses
the following result.

\begin{lem}
  \label{lem:unique-identity}
  Let $\acat{X}$ be a left $\acat{A}$-capsule. For every $a\tin \cA$, there is a unique
  $e:\one_{\aX}$ such that $a\cdot e$ is the unique term of type
  $a\cdot\one_{\aX}$.
\end{lem}

\begin{proof}
Let $a\tin\acat{A}$. 
If $a=\bot$, then $\bot$ is the unique $x\tin \aX$ with  
 $a\lhd = \bot = \lhd x$.
Now suppose $a$ is not $\bot$.  
Since the action is full, there exists
  $x\tin\aX$ such that $a\lhd = \lhd x$, so $a\cdot x$ is defined.
  Since $\acat{X}$ is a left $\acat{A}$-capsule, $a\lhd = \lhd x = \tgt{x}$, so
  \[
    \lhd(\tgt{x})=\tgt{(\tgt{x})}=\tgt{x}.
  \]
  Hence, $a\cdot (\tgt{x})$ is defined and has type $a\cdot \one_{\aX}$.
  Suppose $e\tin\one_{\acat{X}}$ and $f\tin\one_{\acat{X}}$ with 
  $a\lhd = \lhd e = \lhd f$, so that
  $a\cdot e \tin a\cdot \one_{\aX}$ and $a\cdot f \tin a\cdot \one_{\aX}$.  
  Then
  \[
    e = \tgt{e} = \lhd e = \lhd f = \tgt{f} = f,
  \] 
  so $a\cdot e = a\cdot f$, and there is exactly one term with 
  type $a\cdot \one_{\aX}$.
\end{proof}

Under the assumptions of Lemma~\ref{lem:unique-identity}, we simplify notation and
identify $a\cdot \one_{\acat{X}}$ with its unique term. 

\begin{proof}[Proof of Proposition~$\ref{prop:functors-are}$]
  By Lemma~\ref{lem:induced-act}, it remains to prove the forward direction and uniqueness. 
  Suppose that $\cX$ is a left $\cA$-capsule.
  By Lemma~\ref{lem:unique-identity}, for each $a:\cA$ there is a unique
  $\fF(\src{a}):\one_{\cX}$ such that $(\src{a})\cdot \fF(\src{a})$ is defined.
  Since $\cX$ is a left $\cA$-capsule and $\fF(\src{a})$ is an identity,
  \[
    (\src{a})\lhd=a\lhd=\lhd\fF(\src{a})=\tgt\fF(\src{a})=\fF(\src{a}).
  \]
  Thus, $a\cdot \fF(\src{a})$ is also defined. Put $\fF(a)\defeq a\cdot
  \fF(\src{a})$. If $x:\cX$, then $a\cdot x$ is defined whenever 
  \[
    \tgt{x}=\lhd x=a\lhd=\fF(\src{a})=\src{\fF(\src{a})}.
  \]
  Hence, $\fF(\src{a})x$ is also defined in $\cX$. Because $\fF(\src{a})$ is an
  identity, $\fF(\src{a})x=x$.  Since $\acat{X}$ is a left $\acat{A}$-capsule, $a\cdot x=a\cdot
  (\fF(\src{a}) x)=(a\cdot \fF(\src{a}))x=\fF(a)x$. Hence, it remains to prove
  that $\fF:\cA\to \cX$ is a morphism of categories.

  For $a,b\tin \acat{A}$, by the action laws
  \begin{align*}
    \func{F}(ab) 
     & = (ab)\cdot \fF(\src{(ab)})
      = (ab)\cdot \fF(\src{b})
      ~\venturi~ a\cdot (b\cdot \fF(\src{b}))
     = a\cdot \func{F}(b).
  \end{align*}
  Thus, $\func{F}(ab) \venturi a\cdot \func{F}(b)$. 
But $\cX$ is a left $\cA$-capsule, so Fact~\ref{fact:capsule} implies that
  \begin{align}\label{eqn:act-comp}
    a\cdot \func{F}(b) & = a\cdot (\one_{\cX}\func{F}(b)) 
    = (a\cdot \one_{\cX}) \func{F}(b) = \func{F}(a)\func{F}(b).
  \end{align}
  Hence, $\func{F}(ab) \venturi \func{F}(a)\func{F}(b)$. By Lemma~\ref{lem:idempotent-guards}\ref{lempart:guard-reduc} and \eqref{eqn:act-comp} for all $a:\cA$, 
  \[
    \src{\fF(a)} ~=~ \src{(a\cdot \fF(\src{a}))} ~=~ \src{(\fF(a) \fF(\src{a}))} ~=~ \src{\fF(\src{a})} ~=~ \fF(\src{a}).
  \]
  Similarly, $\func{F}(\tgt{a}) = \tgt{\func{F}(a)}$. Hence, $\fF$ is a
  morphism.

  Lastly, we prove uniqueness of  $\fF$. Suppose there exists $\func{G}:
  \cA\to \cX$ such that $a\cdot x = \func{G}(a)x$ for every $a:\cA$ and $x:\cX$
  whenever $a\lhd = \lhd x$. Since $a\lhd=\fF(\src{a})$, it follows that
  $\src{\fG(a)}=\fF(\src{a})$, so $\func{G}(a) = \func{G}(a)\fF(\src{a}) =
  a\cdot \fF(\src{a})  = \fF(a)$.
\end{proof}

If $\acat{B}$ is a subcategory of $\acat{A}$ with inclusion $\func{I} :
\acat{B}\to\acat{A}$, then the (left) \emph{regular action} of $\acat{B}$ on
$\acat{A}$ is defined to be the action given by $\func{I}$. In other words, the regular action
of $\acat{B}$ on $\acat{A}$ is given by $b\cdot a = \func{I}(b)a$ for
$a:\acat{A}$ and $b:\acat{B}$. By Lemma~\ref{lem:induced-act}, each regular
action defines a capsule.  With regular actions we sometimes omit the ``$\cdot$''.

\subsection{Category biactions and cyclic bicapsules}\label{sec:cat-biacts}

We now define the concepts appearing in Theorem~\ref{thm:char-repn}(3).

\begin{defn}\label{defn:bi-action}
  Let $\cA$ and $\cB$ be categories and let $X$ and $Y$ be types.
  \begin{ithm}
  \item An $(\cA,\cB)$-\emph{biaction} on $X$ is a left $\cA$-action on $X$
  and a right $\acat{B}$-action on $X$ such that $a\cdot (x\cdot b) =
  (a\cdot x)\cdot b$ for every  $a\tin\cA$, $b\tin\cB$, and $x\tin X$. Hence,
  writing $a\cdot x\cdot b$ is unambiguous. If, in addition, $\aX$ is a left
  $\cA$-capsule and right $\cB$-capsule, then $\aX$ is an
  \emph{$(\cA,\cB)$-bicapsule}.

  \item Suppose there are $(\cA,\cB)$-biactions on $X$ and $Y$. An
  \emph{$(\cA,\cB)$-morphism} is a function $\fM:X^?\to Y^?$ such that
  $\fM(a\cdot x\cdot b) = a\cdot \fM(x) \cdot b$, whenever
  $a:\cA$, $x:X$, $b:\cB$ with $a\lhd=\lhd x$ and $x\lhd=\lhd b$; that is, $\fM(a\cdot x\cdot b) \venturi a\cdot \fM(x) \cdot b$.
  \end{ithm}
\end{defn}

We sometimes write ${_{\cA}X_{\cB}}$ for an $(\cA,\cB)$-biaction on $X$ for
clarity.  Notice that an $(\cA,\cB)$-morphism $\fM:X^?\to Y^?$ must be defined on
$\cA\cdot X\cdot \cB$. As with capsule morphisms, we do not need to establish
guards. We abbreviate $(\cA,\cA)$-bicapsule to \emph{$\cA$-bicapsule},
$(\cA,\cA)$-morphism to \emph{$\cA$-bimorphism}, and $(\cA,\cA)$-biaction to
\emph{$\cA$-biaction}. 
Just as ring homomorphisms are not always linear maps, morphisms of capsules
need not be morphisms of categories since they need not 
send identities to identities. 

Motivated by Proposition~\ref{prop:functors-are}, we show that bicapsules
provide a computationally useful perspective to record natural transformations
of functors. If $\func{F},\func{G}:\acat{A}\to\acat{B}$ are functors and $\mu:
\func{G}\Rightarrow \func{F}$ is a natural transformation, then, using
Remark~\ref{rem:bi-inter}, the natural transformation property written with guards is 
\[
  \func{F}(a)\mu_{\src{a}} = \mu_{\tgt{a}}\func{G}(a)
\] 
for every morphism $a$ in $\acat{A}$.

\begin{prop}
  \label{prop:nat-trans-biact}
  In the following statements, the category $\cA$ is 
  also regarded as an $\cA$-bicapsule via its 
  regular action.
  \begin{ithm}
  \item\label{proppart:get-biactions}
  For every natural transformation 
  $\mu:\fG\Rightarrow \fF$ between 
  functors $\fF,\fG:\cA\to \cX$, the assignment 
  \begin{align*}
  a\cdot x\cdot a'  \defeq \fF(a)x\fG(a') &&
  (a,a':\cA,\;x:\cX)
  \end{align*}
  makes $\acat{X}$ into an $\cA$-bicapsule, and 
  the assignment
  $\fM(a) \defeq a\cdot \mu_{\src{a}}$ defines an 
  $\cA$-bimorphism $\fM:\aA\to \aX$.

  \item\label{proppart:get-nat-trans}
  Conversely, for every category $\aX$ and $\cA$-bimorphism $\fM:\aA\to \aX$,
  the assignments 
  \begin{align*}
    \fF(a) & \defeq a\cdot \one_{\cX}, 
    & 
    \fG(a) & \defeq \one_{\cX}\cdot a
    &
    (a:\cA)
  \end{align*}
   define functors $\func{F},\func{G} : \cA\to \cX$, and
   the assignment
   \begin{align*} 
    \mu_{\src{a}} & \defeq \fM(\src{a}) & (a:\cA)
   \end{align*}
defines a natural transformation $\mu:\fG\Rightarrow \fF$. 
  \end{ithm}
\end{prop}

\begin{proof}
  \begin{iprf}
\item  By Lemma~\ref{lem:idempotent-guards}\ref{lempart:guard-reduc} for all $a,b,c\tin \acat{A}$,
  \begin{align*}
    \fM(ab) 
    & = (ab) \cdot \mu_{\src{(ab)}} = \func{F}(ab)\mu_{\src{(ab)}}
    ~\venturi~ \func{F}(a)\func{F}(b)\mu_{\src{b}}
    ~=~ a\cdot \fM(b) .
  \end{align*}
Since $\mu$ is a
  natural transformation,
  \begin{align*}
    \fM(bc)
    & ~=~ \func{F}(bc)\mu_{\src{(bc)}} 
      ~=~ \mu_{\tgt{(bc)}}\func{G}(bc)\\ 
    &  ~\venturi~ \mu_{\tgt{b}}\func{G}(b)\func{G}(c) 
      ~=~ \func{F}(b)\mu_{\src{b}}\func{G}(c) 
      ~=~ \fM(b)\cdot c.
  \end{align*}
  Thus, $\fM(abc)\venturi a\cdot \fM(b)\cdot c$, so $\fM$ is an
  $\cA$-bimorphism.

  \item By Proposition~\ref{prop:functors-are}, 
  the left and right actions determine functors
  $\func{F},\func{G}:\acat{A}\to \acat{X}$. For $e:\one_{\acat{A}}$, 
  define $\mu_e\defeq\func{M}(e)$. 
For $a\tin\acat{A}$ 
  \begin{align*}
    \func{F}(a) \mu_{\src{a}} &= \func{F}(a) \fM(\src{a}) = a\cdot \fM(\src{a}) = \fM(a(\src{a})) = \fM(a)\\
    &= \fM((\tgt{a})a) = \fM(\tgt{a})\cdot a = \fM(\tgt{a})\func{G}(a) = \mu_{\tgt{a}}\func{G}(a).
  \end{align*}
  It follows that $\mu$ is a natural transformation.\qedhere
  \end{iprf}
\end{proof}

We summarize the conclusion in
Proposition~\ref{prop:nat-trans-biact}\ref{proppart:get-nat-trans}, namely $\mu_e =
\fM(e)$ for every $e:\one_{\acat{A}}$, by writing $\mu\defeq\mathcal{M}(\one_A)$.
While $\one_{\cA}$ consists of many terms,
each of the types $x\cdot \one_{\cA}$
and $\one_{\cA}\cdot x$ is inhabited by a unique term, so 
$\one_{\cA}$ plays a role similar
to multiplying by $1$. Since $\mathcal{M}:\acat{A}\to\acat{X}$ is an $\acat{A}$-bimorphism,  
\begin{align*}
(\forall a:\acat{A}) &&
\mathcal{M}(a)~=~a\cdot \mathcal{M}(\tgt{a})~=~\mathcal{M}(\src{a})\cdot a, 
\end{align*} 
which shows that $\mathcal{M}$ is determined by $\mathcal{M}(\one_{\acat{A}})$. We 
write
\begin{align}\label{eqn:cyclic-bicap}
  \cA\cdot \mu \cdot \cA \defeq \left\{ a \cdot \mu_e \cdot \acute{a} ~\middle|~ a,\acute{a} 
  :\acat{A}, e:\one_{\acat{A}},\ a\lhd = \lhd \mu_{e},\ \mu_e \lhd = \lhd \acute{a} \right\}.
\end{align}
The bicapsule in \eqref{eqn:cyclic-bicap} is the \emph{cyclic}
$\acat{A}$-bicapsule determined by $\mu = \fM(\one_{\acat{A}})$.

\subsection{Units and counits}
A \emph{unit} in a category $\cA$ is a 
natural transformation $\mu:\id_{\cA}\Rightarrow \fH$, where
$\fH:\cA\to \cA$ is a functor. Similarly, 
a \emph{counit} is a natural transformation 
$\nu:\fH\Rightarrow \id_{\cA}$.  We 
will prove that units and counits 
are responsible for all characteristic structure.  
It therefore makes sense to translate these into capsule actions.
We show that a unit $\mu$ is characterized as 
an $(\cA,\cB)$-bimorphism $\fM:\cA\to \cB$
and a counit $\nu$ by an $(\cA,\cB)$-morphism $\fN:\cB\to \cA$.  
As the relationship is dual, 
and we emphasize substructures instead of quotients, we 
consider this relationship only for counits.

\begin{thm}\label{thm:counit-capsules}
  Let $\cA$ and $\cB$ be categories.
  \begin{ithm}

    \item\label{thmpart:bicap-to-counit} If both $\cA$ and $\cB$ are $(\cA,\cB)$-bicapsules and $\fN:\cB\to \cA$ is an
    $(\cA,\cB)$-morphism, then $\fF(b)\defeq \one_{\cA}\cdot b$ and $\fG(a)\defeq a\cdot \one_{\cB}$ 
    define functors $\func{F} :\cB \to \cA$ and $\fG : \cA \to \cB$, and $\nu\defeq\fN\fG(\one_{\cA})$ 
    is a counit $\nu:\fF\fG\Rightarrow \id_{\cA}$.

  \item\label{thmpart:counit-to-bicap} 
  If $\fF:\cB\to \cA$ and $\fG:\cA\to \cB$ are functors and $\nu:\fF\fG\Rightarrow \id_{\cA}$
  is a counit, then $\cA$ and $\cB$ are $(\cA,\cB)$-bicapsules, 
  where
  $a\cdot y\cdot b\defeq \fG(a)yb$ and $a\cdot x\cdot b=\fF\fG(a)x\fF\fG\fF(b)$ for $a,x:\cA$ and $b,y:\cB$.
  Also,
  $\fN'(b)\defeq \fF(b)\nu_{\src{\fF(b)}}$ is an $(\cA,\cB)$-morphism $\cB\to \cA$ such that
  $\fN'\fG(e)=\nu_{\fF\fG(e)}$ for all $e:\one_{\cA}$.
  \end{ithm}
\end{thm}

\begin{proof}
 \begin{iprf} 
\item  By Proposition~\ref{prop:functors-are}, the maps $\fF$ and $\fG$  define functors 
  where $x\cdot b= x\fF(b)$ and $a\cdot y= \fG(a)y$, for $a,x:\cA$ and $b,y:\cB$.
  Put $\nu=\fN(\fG(\one_A))$. For $a:\acat{A}$, 
  \begin{align*}
    a\nu_{\src{a}}
    & = a\fN\fG(\src{a})
      = a\fN(\src{\fG(a)})
      = \fN(a\cdot \src{(\fG(a))}) 
      = \fN(\fG(a)\src{(\fG(a))}) 
      = \fN\fG(a), 
  \end{align*}
  and 
  \begin{align*}
    \nu_{\tgt{a}} \fF\fG(a)
    & = \fN\fG(\tgt{a})\cdot\fG(a)
      = \fN(\tgt{\fG(a)})\cdot \fG(a) 
      = \fN((\tgt{\fG(a)})\fG(a)) 
      = \fN\fG(a).
  \end{align*}
  Hence, $a\nu_{\src{a}} = \nu_{\tgt{a}} \fF\fG(a)$ for all $a:\acat{A}$, so $\nu:\fF\fG\Rightarrow \id_{\cA}$ is a natural transformation.

  We show that $\mathcal{N}'(b) \defeq \fF(b)\nu_{\src{\fF(b)}}$ yields an
  $(\cA,\cB)$-morphism $\mathcal{N}': \cB \to \cA$. 
  First, if $a:\cA$ and $y:\cB$
  with $a\lhd = \lhd y$, then 
  \begin{align*}
    \mathcal{N}'(a\cdot y) &= \mathcal{N}'(\fG(a)y) \\
    &= \fF(\fG(a)y) \nu_{\src{\fF(\fG(a)y)}} \\
    &= \fF\fG(a)\fF(y) \nu_{\src{(\fF\fG(a)\fF(y))}} \\
    &= \fF\fG(a) \fF(y)\nu_{\src{\fF(y)}} \\
    &= \fF\fG(a) \mathcal{N}'(y) \\
    &= a\cdot \mathcal{N}'(y).
  \end{align*}
  Next, if $b:\cB$ such that $y\lhd = \lhd b$, then 
  \begin{align*}
    \mathcal{N}'(yb) &= \fF(yb)\nu_{\src{\fF(yb)}} \\
    &= \nu_{\tgt{\fF(yb)}}\fF\fG\fF(yb) \\ 
    &= \nu_{\tgt{(\fF(y)\fF(b))}}\fF\fG\fF(y)\fF\fG\fF(b) \\
    &= \nu_{\tgt{\fF(y)}}\fF\fG\fF(y)\fF\fG\fF(b) \\
    &= \fF(y)\nu_{\src{\fF(y)}}\fF\fG\fF(b) \\
    &= \mathcal{N}'(y)\fF\fG\fF(b) \\
    &= \mathcal{N}'(y) \cdot b.
  \end{align*}
  Finally, consider $e:\one_{\cA}$. Since 
  functors map identities to identities, we deduce that 
  \begin{align*}
    \mathcal{N}'(\fG(e)) &= \fF\fG(e) \nu_{\src{\fF\fG(e)}} = \nu_{\fF\fG(e)} . \qedhere
  \end{align*} 
 \end{iprf}
\end{proof}

\subsection{Adjoint functor pairs}
\label{sec:biacts-adjoints}
Adjoint functor pairs are an important 
special case of natural transformations. 
We give one of many equivalent definitions
\cite{Riehl}*{\S 4.1}.
\begin{defn}\label{def:adjoint}
  Let $\cA$ and $\cB$ be categories. An \emph{adjoint functor pair} is
  a pair of functors $\fF:\cB\to\cA$ and $\fG : \cA  \to\cB$ with the 
  following property. For every object $U$ in $\cB$ and
  $V$ in $\cA$, there is an invertible function
  \[
    \Psi_{UV} : \cA_1(\fF(U),V) \to \cB_1(U, \fG(V))
  \]
  that is \emph{natural} in the following sense: if $b\tin\cB_1(X,U)$ and
  $a\tin \cA_1(V,Y)$ for objects $X$ in $\cB$ and $Y$ in
  $\cB$ then, for $x : \cA_1(\fF(U),V)$,
  \begin{equation}\label{def:adjoint-classic}
    \Psi_{XY}(a x \fF(b)) 
    = \fG(a)\Psi_{UV}(x)b . 
  \end{equation}
  We say $\fF$ is \emph{left-adjoint} to $\fG$ and $\fG$ is
  \emph{right-adjoint} to $\func{F}$ and write $\fF : \cB
  \adjoint_{\Psi} \cA : \fG$.
\end{defn}

We now characterize adjoint functor pairs in terms of bicapsules.  
A reader may find it useful to review the translation between 
categories and abstract categories in Remark~\ref{rem:bi-inter}. The invertibility of $\Psi_{UV}$ in
Definition~\ref{def:adjoint} is equivalent to a pseudo-inverse property of
morphisms of bicapsules. 

For types $X$ and $Y$, functions $\fM : X^?\to
Y^?$ and $\fN : Y^?\to X^?$ are \emph{pseudo-inverses} if,
for $x:X^?$ and $y:Y^?$,
$\fM\fN\fM(x)=
\fM(x)$ and $\fN\fM\fN(y)= \fN(y)$.

\pagebreak

\begin{thm}\label{thm:iso-biacts-adjoints} 
  Let $\cA$ and $\cB$ be categories. 
  \begin{ithm}
    \item\label{thmpart:biaction-to-adjoint}
    If $\cA$ and $\cB$ are $(\cA,\cB)$-bicapsules and $\fM : \cA\to \cB$ and
    $\fN:\cB\to \cA$ are $(\cA,\cB)$-morphisms that are pseudo-inverses, then
    $\fF:\cB\dashv_{\Psi} \cA:\fG$ where
    \begin{align*}
      \func{F}&: \acat{B}\to\acat{A},\quad \func{F}(b)\defeq \one_{\cA}\cdot b,\\
      \func{G}&: \acat{A}\to\acat{B},\quad \fG(a) \defeq a\cdot \one_{\cB},
    \end{align*}
    and for $x:\acat{A}_1(\func{F}(U),V)$ and $y:\acat{B}_1(U,\func{G}(V))$ the bijections $\Psi_{UV}$ and $\Psi_{UV}^{-1}$ are given by
    \begin{align*}
      \Psi_{UV}(x) &\defeq \mathcal{M}(x) & 
      \Psi_{UV}^{-1}(x) &\defeq \mathcal{N}(y) . 
    \end{align*}
    
    \item 
      If $\fF:\cB \dashv_{\Psi} \cA:\cG$ is an adjoint functor pair, then $\cA$
      and $\cB$ are $(\cA,\cB)$-bicapsules with actions defined by
      \[a\cdot y \defeq \fG(a)y\quad\text{and}\quad x\cdot b  \defeq x\fF(b)\]
      for $a,x: \cA$ and $b,y:\cB$. Furthermore, $\Psi$ yields a pair
      of $(\cA,\cB)$-morphisms $\fM:\cA\to \cB$ and $\fN:\cB\to \cA$ that are
      pseudo-inverses, where 
    \begin{align*}
      \fM(x) & \defeq \Psi_{UV}(x), & x:\cA\cdot \one_{\cB}\defeq\bigsqcup_{\id_V:\one_{\cA}}\bigsqcup_{\id_U:\one_{\cB}}\cA_1(\fF(U),V),\\
      \fN(y) & \defeq \Psi^{-1}_{UV}(y), & y:\one_{\cA}\cdot \cB\defeq\bigsqcup_{\id_V:\one_{\cA}}\bigsqcup_{\id_U:\one_{\cB}}\cB_1(U,\fG(V)).
    \end{align*}
  \end{ithm}
\end{thm}

\begin{proof}
  \begin{iprf}
  \item Since $\acat{A}$ and $\acat{B}$ are
  $(\acat{A},\acat{B})$-bicapsules, by Proposition~\ref{prop:functors-are} there
  are functors $\fF:\cB\to \cA$ and $\fG:\cA\to \cB$ defining the right
  $\cB$-capsule $\cA_{\cB}$ and the left $\cA$-capsule ${_{\cA} \cB}$
  respectively. Since $\fM$ and $\fN$ are pseudo-inverses and capsule
  actions are full, $\fM$ inverts $\fN$ on $\cA\cdot \one_{\cB}$ and $\fN$
  inverts $\fM$ on $\one_{\cA}\cdot \cB$. For objects $U$ of $\cB$ and $V$ of
  $\cA$, let $e=\id_U$ and $f=\id_V$,  
  so that $\cA_1(\fF(U),V)=f\cA\cdot e$ and $\cB_1(U,\fG(V))=f\cdot \cB e$.
  Define $\Psi_{UV}(x) \defeq \fM(x)$ for $x:\cA_1(\fF(U),V)$.
  For $y:\cB_1(U,\fG(V))$, the map $y\mapsto \fN(y)$ inverts
  $\Psi_{UV}$, so the result follows. 
\item By Proposition~\ref{prop:functors-are}, we can exchange
functors for capsules, so $\fF:\cB\to \cA$ affords a right $\cB$-capsule
$\cA_{\cB}$. We enrich this action by adding the left regular action by $\cA$ to
produce an $(\cA,\cB)$-bicapsule $_{\cA}\cA_{\cB}$.  We do likewise with $\fG:\cA\to \cB$
producing a second $(\cA,\cB)$-capsule $_{\cA}\cB_{\cB}$.

To encode $\Psi$, we define an $(\cA,\cB)$-bimorphism $\fM:\cA\to \cB$ by
$\fM(x) \defeq \Psi_{UV}(x)$ for $x:A_1(\fF(U),V)$. This defines $\fM$ on 
\[
  \bigsqcup_{U:\cB_0}\bigsqcup_{V:\cA_0}\cA_1(\fF(U),V)
  =\bigsqcup_{e:\one_{\cB}}\bigsqcup_{f:\one_{\cA}}f\cA\cdot e
  = \cA\cdot \one_{\cB}.
\]
For all other values, $\fM$ is undefined.  Now \eqref{def:adjoint-classic} shows that  on  $\cA\cdot \one_{\cB}$ with $a:\cA_1(V,Y)$, $b:\cB_1(X,U)$, and $x : \cA_1(\fF(U),V)$, 
\begin{align*}
  \fM(ax\cdot b) & = \Psi_{UV}(ax\fF(b)) = \fG(a)\Psi_{XY}(x)b=a\cdot \fM(x)b,
\end{align*}
so $\fM$ is an $(\cA,\cB)$-bimorphism.  We define $\fN:\cB\to \cA$ analogously: if $y:\one_{\cA}\cdot \cB$, then  $\fN(y)\defeq \Psi^{-1}(y)$ (for suitable subscripts of $\Psi$), and otherwise $\fN(y)$ is undefined. Therefore, for $x: \acat{A}\cdot\one_{\acat{B}}$ and $y: \one_{\cA}\cdot \cB$, 
\begin{align*}
  (\fM\fN\fM)(x) & = \Psi(\Psi^{-1}(\Psi(x)))=\Psi(x)=\fM(x)\\
  (\fN\fM\fN)(y) & = \Psi^{-1}(\Psi(\Psi^{-1}(y)))=\Psi^{-1}(y)=\fN(y).  \qedhere
\end{align*}
  \end{iprf}
\end{proof}

\subsection{A computational model for natural transformations}
\label{sec:comp-nat-trans}

We use the algebraic perspective of Section~\ref{sec:cats-algestruct} to 
discuss briefly a model for computing with natural transformations. The next definition
formalizes how to treat morphisms of a category as functors between two other
categories.

\begin{defn}\label{def:natural-action}
  Let $\acat{N}$, $\acat{A}$ and $\acat{B}$ be abstract categories. A
  \emph{natural map} of $\acat{N}$ from $\acat{A}$ to $\acat{B}$ consists of
  functions $\cdot : \one_{\acat{N}} \times \acat{A} \to \acat{B}$ and
  $\bullet : \acat{N} \times \one_{\acat{A}} \to \acat{B}$ that satisfy the
  following properties:\\[1ex]
  \hspace*{0.5cm}\begin{tabular}{llll}
   (1) & $(\forall x,y\tin\acat{A})$& $(\forall e:\one_{\acat{N}})$ &  $e\cdot (xy)\venturi (e\cdot x)(e\cdot y)$;\\[0.5ex]
    (2) & $(\forall x\tin\acat{A})$& $(\forall e:\one_{\acat{N}})$ &  $e\cdot (\src{x}) = \src{(e \cdot x)}$\; and\; $e\cdot (\tgt{x}) = \tgt{(e\cdot x)}$;\\[0.5ex]
    (3) & $(\forall x\tin\acat{A})$& $(\forall s:\acat{N})$  & $(s \bullet (\tgt{x}))((\src{s})\cdot x)     = ((\tgt{s})\cdot x)(s\bullet (\src{x}))$;\\[0.5ex]
   (4) &  $(\forall f:\one_{\acat{A}})$& $(\forall s,t \tin \acat{N})$ & $(st)\bullet f \venturi (s \bullet f)(t\bullet f)$.
  \end{tabular}
\end{defn}

The use of $\venturi$ in (1) and (4) depends only on $xy$ and $st$,
respectively, being defined. For the composition signature $\Omega$
from~\eqref{eqn:comp-sig},
conditions (1) and (2) imply that there is a function
$\one_{\acat{N}} \to \mathrm{Hom}_{\Omega}(\acat{A}, \acat{B})$ given by
$e\mapsto (x\mapsto e\cdot x)$ where $\mathrm{Hom}_{\Omega}(\acat{A}, \acat{B})$
is the type of morphisms between abstract categories; see also
Definition~\ref{def:abs-cat}. This function $\one_{\acat{N}} \to
\mathrm{Hom}_{\Omega}(\acat{A}, \acat{B})$ enables us to treat the objects of
$\cat{N}$ as functors from $\cat{A}$ to $\cat{B}$. As illustrated in Example
\ref{ex:natural-maps}, conditions (3) and (4) are equivalent to the commutative
diagrams in Figure~\ref{fig:functor-cat-act} in the shaded $(2,2)$ and $(3,1)$
entries, respectively.

\begin{figure}[!htbp]
  \centering
  \includegraphics[scale=0.98]{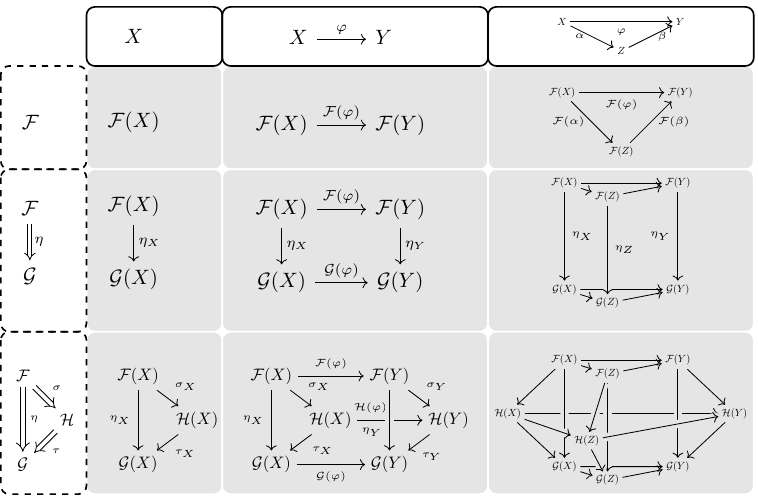} 
  \caption{A natural map of $\cat{N}$ (displayed in the left dotted column) from
  $\cat{A}$ (displayed in top row) to $\cat{B}$ (shaded gray)}
  \label{fig:functor-cat-act}
\end{figure}

\begin{ex}\label{ex:natural-maps} We illustrate how the four conditions of
Definition~\ref{def:natural-action} translate to categories with objects and
morphisms. Let $\acat{A}$ and $\acat{B}$ be two such categories, and let
$\Omega$ be the composition signature from~\eqref{eqn:comp-sig}. Then
$\mathrm{Hom}_{\Omega}(\acat{A},\acat{B})$ is the type of {functors} from $\acat{A}$ to
$\acat{B}$. Let $\acat{N}$ be the category whose objects are the {functors} in
$\mathrm{Hom}_{\Omega}(\acat{A},\acat{B})$ and whose morphisms are natural
transformations. Let $\eta : \func{F} \Rightarrow \func{G}$ be a natural
transformation between $\func{F},\func{G}\tin
\mathrm{Hom}_{\Omega}(\acat{A},\acat{B})$. Treating $\acat{N}$ as an abstract category,
the guards are defined as follows: $\src{\eta} \defeq (\id_{\func{F}}: \func{F}
\Rightarrow \func{F})$ and $\tgt{\eta} \defeq (\id_{\func{G}} :
\func{G}\Rightarrow \func{G})$. Define $\cdot : \one_{\acat{N}}\times \acat{A}
\to \acat{B}$ by  $(\id_{\func{F}}, \varphi) \longmapsto \func{F}(\varphi)$,
  and  $\bullet : \acat{N} \times \one_{\acat{A}} \to \acat{B}$ by
    $(\eta, \id_X) \longmapsto \eta_X$. Now the  conditions of
    Definition~\ref{def:natural-action} become: \\[1.5ex]
  \hspace*{0.5cm}\begin{tabular}{llll}
   (1) & $(\forall a,b \tin \acat{A})$ & $(\forall \id_{\func{F}}\tin \one_{\acat{N}})$ & $\func{F}(ab) \venturi \func{F}(a)\func{F}(b)$;\\[0.7ex]
  (2) & $(\forall a \tin \acat{A})$ & $(\forall \id_{\func{F}}\tin \one_{\acat{N}})$ & $\func{F}(\src{a}) = \src{(\func{F}(a))}$ and  $\func{F}(\tgt{a}) = \tgt{(\func{F}(a))}$;\\[0.7ex]
   (3) & \multicolumn{3}{l}{$(\forall (a:X\to Y) \tin \acat{A})\quad (\forall (\eta : \func{F}\Rightarrow \func{G})\tin\acat{N})\quad \eta_Y\func{F}(a) = \func{G}(a)\eta_X$;}\\[0.7ex]
   (4) & $(\forall \id_X\tin\one_{\acat{A}})$ & $(\forall \eta,\epsilon
    \tin\acat{N})$ & $(\eta\epsilon)_X \venturi \eta_X\epsilon_X$.~\exqed
  \end{tabular}
\end{ex} 

The theory of functors and natural transformations is equivalent to that of natural maps on
abstract categories, but the latter allows us to
use multiple encodings of functors and natural transformations such as those
available in computer algebra systems. If, for example, we compute the 
derived subgroup
$\gamma_2(G)$ of a group $G$ in \textsc{Magma}, then the system may use an
encoding for $\gamma_2(G)$ that differs from 
that supplied for $G$. In such cases, \textsc{Magma} also returns 
an inclusion homomorphism
$\lambda_G : \gamma_2(G)\hookrightarrow G$.

\section{The Extension Theorem}
\label{sec:induced}   

One of our goals is a categorification of characteristic subgroups and their analogues
 in varieties of algebraic structures. We start by translating the characteristic condition into
 the language of natural transformations. 

\subsection{Natural transformations express characteristic subgroups}
\label{sec:nat-trans-express}

Suppose that $H$ is a characteristic subgroup of a group $G$.   
Hence, every automorphism $\varphi:G\to G$ restricts to an 
automorphism $\varphi|_H:H\to H$ of $H$.  In categorical terms, 
we now treat $\acat{A}\defeq \Autcat(G)$ as the subcategory of 
$\cat{Grp}$ consisting of a single
object $G$ and all isomorphisms $G\to G$.  Likewise, we treat $\acat{B}\defeq
\Autcat(H)$ as a subcategory of $\cat{Grp}$.  The restriction defines a
functor $\fC:\acat{A}\to \acat{B}$.  
Of course, $\Autcat(G)$ and $\Autcat(H)$ are also groups
and $\fC$ is a group homomorphism, but the discussion below 
justifies the functor language.  

Now we use the fact that $H$ is a subgroup of $G$ (by using the inclusion map
$\rho_G:H\hookrightarrow G$).
That $\varphi(H)$ is a subgroup of $H$ can be expressed as 
\(
  \varphi\rho_G=\rho_G\varphi|_H=\rho_G\fC(\varphi).
\)  
Recognizing the different categories, we use 
the inclusion functors $\fI:\acat{A}\to \acat{Grp}$ 
and $\fJ:\acat{B}\to \acat{Grp}$ to deduce the following:
\begin{equation*} 
  \fI(\varphi)\rho_G=\rho_G\fJ\fC(\varphi).
\end{equation*}
Thus, a characteristic subgroup determines a natural transformation
\[
  \rho:\fJ\fC\Rightarrow \fI.
\]
The next definition generalizes Definition~\ref{def:firstcounital}.

\begin{defn}\label{def:counital}
  Fix a variety $\cat{E}$ 
with subcategories $\cat{A}$ and $\cat{B}$ and
  inclusion functors $\func{I} : \cat{A} \to \cat{E}$ and $\func{J} : \cat{B}
  \to \cat{E}$. 
A \emph{counital} is
  a natural transformation $\rho : \func{JC}\Rightarrow \func{I}$ for some functor $\func{C}:\cat{A}\to \cat{B}$. 
  The counital $\rho : \func{JC}\Rightarrow \func{I}$ is \emph{monic} if
  $\rho_X$ is a monomorphism for all objects $X$ in $\cat{A}$. 
\end{defn}

A common way to illustrate categories, functors, and natural transformations 
uses a 2-dimensional diagram 
where categories are vertices, functors are directed edges, and 
natural transformations are oriented 2-cells. 
The next diagram illustrates the counital discussed above.

\begin{center}
\pgfmathsetmacro{\rad}{1.25}
\begin{tikzpicture}[scale=0.8]
  \coordinate (t) at (0,0);
  \coordinate (b) at (-1.25,-1.8);
  \coordinate (d) at (1.75,-2.5);

  \fill[color=red!30] (t) -- (b) -- (d) -- cycle;

  \node[scale=0.8, rounded corners, fill=black!10] (Grp) at (t) {$\cat{Grp}$};
  \node[scale=0.8, rounded corners, fill=black!10] (G) at (b) {$\cat{Aut}(G)$};
  \node[scale=0.8, rounded corners, fill=black!10] (K) at (d) {$\cat{Aut}(H)$};

  \draw (G) edge[->, "{$\fI$}"{name=I, above left}, thick] (Grp);
  \draw (K) edge[->, "{$\fJ$}"{above right}, thick] (Grp);
  \draw (G) edge[->, "{$\func{C}$}"{below}, thick] (K);

  \draw[Rightarrow] (1.2,-2.1) -- (-0.5,-0.95) node[midway, above, yshift=2pt] {$\rho$};
\end{tikzpicture}
\end{center}

We now generalize the notion of a characteristic subgroup to an 
arbitrary algebra in a variety.

  \begin{defn}\label{def:A-invariance}
Let $\acat{E}$ be a variety and $G,H:\acat{E}$. Let $\acat{A}$ be a subcategory of $\acat{E}$ whose objects are those of $\acat{E}$.  We denote by  $\acat{A}(G)$ the full subcategory of $\acat{A}$ with a single object $G$. Let $\fI: \acat{A}(G) \to \acat{A}$ and $\fJ :
  \acat{A}(H) \to \acat{A}$ be inclusions. A monomorphism $\iota :
  H\hookrightarrow G$ is \emph{$\acat{A}$-invariant} if there is a functor $\fC:
  \acat{A}(G) \to \acat{A}(H)$ and a monic counital $\eta : \func{JC}
  \Rightarrow \fI$ such that $\eta_G$ is equivalent to $\iota$ (see Section~\ref{sec:subobjects-images} for the definition of equivalence).
\end{defn}

Using the language of Definition~\ref{def:A-invariance}, a characteristic
subgroup $H$ of a group $G$ determines and is determined by a
$\lcore{Grp}$-invariant monomorphism $H\hookrightarrow G$. For fully invariant
subgroups, the corresponding monomorphism is $\acat{Grp}$-invariant.

\subsection{The extension problem and representation theory}
\label{sec:ext-prob}

In Section~\ref{sec:nat-trans-express}, we observed that a characteristic
subgroup $H$ of $G$ determines a functor $\fC : \acat{A} \to\acat{B}$ and a
natural transformation $\rho : \func{JC}\Rightarrow \func{I}$, where $\acat{A}$
and $\acat{B}$ are categories with one object, namely $G$ and $H$ respectively.
If a group $\hat{G}$ is isomorphic to
$G$, then, by Fact~\ref{fact:aut-iso}, $\hat{G}$ has a characteristic subgroup corresponding to $H$. It
seems plausible that we may be able to extend the functor $\func{C}$ to more
groups and, hence, to larger categories. We now make this notion of extension
precise and generalize it to the setting of varieties of algebras.

Fix a variety $\acat{E}$.
Let $\acat{A}$, $\acat{B}$, and $\acat{C}$ be subcategories of 
$\acat{E}$ where $\acat{A} \leq \acat{C}$.
We have inclusion functors $\fI:\acat{A}\to \acat{E}$,
$\fJ:\cB\to \cE$, $\fK:\acat{C}\to \acat{E}$ and $\fL:\acat{A}\to \acat{C}$
such that $\fI=\fK\fL$.  Suppose that  $\rho:\fJ\fC\Rightarrow \fI$ is a monic
counital as depicted in Figure~\ref{fig:char-ext-2-cells}(a). The extension problem
asks whether there is a functor $\fD:\acat{C}\to \acat{E}$ and a natural
transformation $\sigma:\fD\Rightarrow \fK$ such that 
\begin{align}\label{eqn:extension-nat-trans}
  \rho_X = \sigma_{\fL(X)}\tau_X
\end{align}
for some invertible morphism $\tau_X : \func{JC}(X) \to
\func{DL}(X)$ for all objects $X$ of $\acat{A}$. This is depicted in Figure~\ref{fig:char-ext-2-cells}(b).

\begin{figure}[!htbp]
  \pgfmathsetmacro{\rad}{1.5}
  \begin{subfigure}[t]{0.49\textwidth}
    \centering
      \begin{tikzpicture}
        \fill[color=blue!30] (0,0) circle (1.02*\rad cm);
      
        \node[scale=0.8, rounded corners, fill=black!10] (G) at (-\rad, 0) {$\cat{A}$};
      
        \node[scale=0.8, rounded corners, fill=black!10] (B) at (0, -\rad) {$\cat{B}$};

        \node[scale=0.8, rounded corners, fill=black!10] (Grp) at (\rad, 0) {$\cat{E}$};
      
        \node (BB) at (0, -\rad + 0.125*\rad) {};
      
        \draw[very thick, color=black,->] (G) edge[
          out=84, in=97, looseness=1.5, 
          "$\fI$"{below,name=I}] (Grp);
        \draw[very thick,->] (G) edge[
          out=-84, in=173, looseness=1.05, 
          "$\fC$"{left}] (B);
          \draw[very thick,->] (B) edge[
          out=7, in=-97, looseness=1.05, 
          "$\fJ$"{right}] (Grp);
      
        \draw (BB) edge[
            arrows=-Implies,
            double distance=3pt,
            scaling nfold=2, 
            "$\rho$"{right, name=Iota}, pos=0.5, yshift=1em] (I);
      \end{tikzpicture}
    \caption{A diagram of a counital}
    \label{fig:counital-plane}
    \end{subfigure}
    \hfill
    \begin{subfigure}[t]{0.49\textwidth}
      \centering
        \begin{tikzpicture}
          \fill[color=blue!30] (0,0) circle (1.02*\rad cm);
          \fill[color=red!30] (0.5*\rad,0) circle (0.52*\rad cm);
        
          \node[scale=0.8, rounded corners, fill=black!10] (A) at (-\rad, 0) {$\cat{A}$};

          \node[scale=0.8, rounded corners, fill=black!10] (B) at (0, -\rad) {$\cat{B}$};
          \node (BB) at (0, -\rad + 0.125*\rad) {};
        
          \node[scale=0.8, rounded corners, fill=black!10] (C) at (0, 0) {$\cat{C}$};
        
          \node[scale=0.8, rounded corners, fill=black!10] (E) at (\rad, 0) {$\cat{E}$};
        
          \draw[very thick, color=black,->] (A) edge["$\fL$", pos=0.2] (C);
        
          \draw[very thick, color=black,->] (A) edge[
            out=84, in=97, looseness=1.5, 
            "$\fI$"{below,name=I}] (E);
          \draw[very thick,->] (A) edge[
            out=-84, in=173, looseness=1.05, 
            "$\func{C}$"{below}] (B);
          \draw[very thick,->] (B) edge[
            out=7, in=-97, looseness=1.05, 
            "$\fJ$"{right}] (E);
        
          \draw[very thick, color=black,->] (C) edge[
            out=81, in=99, looseness=1.3, 
            "$\fK$"{below,name=K}] (E);
          \draw[very thick,->] (C) edge[
            out=-81, in=-99, looseness=1.3, 
            "$\fD$"{below}, 
            ""{name=D, outer sep=2pt}] (E);
        
          \draw (D) edge[
              arrows=-Implies,
              double distance=3pt,
              scaling nfold=2, 
              "$\sigma$"{right, name=Kappa}, pos=0.4] (K);
          \draw (BB) edge[
            out=135, in=-120,
            arrows=-Implies,
            double distance=3pt,
            scaling nfold=2, 
            "$\tau$"{right, name=Tau}, pos=0.4] (C);
          \draw (C) edge[
            out=135, in=-120,
            arrows=-Implies,
            double distance=3pt,
            scaling nfold=2, 
            pos=0.4] (I);
            \end{tikzpicture}
      \caption{A diagram of an extension}
      \label{fig:ext-rho}
  \end{subfigure}
  \caption{Extending a counital}
  \label{fig:char-ext-2-cells}
\end{figure}
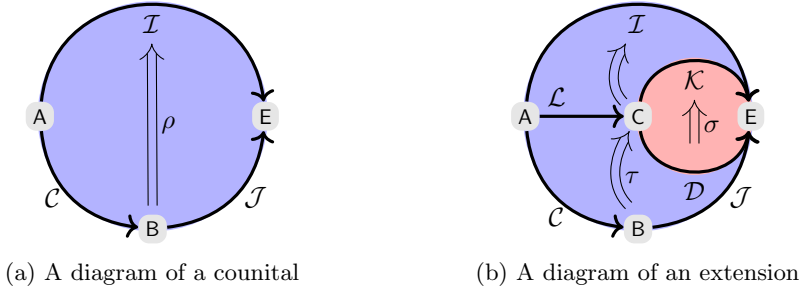

For now, we are concerned only with the existence and construction of such extensions. For use
within an isomorphism test, it will be necessary to develop tools to compute
efficiently with categories; the data types of Section~\ref{sec:comp-nat-trans}
are designed for that purpose.

In light of Proposition~\ref{prop:nat-trans-biact}, we can explore the natural
transformations from Figure~\ref{fig:char-ext-2-cells} through the lens of actions.
Recall that concatenation always denotes regular actions. The natural transformation
$\rho$ defined above is encoded as an $\cA$-bimorphism $\fR:\cA\to \cE$, where
$\rho=\fR(\one_{\cA})$, and this bimorphism defines a cyclic $\cA$-bicapsule
$\Delta\defeq \cA\rho \cdot \cA$ via \eqref{eqn:cyclic-bicap}, which we fix
throughout.

Our goal in part is to extend the cyclic $\cA$-bicapsule $\Delta$ to a cyclic
$\cC$-bicapsule $\Sigma$. Specifically, we will define $\Sigma\defeq\cC \sigma
\cdot \cC$, where $\sigma : \func{D} \Rightarrow \func{K}$ is depicted in
Figure~\ref{fig:char-ext-2-cells}(b).  This is the content of
Theorem~\ref{thm:extension}, but given in  the general setting of 
algebras in a variety.
By construction (Proposition~\ref{prop:nat-trans-biact}), the left actions on
$\Delta$ and $\Sigma$ are regular; hence, we focus on right actions.

\begin{ex}\label{ex_ringstuff}
For the purposes of illustration, we consider 
a familiar construction that is similar to our context, 
namely Frobenius reciprocity and Morita
condensation \cite[Theorem 25A.19]{Rowen}. 
Here $E$ is a ring, and $A$ and $C$ are subrings. Considering 
a Peirce decomposition of $E$, let $\one_A$ and $\one_C$ be
idempotents in $E$ such that
$\one_A\one_C=\one_A=\one_C\one_A$. Then $C = \one_C E\one_C$ is a
(non-unital) subring of $E$, and $A = \one_A E\one_A$ is a subring of both
$C$ and $E$. Furthermore, $\one_A E\one_C=\one_A C$ is an $(A,C)$-bimodule and
$C\one_A=\one_C E\one_A$ is a $(C,A)$-bimodule. Suppose $\Delta$ is a right
$A$-module and $\Sigma$ a right $C$-module. 
The theory of induction and restriction provides us, respectively, 
with a right $C$-module and a right $A$-module: namely, 
\begin{equation*}
  \mathrm{Ind}^C_A(\Delta) \defeq \Delta\otimes_A (\one_A C)\quad\text{and}\quad \mathrm{Res}^C_A(\Sigma) \defeq\Sigma\otimes_C (C\one_A).
\end{equation*}
Thus, $A\to (\one_A C)\otimes_C (C\one_A)$ yields a map
$\Delta\cong\Delta\otimes_A A\to \mathrm{Res}_A^C(\mathrm{Ind}_A^C(\Delta))$.
If, for example, $C\one_A C=C$, then $\Delta \cong
\mathrm{Res}_A^C(\mathrm{Ind}_A^C(\Delta))$.~\exqed
\end{ex}
Guided by the Peirce decomposition from Example~\ref{ex_ringstuff}, we seek similar
constructions for categories and capsules.  Recall that $\cE$ contains a
subcategory $\cC$ that contains a subcategory~$\cA$.  This containment implies
that $\one_{\cA}$ is contained in (or rather embeds under the inclusion functors into)
$\one_{\cC}$. The bicapsule action of $\cA$ on $\Delta$ induces a
$\cC$-bicapsule, denoted $\mathrm{Ind}_{\cA}^{\cC}(\Delta)$.  By mimicking
modules, we can consider a formal extension process. We form the type
$\Delta\otimes_{\cA}\cC$ whose terms are 
pairs, denoted $\delta \otimes c$ for $\delta:\Delta$ and $c:\cC$, subject to the equivalence
relation $(\delta\cdot a)\otimes c=\delta \otimes (ac)$. Then we equip this type
with the right $\cC$-action $(\delta\otimes c)\cdot c'=\delta \otimes (cc')$. Defining $ \cA\backslash \cC \defeq \{\cA c \mid c:\cC\}$, we write
\begin{align*}
  \Delta\otimes_{\cA} \cC &\defeq \coprod_{\cA c:\cA\backslash \cC}\Delta\otimes_{\cA} \cA c . 
\end{align*}
We return to this construction in Section~\ref{sec:compose}.
\noindent
Finally, since $\Delta = \cA\rho\cdot \cA$ and $\cC$ are both subtypes of $\cE$,
the product in $\cE$ defines a map $\Delta\times \cC\to \cE$ that factors
through $\Delta\otimes_{\cA}\cC$.  The image of the map is a cyclic
$\cC$-bicapsule $\Sigma$ embedded in $\cE$, with corresponding
$\acat{C}$-bimorphism $\func{S}:\acat{C}\to\acat{E}$.   
The following theorem
states that this is always possible if $\acat{A}$ is full in $\acat{C}$. 
Table~\ref{tab_setup_exthm}  summarizes some of the notation fixed 
throughout this section.

\begin{table}[!htb]
  \begin{center}
  \hspace*{1.22cm}\begin{tabular}{l|p{2in}}
    $\cE$ & variety of algebraic structures\\
    $\cB,\cC$ &  subcategories of $\cE$\\
    $\cA$ & full subcategory of $\cC$
  \end{tabular}
  \end{center}
  \smallskip
  
  \renewcommand\arraystretch{1.3}
 \begin{center}
\begin{tabular}{c|c|c}
  bimorphism & monic counital & cyclic bicapsule\\\hline
  $\mathcal{R}: \cat{A}\to\cat{E}$ &  $\rho\defeq\func{R}(\one_{\acat{A}})$ & $\Delta\defeq A\rho\cdot A$\\
  $\mathcal{S}: \cat{C}\to\cat{E}$ &  $\sigma\defeq\func{S}(\one_{\acat{C}})$ & $\Sigma\defeq C\sigma\cdot C$
\end{tabular}
\end{center}
\caption{Data for the proof of Theorem~\ref{thm:extension}}
\label{tab_setup_exthm}
\end{table}

\begin{thm}[Extension]\label{thm:extension}
  Let $\acat{E},\acat{C},\acat{A},\Delta$ be as in Table~$\ref{tab_setup_exthm}$. 
  If $\acat{A}$ is full in $\acat{C}$, then 
there is a cyclic $\cC$-bicapsule $\cat{\Sigma}$ on $\cE$ and unique
  cyclic $\cA$-bicapsules $\cat{\Upsilon}$, $\cat{\Lambda}$ on $\cE$ such that 
\[
    \Delta = \mathrm{Res}_{\cA}^{\cC}(\Sigma)\otimes_{\cA}\Upsilon\quad\text{and}\quad
    \mathrm{Res}_{\cA}^{\cC}(\Sigma) = \Delta \otimes_{\acat{A}} \Lambda.
\]
\end{thm}

We briefly describe the idea of the proof. We start with a cyclic
$\cA$-bicapsule $\Delta$ with associated $\acat{A}$-bimorphism $\fR:\cA\to \cE$.
We seek an extension of $\func{R}$ to a $\acat{C}$-bimorphism $\func{S}\colon
\cat{C}\to \cat{E}$ that
satisfies $\func{S}\func{L}=\func{R}$, where
$\func{L}\colon \acat{A}\to\acat{C}$ is the inclusion functor. If this holds,
then, for every $e:\one_{\acat{A}}$ and every $c:\acat{C}$ with
$\src{c}=\tgt{\func{R}(e)}$,
\[c\cdot\func{R}(e)= c\cdot \func{S}\func{L}(e)=\func{S}(c\cdot
e)=\func{S}(c)=\func{S}(\tgt{c})\cdot c.\]
We now derive some necessary conditions for a putative $\func{S}$ of this type. Recall from \eqref{eq:pre-order} that
the notation $\alpha\ll \beta$ for morphisms $\alpha$ and $\beta$ implies that
there is a morphism $\gamma$ such that $\alpha=\beta\gamma$. Applying
Lemma~\ref{lem:im-monic} to $c\cdot \func{R}(e)=\func{S}(\src{c})\cdot c$ yields
$\mathrm{im}(c\cdot \func{R}(e))\ll \mathrm{im}(\func{S}(\tgt{c}))$. 
For $f:\one_{\acat{C}}$ define
\begin{align*}
\mathbb{U}_{\acat{C}}(f)
\defeq \bigsqcup_{e:\one_{\acat{A}}}f\acat{C}\cdot e.
\end{align*} 
The
  left $\acat{A}$-actions on $\acat{C}$ and $\acat{E}$ are regular, as is
  the left $\acat{C}$-action on $\acat{E}$. 
Hence, for $\langle e,c\rangle:\mathbb{U}_{\acat{C}}(f)$, 
  \[ 
    c\lhd = \src{(c\cdot \one_{\acat{E}})} = (\src{c})\cdot \one_{\acat{E}} = e\cdot \one_{\acat{E}}.
  \]
  Therefore $c\lhd = e\cdot \one_{\acat{E}} = \tgt{(e\cdot \fR(e))} =
  \tgt{\fR(e)} = \lhd \fR(e)$, so $c\cdot \fR(e)$ is defined. Since
  $\mathrm{im}(c\cdot \func{R}(e))\ll \mathrm{im}(\func{S}(f))$ holds for every
  $\langle e,c\rangle:\mathbb{U}_{\acat{C}}(f)$, we can make a single inclusion (see \eqref{eq:coprod_im_comment}):
\begin{equation}\label{eq_Sdef}
  \mathrm{im}\left(\coprod_{{\langle e,c\rangle \in \mathbb{U}_{\acat{C}}(f)}} \mathrm{im}(c\cdot \func{R}(e))\right) \ll \mathrm{im}(\func{S}(f)).
\end{equation} 
Observe that \eqref{eq_Sdef} also holds if, instead of $\func{R}=\func{SL}$,
we assume that there exists $\func{T}:\acat{A}\to\acat{E}$ such that
$\func{R}(a)=\mathrm{Res}_{\acat{A}}^{\acat{C}}(\func{S})(a)\func{T}(\src{a})$
for all $a:\acat{A}$; here $\mathrm{Res}_{\acat{A}}^{\acat{C}}(\func{S})(a)=\func{S}(\one_{\acat{C}}\cdot a)$ denotes the restriction of $\func{S}$ to $\acat{A}$. This motivates us to choose $\func{S}$ such that
$\func{S}(c)$ is defined as the left hand side of \eqref{eq_Sdef}, and then solve
for a suitable~$\func{T}$.

In the language of bimorphisms,  Theorem~\ref{thm:extension} asserts that there
exists an $\acat{A}$-bimorphism $\func{T}\colon \acat{A}\to\acat{E}$ such that,
for $a:\acat{A}$,
\begin{align}\label{eqn:RST-bimorphisms}
  \fR(a) &= \fS(\one_{\acat{C}} \cdot (\tgt{a})) \func{T}(a)
  =\fS(\one_{\acat{C}} \cdot a) \func{T}(\src{a})
\end{align}
where $\func{S}$ is the $\acat{C}$-bimorphism corresponding to $\Sigma$; we use
this language in its proof. The second equality in~\eqref{eqn:RST-bimorphisms}
reflects the tensor product over $\cat{A}$ shown in Theorem~\ref{thm:extension}.

\subsection{Building blocks}\label{sec:blocks}

We prove Theorem~\ref{thm:extension} in Section~\ref{sec:extension-proof} using the three
intermediate results presented in this section.

\begin{lem}\label{lem:S-sufficient}
Let $\cC$ and $\cE$ be as in Table~$\ref{tab_setup_exthm}$. 
  For  $\sigma:\prod_{f:\one_{\cC}}(f~\cdot\mono{E})$, the
  following are equivalent. 
  \begin{ithm}
    \item[\rm (1)] There is a $\cC$-bicapsule $\Sigma$ on $\cE$ such that the
    function $\fS : \cC \to \Sigma$ given by 
    \(
      \fS(c) \defeq c\cdot \sigma_{\src{c}}
    \)
    is a $\cC$-bimorphism. 

    \item[\rm (2)] For all $c:\acat{C}$, $c\cdot \sigma_{\src{c}} \ll
    \sigma_{\tgt{c}}$.
  \end{ithm}
\end{lem}

\begin{proof}
  We assume (1) holds and prove (2). By
  Proposition~\ref{prop:nat-trans-biact}\ref{proppart:get-nat-trans}, there exists a unique
  functor $\func{G} : \acat{C}\to \acat{E}$ that induces the action of
  $\acat{C}$ on the right of $\acat{E}$. Since $\fS$ is a function and a 
  $\cC$-bimorphism by assumption, $c\lhd= \lhd \sigma_{\src{c}}$ for all
  $c:\acat{C}$, and 
  \begin{align*}
    c\cdot \sigma_{\src{c}} 
    & = \fS(c) =\fS((\tgt c)\cdot c)= \fS(\tgt{c})\cdot c
    = ((\tgt{c}) \cdot \sigma_{\src{(\tgt{c})}}) \cdot c  
    =\sigma_{\tgt{c}}\fG(c) . 
  \end{align*}
  Thus, (2) holds.

  We now assume (2) holds and prove (1). First, we show that $x:\acat{E}$
  satisfying $c\cdot \sigma_{\src{c}} = \sigma_{\tgt{c}}x$ is unique. Suppose
  $y:\acat{E}$ satisfies $c\cdot \sigma_{\src{c}} = \sigma_{\tgt{c}}y$, so
  $\sigma_{\tgt{c}}x = \sigma_{\tgt{c}}y$. Since $\sigma_{\tgt{c}}$ is a
  monomorphism, $x=y$. We denote this unique morphism by $u_c:\acat{E}$. Since
  $\sigma_{\tgt{c}}u_c$ is defined for all $c:\cC$, 
  \begin{align*}
    \tgt{u_{\tgt{c}}} = \src{(\sigma_{\tgt{(\tgt{c})}})} 
    = \src{(\sigma_{\tgt{c}})} = \tgt{u_c}.
  \end{align*}
  Next, we define a right $\cC$-capsule structure on $\acat{E}$ as
  follows. Let $\lhd (-) : \cC \to \one_{\cE}$ be given by $\lhd c \defeq
  \tgt{u_c}$, and let $(-)\lhd : \cE \to \one_{\cE}$ be given by
  $x\lhd \defeq \src{x}$. For all $c:\cC$ and $x:\cE$, let $x\cdot c \defeq
  xu_c$, which is defined if, and only if, $\src{x} = \tgt{u_c}$. Condition (2)
  of Definition~\ref{def:cat-act} follows from $\lhd (\tgt{c}) = \tgt{u_{\tgt{c}}} =
  \tgt{u_c} = \lhd c$ and $u_e:\one_{\cE}$ for all $e:\one_{\cC}$
  since $\sigma_e$ is monic. Lastly, let $c,d:\cC$ with $\src{c}=\tgt{d}$.
  Now $(cd) \cdot \sigma_{\src{(cd)}} = \sigma_{\tgt{(cd)}}u_{cd} =
  \sigma_{\tgt{c}}u_{cd}$ and, since we have a regular left action,  
  \begin{align*}
    (cd) \cdot \sigma_{\src{(cd)}}  
    = c \cdot (d \cdot \sigma_{\src{(cd)}})
    = c\cdot (d\cdot \sigma_{\src{d}}) 
    = c\cdot \sigma_{\tgt{d}} u_d 
    = \sigma_{\tgt{c}}u_cu_d.
  \end{align*}
  Since $\sigma_{\tgt{c}}$ is a monomorphism, $u_{cd}=u_cu_d$. Hence, this
  defines a right $\cC$-capsule on $\cE$ since $\lhd c = \tgt{u_c}:\one_{\cE}$
  for all $c:\cC$. Since $\acat{C}$ acts regularly on $\acat{E}$ on the left,
  there exists a $\acat{C}$-bicapsule $\Sigma$ on $\acat{E}$ by
  Proposition~\ref{prop:functors-are}, with the regular left and right actions
  just defined. Finally, we prove that $\fS$ is a $\cC$-bimorphism. For all
  $c,x,y:\cC$, 
  \begin{align*}
    \func{S}(cx) &= (cx)\cdot \sigma_{\src{(cx)}} 
    = (cx)\cdot \sigma_{\src{x}} 
    = c\cdot (x\cdot \sigma_{\src{x}}) 
    = c\cdot \fS(x),
  \end{align*}
  provided $\src{c}=\tgt{x}$. If $\src{y} = \tgt{c}$, then 
  \begin{align*}
    \fS(yc) &= (yc)\cdot \sigma_{\src{(yc)}} 
    = (yc)\cdot \sigma_{\src{c}} 
    = y\cdot (c\cdot \sigma_{\src{c}}) 
    = y\cdot (\sigma_{\tgt{c}}u_c)
    = (y\cdot \sigma_{\src{y}})u_c \\
    &= \fS(y)\cdot c.\qedhere
  \end{align*}
\end{proof}

\begin{lem}\label{lem:def-sigma}
For $\acat{E},\acat{C},\func{R}$ as in Table~$\ref{tab_setup_exthm}$,  
define $\sigma :
  \coprod_{f:\one_{\acat{C}}}(f~\cdot \mono{E})$ via
  \begin{equation*} 
    \sigma_{f}  \defeq \mathrm{im}\left(
      \coprod_{\langle e,x\rangle:\mathbb{U}_{\acat{C}}(f)} 
        \mathrm{im}(x\cdot \fR(e))
        \right). 
  \end{equation*}
  For each $c:\acat{C}$, there is a unique $y:\acat{E}$ satisfying $c\cdot
  \sigma_{\src{c}}=\sigma_{\tgt{c}}y$.
\end{lem}

\begin{proof}
Let $c:\acat{C}$, so $c \lhd = \src{c}$, and $\src{c} = \lhd \sigma_{\src{c}}$
  by definition of $\sigma$. We show that $c\cdot \sigma_{\src{c}}
  \ll\sigma_{\tgt{c}}$. By Lemma~\ref{lem:im}, 
  \begin{align*} 
    &(\forall \langle e,x\rangle:\mathbb{U}_{\acat{C}}(f))&  c\cdot \mathrm{im}(x\cdot \fR(e))& \ll\mathrm{im}(c\cdot (x\cdot \fR(e))) = \mathrm{im}((cx)\cdot \fR(e)).
  \end{align*}
  Thus, by Fact~\ref{fact:coprod}\ref{factpart:ignore-inside},
  \begin{align}
    \label{eq:coprod-ineq}
    \coprod_{\langle e,x\rangle:\mathbb{U}_{\acat{C}}(f)} \left(c\cdot \mathrm{im}(x\cdot \fR(e))\right)
      &\ll \coprod_{\langle e,x\rangle:\mathbb{U}_{\acat{C}}(f)} \mathrm{im}((cx)\cdot \fR(e)).
  \end{align}
  Therefore, using \eqref{eq:coprod_im_comment},
\begin{align*}
  c\cdot \mathrm{im}\left(
    \coprod_{\langle e,x\rangle}
        \mathrm{im}(x\cdot \fR(e))
  \right) 
  & \ll  \mathrm{im}\left(
    c\cdot 
    \coprod_{\langle e,x\rangle}
        \mathrm{im}(x\cdot \fR(e))
  \right) & & (\text{Lemma~\ref{lem:im}}) \\
  & =  \mathrm{im}\left(
    \coprod_{\langle e,x\rangle}
      (c\cdot 
        \mathrm{im}(x\cdot \fR(e)))
  \right) & & (\text{Fact~\ref{fact:coprod}\ref{factpart:factor-out}}) \\
  & \ll  \mathrm{im}\left(
    \coprod_{\langle e,x\rangle}
        \mathrm{im}((cx)\cdot \fR(e))
  \right) & & (\text{Equation (\ref{eq:coprod-ineq})}) 
\end{align*} 
where all of the coproducts are over
$\langle e,x\rangle:\mathbb{U}_{\acat{C}}(\src{c})$. By 
Fact~\ref{fact:coprod}\ref{factpart:smaller-coprod}, 
\[ 
  \mathrm{im}\left(
    \coprod_{\langle e,x\rangle:\mathbb{U}_{\acat{C}}(\src{c})}
        \mathrm{im}((cx)\cdot \fR(e))
  \right) \ll \mathrm{im}\left(
        \coprod_{\langle e,z\rangle:\mathbb{U}_{\acat{C}}(\tgt{c})} 
          \mathrm{im}(z\cdot \fR(e))
    \right).
\]
Putting this together, we deduce that 
\begin{align*}
  c\cdot \sigma_{\src{c}} &= c\cdot \mathrm{im}\left(
    \coprod_{\langle e,x\rangle:\mathbb{U}_{\acat{C}}(\src{c})}
        \mathrm{im}(x\cdot \fR(e))
  \right) 
  \ll \mathrm{im}\left(
    \coprod_{\langle e,z\rangle:\mathbb{U}_{\acat{C}}(\tgt{c})} 
      \mathrm{im}(z\cdot \fR(e))
\right) = \sigma_{\tgt{c}}, 
\end{align*}
so $c\cdot \sigma_{\src{c}}=\sigma_{\src{c}}y$ for some $y:\acat{E}$. Since
$\sigma_{\tgt{c}}$ is monic, $y$ is unique. 
\end{proof}

\begin{prop}\label{prop:T}
 Let $\acat{E},\acat{C},\acat{A},\func{R}$ be as in Table~$\ref{tab_setup_exthm}$. 
If $\acat{A}$ is full in $\acat{C}$, 
 then $\func{S}: \acat{C}\to\acat{E}$ defined by 
  \[\func{S}(c)\defeq c\cdot \mathrm{im}\left(\coprod_{\langle e,x\rangle:\mathbb{U}_{\acat{C}}(\src{c})} \mathrm{im}(x\cdot\func{R}(e))  \right)\]
   is a  $\acat{C}$-bimorphism, and there exists a unique
  $\lambda:\prod_{f:\one_{\acat{A}}}f\!\!\core{E}\!\!f$ such that for all $a:\acat{A}$,  
  \begin{equation*}
    \fR(a) = \fS(a \cdot\one_{\acat{C}}) \lambda_{\src{a}}\quad\text{and}\quad \fS(a \cdot\one_{\acat{C}}) = \fR(a) \lambda_{\src{a}}^{-1}. 
  \end{equation*}
\end{prop}

\begin{proof}
  Take $e,f:\one_{\acat{A}}$, and recall that $\acat{A}$ acts regularly on both
  the left and right of $\acat{C}$. Since $\acat{A}$ is full in $\acat{C}$, 
these actions are full, so 
  \[
    f\cdot\acat{C}\cdot e = \one_{\acat{C}} \cdot (f\acat{A}e) = (f\acat{A}e) \cdot \one_{\acat{C}} . 
  \]
  Since the left actions of $\acat{A}$ on $\acat{C}$ and on $\acat{E}$ 
and the left action of $\acat{C}$ on $\acat{E}$ are regular, for each
  $a:\acat{A}$ and $x:\acat{E}$ with $a\lhd = \lhd x$,
  \begin{align}\label{eqn:C-to-A}
    (a\cdot \one_{\acat{C}}) \cdot x &= a\cdot x. 
  \end{align}
  Fix $a:\acat{A}$. Set $c = a\cdot\one_{\acat{C}}$, so $(\src{c})\acat{C} = (\src{a})\cdot
  \acat{C}$. Thus, since the $\acat{A}$-action on $\acat{C}$ is full,
  \begin{align}\label{eqn:subscripts}
    \mathbb{U}_{\acat{C}}(\src{c}) \defeq \bigsqcup_{e:\one_{\acat{A}}} \left((\src{c})\acat{C}\cdot e\right) &= \bigsqcup_{e:\one_{\acat{A}}} \left((\src{a})\cdot \acat{C}\cdot e\right) = ((\src{a})\acat{A})\cdot \one_{\acat{C}},
  \end{align}
  where $(\src{a})\acat{A}$ acts on $\one_{\acat{C}}$ in the final expression.
  Therefore
  \begin{align*}
    \fS(a\cdot\one_{\acat{C}}) 
    & = a\cdot \mathrm{im}\left(\coprod_{\langle e,x\rangle:\mathbb{U}_{\acat{C}}((\src{a})\cdot \one_{\acat{C}})} \mathrm{im}(x\cdot \fR(e))\right) 
        & & (\text{Equation \eqref{eqn:C-to-A}}) \\ 
    & = a\cdot \mathrm{im}\left(\coprod_{a':(\src{a})\acat{A}} \mathrm{im}(a'\cdot \fR(\src{a'}))\right) 
        & & (\text{Equation \eqref{eqn:subscripts}})\\
    & = a\cdot \mathrm{im}\left(\coprod_{a':(\src{a})\acat{A}} \mathrm{im}(\fR(\src{a})\cdot a')\right) 
        & & (\text{$\acat{A}$-bimorphism, $\tgt{a'} = \src{a}$})\\
    & \ll a\cdot \fR(\src{a}) = \fR(a).
    & & (\text{Lemma~\ref{lem:im-monic},  Fact~\ref{fact:coprod}\ref{factpart:factor-out}})
  \end{align*}
\noindent 
For the application of Lemma~\ref{lem:im-monic} in the last step, recall that 
$\func{R}(e)$ is monic for $e:\acat{A}$ by our assumption 
(Table~\ref{tab_setup_exthm}). 

We establish the other direction as follows:
  \begin{align*}
    \fR(\src{a}) &\ll \mathrm{im}(\fR(\src{a})) = \mathrm{im}((\src{a})\cdot \fR(\src{a})) & & (\text{Theorem~\ref{thm:Noether}}) \\
    &\ll \coprod_{a':(\src{a})\acat{A}} \mathrm{im}(a' \cdot \fR(\src{a'})) & & (\text{Fact~\ref{fact:coprod}\ref{factpart:smaller-coprod}}) \\
    &\ll \mathrm{im}\left(\coprod_{a':(\src{a})\acat{A}} \mathrm{im}(a' \cdot \fR(\src{a'}))\right). & & (\text{Theorem~\ref{thm:Noether}}) 
  \end{align*}
  Acting with $a:\acat{A}$ from the left, we obtain $\fR(a) \ll \fS(a\cdot \one_{\acat{C}})$.

  From both computations, there exist $\lambda,\mu : \prod_{f:\one_{\acat{A}}}
  f\acat{E}f$ such that 
  \begin{align*}
    \fR(a) &= \fS(a \cdot\one_{\acat{C}}) \lambda_{\src{a}}, 
    & \fS(a \cdot\one_{\acat{C}}) &= \fR(a) \mu_{\src{a}}. 
  \end{align*}
  It remains to show that $\mu_{\src{a}} = \lambda_{\src{a}}^{-1}$ for all
  $a:\acat{A}$ and $\lambda$ is unique. For all $e:\one_{\acat{A}}$, 
  $\fR(e)$ is monic by the assumptions in Theorem~\ref{thm:extension}, and 
  $\fS(e\cdot \one_{\acat{C}})$ is also monic by the definition of $\fS$.
  Since $\fR(e)$ is monic, 
  \begin{align*}
    \fR(e) &= \fS(e \cdot\one_{\acat{C}}) \lambda_{e} = \fR(e)\mu_{e}\lambda_{e}
  \end{align*}
  implies $\mu_{e}\lambda_{e} : \one_{\acat{E}}$. Similarly, $\lambda_{e} \mu_{e}
  : \one_{\acat{E}}$ because $\fS(e\cdot\one_{\acat{C}})$ is monic and
  \begin{align*}
    \fS(e \cdot\one_{\acat{C}}) &= \fR(e) \mu_{e} = \fS(e \cdot\one_{\acat{C}}) \lambda_{e} \mu_{e} .
  \end{align*}
  The uniqueness of $\lambda$ follows since $\fR(e)$ is monic.
\end{proof}

\subsection{Proof of Theorem~\ref{thm:extension}}
\label{sec:extension-proof}
Let $\fS : \acat{C} \to \acat{E}$ be the $\acat{C}$-bimorphism in Proposition~\ref{prop:T}. 
This proposition shows that there exists a unique
$\lambda:\prod_{f:\one_{\acat{A}}}f\!\!\core{E}\!\!f$ such that for all
$a:\acat{A}$,
\begin{align}\label{eqn:R-lambda}
  \fR(a) &= \fS(a\cdot \one_{\acat{C}}) \lambda_{\src{a}}. 
\end{align}
Since $\fR$ is an $\acat{A}$-bimorphism, 
\begin{align}\label{eqn:S-CC}
  (a\cdot \one_{\acat{E}}) \cdot \fR(\src{a}) &= \fR(a) = \fR(\tgt{a}) \cdot (\one_{\acat{E}}\cdot a).
\end{align}
Let $\Sigma$ be the $\acat{C}$-bicapsule on $\cE$ defined by the
$\cC$-bimorphism $\func{S}$. Both  $\Delta$ and $\Sigma$ are bicapsules, so
applying \eqref{eqn:R-lambda} to \eqref{eqn:S-CC} yields
\begin{align}\label{eqn:baseline}
  (a\cdot \one_{\acat{E}})\fS((\src{a})\cdot \one_{\acat{C}}) \lambda_{\src{a}} &=\fS((\tgt{a})\cdot \one_{\acat{C}}) \lambda_{\tgt{a}}(\one_{\acat{E}}\cdot a).
\end{align}
Since the left $\acat{C}$-action on $\acat{E}$ and the left $\acat{A}$-action on
$\acat{C}$ are regular, $a\cdot \one_{\acat{E}}=(a \cdot \one_{\acat{C}}) \cdot
\one_{\acat{E}}$. Thus, $(a\cdot \one_{\acat{E}})\fS((\src{a})\cdot
\one_{\acat{C}}) = \fS(a\cdot \one_{\acat{C}}) = \fS((\tgt{a})\cdot
\one_{\acat{C}})(\one_{\acat{E}}\cdot (\one_{\acat{C}} \cdot a))$. Applying this
to \eqref{eqn:baseline} and using the monic property of $\fS((\tgt{a})\cdot
\one_{\acat{C}})$, we deduce that
\begin{align*}
  (\one_{\acat{E}}\cdot (\one_{\acat{C}} \cdot a))\lambda_{\src{a}} &= \lambda_{\tgt{a}}(\one_{\acat{E}}\cdot a) . 
\end{align*}
Since the actions are capsules, $\lambda$ defines a natural transformation. 
By Proposition~\ref{prop:nat-trans-biact}\ref{proppart:get-biactions}, the
function $a\mapsto (\one_{\acat{E}}\cdot (\one_{\acat{C}} \cdot
a))\lambda_{\src{a}}$ defines an $\acat{A}$-bimorphism $\func{T} :
\acat{A}\to \acat{E}$. Thus, $\func{T}(\src{a}) = \lambda_{\src{a}} :\,\core{E}$,
and therefore
\begin{align*}
  \fR(a) &= \fS(\one_{\acat{C}} \cdot a) \func{T}(\src{a}) = \fS(\one_{\acat{C}} \cdot (\tgt{a})) (\one_{\acat{E}}\cdot (\one_{\acat{C}} \cdot a)) \func{T}(\src{a}) = \fS(\one_{\acat{C}} \cdot (\tgt{a})) \func{T}(a).
\end{align*}
The uniqueness of $\func{T}$ follows from
Proposition~\ref{prop:nat-trans-biact}\ref{proppart:get-nat-trans} and the uniqueness of
$\lambda$. \qed

\subsection{Proof of Theorem~\ref{thm:char-counital} for varieties}
\label{sec:proofThm1} 

Recall that
$\text{Counital}(\cB,\cE)$ denotes the type of all counitals
$\iota:\func{K}\fC\Rightarrow \fI$, where $\cB\leq \cE$ and $\cat{C}$ are 
categories, $\fC: \cB\to \cC$ is a functor, and $\fI:\cB\to \cE$ and
$\func{K}:\cC\to \cE$ are inclusion functors.   Let $\text{Unital}(\cB,\cE)$ be
the type  of all unitals $\pi:\fI\Rightarrow \fK\fC$; these are the duals of
counitals. Recall from Theorem~\ref{thm:Noether}
that $\mathrm{im}$ and $\mathrm{coim}$ produce categorical morphisms, and from
Section~\ref{sec:subobjects-images} the equivalence relations on monomorphisms
and epimorphisms. Invariance of monomorphisms is defined in  Definition \ref{def:A-invariance}.

Our use of set theory notation in the 
following generalization of Theorem~\ref{thm:char-counital} 
is justified because we compare subsets of a fixed algebra. 

\begingroup   
\renewcommand{\themainthm}{1-cat} 
\begin{mainthm}\label{thm:char-counital-eastern}
  Let $\cE$ be a variety. For every $G:\cE$, the following equalities of sets hold up to equivalence:
    \begin{align}
      \tag{1}
      \{\iota:H\hookrightarrow G \mid \iota\text{ is $\core{E}$-invariant} \} 
      & = 
      \left\{
        \text{\rm im}(\eta_G) \mid \eta: \text{\rm Counital}\left(\core{\cE},\cE\right)
      \right\};\\[5pt]
      \tag{2}
      \{\pi:G\twoheadrightarrow Q \mid \pi\text{ is $\core{E}$-invariant} \} 
      & =  
      \left\{
        \text{\rm coim}(\tau_G) \mid \tau: \text{\rm Unital}\left(\core{E},\cE\right)
      \right\};\\[5pt]
      \tag{3}
      \{\iota:H\hookrightarrow G \mid \iota\text{ is $\acat{E}$-invariant} \} 
      & =
      \left\{
        \text{\rm im}(\eta_G) \mid \eta: \text{\rm Counital}(\cat{E},\cE)
      \right\};\\[5pt]
      \tag{4}
      \{\pi:G\twoheadrightarrow Q \mid \pi\text{ is $\acat{E}$-invariant} \} 
      & = 
      \left\{
        \text{\rm coim}(\tau_G) \mid \tau: \text{\rm Unital}(\cat{E},\cE)
      \right\}.
    \end{align}
\end{mainthm}
\endgroup

\enlargethispage{0.5cm}    
\begin{proof}
  We prove (1) in detail; the proof of (3) is analogous but requires replacing $\core{E}$ with $\cat{E}$. The proofs of (2) and (4) are dual to the proofs of (1) and (3), respectively. Recall that
the single-object category $\Autcat(G)$ consists of $G$ and all its automorphisms. It is a subcategory of $\cat{E}$ and a full subcategory of $\core{E}$. Let $\fI:\Autcat(G)\to \cE$, $\fL:\Autcat(G)\to \;\core{E}$, and $\fK:\;\core{E}\;\to \cE$ be the inclusion functors with   $\fK\fL=\fI$.

Let $\iota : H\hookrightarrow G$ be an $\core{E}$-invariant morphism in $\cat{E}$. Consider the single-object category $\Autcat(H)$ with inclusion functor $\fJ:\Autcat(H)\to \cE$. As in
Section~\ref{sec:nat-trans-express}, we obtain a natural transformation
$\rho:\fJ\fC\Rightarrow \fI$ with (restriction) functor
$\fC:\Autcat(G)\to\Autcat(H)$, so $\rho: \text{Counital}(\Autcat(G),\cE)$ is a
monic counital.

We now use Proposition~\ref{prop:nat-trans-biact} to pass to the associated
cyclic $\Autcat(G)$-bicapsule  $\Delta\defeq\Autcat(G)\cdot \rho\cdot \Autcat(G)$.
Recall that the left action is defined by $\func{I}$, hence it is regular, and the
right action is defined by $\func{JC}$.  By construction, $\Delta$ satisfies the
conditions of Theorem~\ref{thm:extension} since $\Autcat(G)$ is full in
$\core{E}$. We extend $\Delta$ to a cyclic $\core{E}$-bicapsule
$\Sigma=\core{E}\cdot \sigma \cdot \core{E}$ where $\sigma:\fK\fD\Rightarrow
\fK$ is a monic counital extending $\rho$. Thus, there exists an isomorphism
$\tau_G : \func{JC}(G) \to \func{DL}(G)$ such that $\iota =
\rho_G=\sigma_{\func{L}(G)}\tau_G$; see \eqref{eqn:extension-nat-trans}. Since
$\fL$ is the inclusion functor, there exists an isomorphism $\tau':\acat{E}$
such that $\iota = \sigma_G\tau'$, so $\iota$ and $\sigma_G$ are equivalent.
Hence, $\iota$ and $\mathrm{im}(\sigma_G)$ are equivalent. Since
$\sigma:\text{Counital}(\core{E},\cE)$, this proves the ``$\subseteq$'' part of
(1). 

For the converse, consider $\eta: \text{Counital}(\core{E},\cE)$, say $\eta:
\func{HD}\Rightarrow\func{K}$ for some functor $\fD:\; \core{E}\;\to \cat{C}$,
subcategory $\cat{C}\leq\acat{E}$, and inclusion $\func{H}:\cat{C}\to \cE$.  If
$\varphi:\Autcat(G)$, then $\fL(\varphi):\; \core{E}$, and so
$\func{K}\fL(\varphi) \eta_G =\eta_G \func{H}\fD\func{L}(\varphi)$. Since
$G=\func{L}(G)=\func{K}(G)$, the morphism  $\eta_G:
\func{HD}(G)\to G$ is characteristic, and therefore so is its monic image
$\mathrm{im}(\eta_G)$. This proves the ``$\supseteq$'' part of (1).
\end{proof}

\section{Categorification of characteristic substructure}
\label{sec:inv-cat}

The final step in our work is to 
describe the source of all characteristic subgroups, and more generally 
of characteristic substructures in algebras in fixed varieties.  
In Section~\ref{sec:induced}, we showed that characteristic structure 
arises naturally from counitals.  Now we demonstrate that all counitals 
are derived from counits. In particular, 
in Section~\ref{sec:proofs}, we prove the following generalization of
Theorem~\ref{thm:char-repn} to varieties of algebras.  
\begingroup
\renewcommand{\themainthm}{2-cat} 
\begin{mainthm}\label{thm:char-repn-eastern}
  Fix a variety $\cE$. Let $G$ be an object in $\cE$ with subobject 
$H$ and inclusion $\iota:H\hookrightarrow G$.
   There exist categories $\cat{A}$ and $\cat{B}$, where
   $\core{E}\;\leq\acat{A}\leq\acat{E}$, such that the following are
   equivalent. 
  \begin{ithm}
    \item[\rm (1)] $H$ is characteristic in $G$.
    \item[\rm (2)] There is a functor $\func{C} : \cat{A} \to \cat{A}$ and
     a counit $\eta:\func{C}\Rightarrow \id_{\cat{A}}$ such that 
     $H = \im(\eta_G)$.
    \item[\rm (3)] There is an $(\acat{A},\acat{B})$-morphism $\mathcal{M}:\acat{B}\to
    \acat{A}$ such that $\iota=\mathcal{M}(\id_G \cdot \one_{\acat{B}})$.
  \end{ithm} 
\end{mainthm}
  \endgroup
  
Our proof relies on  the Extension Theorem~\ref{thm:extension} and additional
consideration of counitals.

\begin{defn}\label{def:counital-type}
  Fix a category $\cE$ with subcategories $\cat{A}$ and $\cat{B}$ and
  inclusion functors $\func{I} : \cA \to \cE$ and $\func{J} : \cB
  \to \cE$. 
  A counital  $\eta : \func{JC}\Rightarrow \func{I}$ is \emph{isosceles} 
if $\cat{A}=\cat{B}$ and $\fI=\fJ$, and 
    \emph{flat} if, in addition, $\cA=\cB=\cE$ and $\fI=\fJ=\id_{\cE}$. 
Otherwise, it is \emph{scalene}.
\end{defn}

\begin{ex}\label{ex:three-chars}
  We mention three examples in $\cat{Grp}$ and 
illustrate their natural transformations in 
Figure~\ref{fig:counital}. 
The first two are the derived subgroup and the center of a group $G$,
as considered in Example~\ref{ex:three-chars_first}. For the third example, 
we consider an arbitrary characteristic subgroup $H$ of $G$. 
  As discussed in Section~\ref{sec:ext-prob}, define
  $\cat{Aut}(G)$ to be the category with one object $G$ and 
its morphisms are the
  automorphisms of $G$. Hence, $\cat{Aut}(G)$ and $\cat{Aut}(H)$ are
  subcategories of $\cat{Grp}$ with inclusion functors $\func{J}$ and
  $\func{K}$, respectively. We define a functor $\func{C} : \cat{Aut}(G) \to
  \cat{Aut}(H)$ by mapping $G$ to $H$ and automorphisms of $G$ to their
  restriction to $H$, and so obtain a natural transformation $\iota :
  \func{K}\func{C} \Rightarrow \func{J}$.~\exqed
\end{ex}

\begin{figure}[!htbp]
  \begin{subfigure}[t]{0.32\textwidth}
    \centering
    \pgfmathsetmacro{\rad}{1.25}
    \begin{tikzpicture}
      \fill[color=ForestGreen!30] (0,0) circle (1.02*\rad cm);
    
      \node[scale=0.8, rounded corners, fill=black!10] (G) at (-\rad, 0) {$\cat{Grp}$};
    
      \node[scale=0.8, rounded corners, fill=black!10] (Grp) at (\rad, 0) {$\cat{Grp}$};
    
      \draw[very thick, color=black,->] (G) edge[
      out=84, in=97, looseness=1.5, 
      "$\id$"{below,name=I}] (Grp);
      \draw[very thick,->] (G) edge[
      out=-84, in=-97, looseness=1.5, 
      "$\func{D}$"{above, name=C}, 
      ] (Grp);
    
      \draw (C) edge[
      arrows=-Implies,
      double distance=3pt, 
      scaling nfold=2, 
      "$\lambda$"{right, name=Iota}, 
      pos=0.4] (I);
    \end{tikzpicture} 
    \caption{The derived subgroup}
    \label{fig:comm-counital}
    \end{subfigure}
    \hfill
  \begin{subfigure}[t]{0.32\textwidth}
	\centering
	\pgfmathsetmacro{\rad}{1.25}
  \begin{tikzpicture}
    \coordinate (t) at (0,0);
    \coordinate (b) at (-1.5,-2.5);
    \coordinate (d) at (1.5,-2.5);

    \fill[color=blue!30] (t) -- (b) -- (d) -- cycle;

    \node[scale=0.8, rounded corners, fill=black!10] (Grp) at (t) {$\cat{Grp}$};
    \node[scale=0.8, rounded corners, fill=black!10] (G) at (b) {$\lcore{Grp}$};
    \node[scale=0.8, rounded corners, fill=black!10] (K) at (d) {$\lcore{Grp}$};

    \draw (G) edge[->,"{$\func{I}$}"{name=I, above left}, very thick] (Grp);
    \draw (K) edge[->,"{$\func{I}$}"{above right}, very thick] (Grp);
    \draw (G) edge[->, "{$\func{Z}$}", very thick] (K);

    \draw[Rightarrow] (1,-2.15) -- (-0.6,-1.175) node[midway, above, yshift=2pt] {$\rho$};
  \end{tikzpicture}
	\caption{The center}
	\label{fig:center-counital}
  \end{subfigure}
  \hfill
  \begin{subfigure}[t]{0.33\textwidth}
    \centering
    \pgfmathsetmacro{\rad}{1.25}
    \begin{tikzpicture}
      \coordinate (t) at (0,0);
      \coordinate (b) at (-1.25,-1.8);
      \coordinate (d) at (1.75,-2.5);
  
      \fill[color=red!30] (t) -- (b) -- (d) -- cycle;
  
      \node[scale=0.8, rounded corners, fill=black!10] (Grp) at (t) {$\cat{Grp}$};
      \node[scale=0.8, rounded corners, fill=black!10] (G) at (b) {$\cat{Aut}(G)$};
      \node[scale=0.8, rounded corners, fill=black!10] (K) at (d) {$\cat{Aut}(H)$};
  
      \draw (G) edge[->, "{$\func{J}$}"{name=I, above left}, very thick] (Grp);
      \draw (K) edge[->, "{$\func{K}$}"{above right}, very thick] (Grp);
      \draw (G) edge[->, "{$\func{C}$}"{below}, very thick] (K);
  
      \draw[Rightarrow] (1.2,-2.1) -- (-0.5,-0.95) node[midway, above, yshift=2pt] {$\iota$};
    \end{tikzpicture}
    \caption{A characteristic subgroup}
    \label{fig:single-counital}
    \end{subfigure}  
	\caption{Natural transformations from Example~\ref{ex:three-chars}}
	\label{fig:counital}
\end{figure}
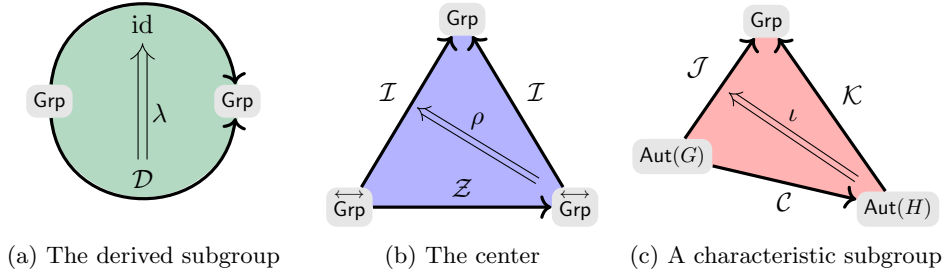

In our study of characteristic structure we 
use  induced actions from Theorem~\ref{thm:char-from-biacts} to
pass from a scalene counital to one that is isosceles. 
We then work with isosceles counitals to determine an intermediate class of isosceles 
counitals known as \emph{internal} counitals. Finally, we show that an internal counital
is completely determined by a morphism of bicapsules.

Counits are common in many categorical contexts; for example, they occur for
every adjoint functor pair.  The case of flat counitals coincides precisely with
the stricter class of fully-invariant substructures.

\subsection{Composing counitals}
In this section, we describe two ways to construct new counitals from given
counitals by composing natural transformations and functors in different ways. These are two instances of a much larger theory; see \cite{Baez, Power}. Figure~\ref{fig:char-ext-2-cells} illustrates the usual composition
of natural transformations. We now  describe how to compose a functor with a
natural transformation. Consider functors $\func{F},\func{G}:\cat{B}\to
\cat{A}$, $\func{H}:\cat{C}\to\cat{B}$, and $\func{K}:\cat{A}\to \cat{D}$ for
categories $\cat{A}$, $\cat{B}$, $\cat{C}$, and $\acat{D}$, and a natural
transformation $\eta : \func{F} \Rightarrow \func{G}$. Define $\eta
\func{H}:\func{F}\func{H}\Rightarrow \func{G}\func{H}$ by setting $(\eta
\func{H})_X\defeq \eta_{\func{H}(X)}$ for each object $X$ in $\cat{C}$.
Similarly, define $\func{K}\eta:\func{K}\func{F}\Rightarrow \func{K}\func{G}$ by
setting $(\func{K}\eta)_Y\defeq\func{K}(\eta_Y)$ for each $Y$ in $\cat{B}$. The
effects of $\eta \func{H}$ and $\func{K}\eta$ are displayed
in~Figure~\ref{fig:nat-trans-functor}.

\begin{figure}[!htbp]
  \pgfmathsetmacro{\rad}{1.5}
  \begin{subfigure}[t]{0.49\textwidth}
    \centering
      \begin{tikzpicture}
        \fill[color=blue!30] (0,0) circle (1.02*\rad cm);
        \fill[color=red!30] (0.5*\rad,0) circle (0.52*\rad cm);
      
        \node[scale=0.8, rounded corners, fill=black!10] (G) at (-\rad, 0) {$\cat{C}$};
      
        \node[scale=0.8, rounded corners, fill=black!10] (H) at (0, 0) {$\cat{B}$};
      
        \node[scale=0.8, rounded corners, fill=black!10] (Grp) at (\rad, 0) {$\cat{A}$};
      
        \draw[very thick, color=black,->] (G) edge["$\func{H}$", pos=0.2] (H);
      
        \draw[very thick, color=black,->] (G) edge[
          out=84, in=97, looseness=1.5, 
          "$\func{G}\func{H}$"{below,name=I}] (Grp);
        \draw[very thick,->] (G) edge[
          out=-84, in=-97, looseness=1.5, 
          "$\func{F}\func{H}$"{below}, 
          ""{name=C, outer sep=2pt}] (Grp);
      
        \draw[very thick, color=black,->] (H) edge[
          out=81, in=99, looseness=1.33, 
          "$\func{G}$"{below,name=J}] (Grp);
        \draw[very thick,->] (H) edge[
          out=-81, in=-99, looseness=1.33, 
          "$\func{F}$"{below}, 
          ""{name=D, outer sep=2pt}] (Grp);
      
        \draw (C) edge[
            out=135, in=-135, 
            arrows=-Implies,
            scaling nfold=2, 
            double distance=3pt,
            "$\eta \func{H}$"{right, name=Iota}, pos=0.2] (I);
        \draw (D) edge[
            arrows=-Implies,
            double distance=3pt,
            scaling nfold=2, 
            "$\eta$"{right, name=Kappa}, pos=0.4] (J);
      \end{tikzpicture}
    \caption{A diagram for $\eta \func{H}$}
    \label{fig:eta-H}
    \end{subfigure}
    \hfill
  \begin{subfigure}[t]{0.49\textwidth}    
    \centering
    \begin{tikzpicture}[xscale=-1] 
      \fill[color=blue!30] (0,0) circle (1.02*\rad cm);
      \fill[color=red!30] (0.5*\rad,0) circle (0.54*\rad cm);
    
      \node[scale=0.8, rounded corners, fill=black!10] (G) at (-\rad, 0) {$\cat{D}$};
    
      \node[scale=0.8, rounded corners, fill=black!10] (H) at (0, 0) {$\cat{A}$};
    
      \node[scale=0.8, rounded corners, fill=black!10] (Grp) at (\rad, 0) {$\cat{B}$};
    
      \draw[very thick, color=black,->] (H) edge["$\func{K}$", pos=0.8] (G);
    
      \draw[very thick, color=black,->] (Grp) edge[
        out=97, in=84, looseness=1.47, 
        "$\func{K}\func{G}$"{below,name=I}] (G);
      \draw[very thick,->] (Grp) edge[
        out=-97, in=-84, looseness=1.47, 
        "$\func{K}\func{F}$"{below}, 
        ""{name=C, outer sep=2pt}] (G);
    
      \draw[very thick, color=black,->] (Grp) edge[
        out=99, in=81, looseness=1.3, 
        "$\func{G}$"{below,name=J}] (H);
      \draw[very thick,->] (Grp) edge[
        out=-99, in=-81, looseness=1.3, 
        "$\func{F}$"{below}, 
        ""{name=D, outer sep=2pt}] (H);
    
      \draw (C) edge[
          out=155, in=-150, 
          arrows=-Implies,
          scaling nfold=2, 
          double distance=3pt,
          "$\func{K}\eta$"{right, name=Iota}, pos=0.2] (I);
      \draw (D) edge[
          arrows=-Implies,
          double distance=3pt,
          scaling nfold=2, 
          "$\eta$"{right, name=Kappa}, pos=0.4] (J);
    \end{tikzpicture}
    \caption{A diagram for $\func{K}\eta$}
    \label{fig:K-eta}
  \end{subfigure}
  \caption{Composing natural transformations with functors}
  \label{fig:nat-trans-functor}
\end{figure}
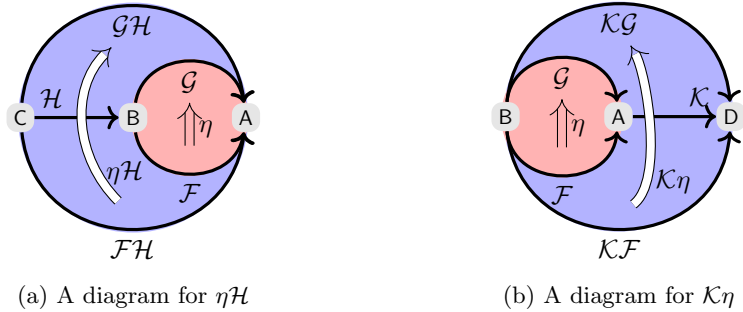

The composition we describe next is specific to natural transformations of a
particular form, which include counitals. It composes two natural
transformations that share a functor and reflects our expectation that the
characteristic relation is transitive. In $\cat{Grp}$, for example,  given a
counital describing a characteristic subgroup $H$ of $G$, and a counital
describing a characteristic subgroup $K$ of $H$, we expect to have a counital
that prescribes how $K$ is characteristic in $G$.

To that end, suppose $\acat{E}$ is a variety with subcategories $\acat{A}$, $\acat{B}$, and $\acat{C}$,
and respective inclusions $\func{I}$, $\func{J}$, and $\func{K}$. Suppose
$\eta:\fJ\fC\Rightarrow \fI$ and $\mu:\func{K}\fD \Rightarrow \fJ$ are
natural transformations. Define $\mu\triangledown \eta :
\func{KDC} \Rightarrow \func{I}$ by
\begin{align*} 
  (\mu\triangledown\eta)_X \defeq \eta_{X}\mu_{\fC(X)}
\end{align*} 
for all objects $X$ in $\acat{A}$, see  Figure~\ref{fig:compose-counitals}. This
construction reflects the fact that being a characteristic substructure is a
transitive property.

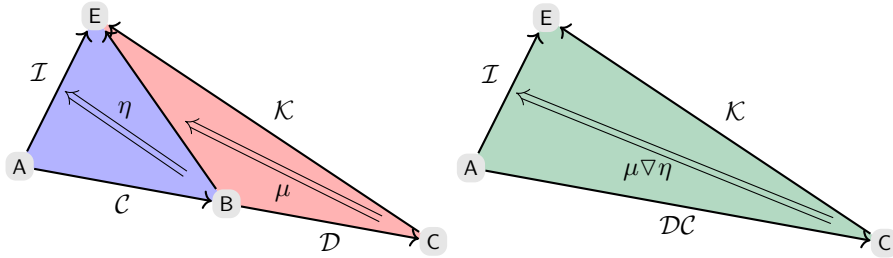
\begin{figure}[!htbp]
  \centering
  \begin{tikzpicture}
    \node at (-3,0) {
      \begin{tikzpicture}
        \coordinate (t) at (0,0);
        \coordinate (b) at (-1,-2);
        \coordinate (d) at (1.75,-2.5);
        \coordinate (c) at (4.5,-3);
    
        \fill[color=blue!30] (t) -- (b) -- (d) -- cycle;
        \fill[color=red!30] (t) -- (c) -- (d) -- cycle;
    
        \node[scale=0.8, rounded corners, fill=black!10] (Grp) at (t) {$\cat{E}$};
        \node[scale=0.8, rounded corners, fill=black!10] (G) at (b) {$\cat{A}$};
        \node[scale=0.8, rounded corners, fill=black!10] (K) at (d) {$\cat{B}$};
        \node[scale=0.8, rounded corners, fill=black!10] (H) at (c) {$\cat{C}$};
    
        \draw (G) edge[->, "{$\func{I}$}"{name=I, above left}, thick] (Grp);
        \draw (K) edge[->, thick] (Grp);
        \draw (G) edge[->, "{$\func{C}$}"{below}, thick] (K);
        \draw (K) edge[->, "{$\func{D}$}"{below}, thick] (H);
        \draw (H) edge[->, "{$\func{K}$}"{name=K, above right}, thick] (Grp);
    
        \draw[Rightarrow] (1.2,-2.1) -- (-0.4,-1) node[midway, above, yshift=2pt] {$\eta$};
        \draw[Rightarrow] (3.8,-2.7) -- (1.2,-1.4) node[midway, below, yshift=-2pt] {$\mu$};
      \end{tikzpicture}
    };
    \node at (3,0) {
      \begin{tikzpicture}
        \coordinate (t) at (0,0);
        \coordinate (b) at (-1,-2);
        \coordinate (d) at (1.75,-2.5); 
        \coordinate (c) at (4.5,-3); 
    
        \fill[color=ForestGreen!30] (t) -- (c) -- (b) -- cycle;
    
        \node[scale=0.8, rounded corners, fill=black!10] (Grp) at (t) {$\cat{E}$};
        \node[scale=0.8, rounded corners, fill=black!10] (G) at (b) {$\cat{A}$};
        \node[scale=0.8, rounded corners, fill=black!10] (H) at (c) {$\cat{C}$};
    
        \draw (G) edge[->, "{$\func{I}$}"{name=I, above left}, thick] (Grp);
        \draw (G) edge[->, "{$\func{DC}$}"{below}, thick] (H);
        \draw (H) edge[->, "{$\func{K}$}"{name=K, above right}, thick] (Grp);
    
        \draw[Rightarrow] (3.8,-2.7) -- (-0.4,-1) node[midway, below, yshift=2pt, xshift=-10pt] {$\mu\triangledown\eta$};
      \end{tikzpicture}
    };
  \end{tikzpicture}   

  \caption{The $\triangledown$-composition of counitals explains transitivity}
  \label{fig:compose-counitals}
\end{figure}

\subsection{Categorifying isosceles counitals}
All extensions used in 
our proof of Theorem~\ref{thm:char-counital-eastern} 
lead to isosceles counitals. 
Counits---namely, counitals
$\fJ\fC\Rightarrow \fI$ in which $\fJ=\fI$ is 
the identity functor---are one
source of isosceles counitals. This hints at a way to characterize
characteristic subgroups.

We now prove that all counitals arising from characteristic subgroups extend to
isosceles counitals, thereby proving Theorem~\ref{thm:char-repn-eastern}. The most direct proof
might utilize \emph{Kan lifts}, the dual of the better known \emph{Kan
extensions}~\cite{Riehl}*{Chapter 6}, but we give a self-contained proof.

\begin{defn}\label{def:internal-counital} 
  Let $\fI:\cat{A}\to\cat{C}$ be an
  inclusion functor of categories, let $\fC:\cat{A}\to \cat{C}$ be a functor,
  and let $\counital : \fC \Rightarrow \func{I}$ be a natural transformation.
  If, for every object $X$ in $\acat{A}$, the morphism $\counital_X :
  \func{C}(X) \to \func{I}(X)$ in $\acat{C}$ is a morphism in $\acat{A}$ (more
  precisely, the image of a morphism in $\acat{A}$ under $\fI$), then  
  $\counital$ is \emph{internal}.
\end{defn}

The property of being internal is strong. Take, for example, $\cat{A} =
\;\core{C}$, so the morphisms are exclusively isomorphisms. If $\counital$ is
internal, then $\iota_X : \func{C}(X) \to X$ is an isomorphism. Such an $\iota$
does not identify a new substructure.  In other words, $\cat{A}$ has too few
morphisms for our purposes. By extending the types of morphisms, we prove in
Proposition~\ref{prop:isosceles-to-internal} that every monic isosceles counital lifts to
an internal one; see~Figure~\ref{fig:extend-isosceles} for an illustration.

\begin{figure}[h]
  \begin{tikzpicture}[thick] 
    \coordinate (t) at (0,0);
    \coordinate (a) at (-1.5,-1.5);
    \coordinate (b) at (-2.5,-2.5);
    \coordinate (c) at ( 1.5,-1.5);
    \coordinate (d) at (2.5,-2.5);

    \fill[color=blue!30] (t) -- (b) -- (d) -- cycle;
    \fill[color=blue!30] (0,-2.5) ellipse (2.4cm and 0.8cm);
    \fill[color=red!30] (t) -- (a) -- (c) -- cycle;
    \fill[color=red!50] (0,-1.5) ellipse (1.4cm and 0.5cm);

    \node[scale=0.8, rounded corners, fill=black!10] (Grp) at (t) {$\cat{E}$};
    \node[scale=0.8, rounded corners, fill=black!10] (H) at (a) {$\cat{A}$};
    \node[scale=0.8, rounded corners, fill=black!10] (G) at (b) {$\cat{B}$};
    \node[scale=0.8, rounded corners, fill=black!10] (L) at (c) {$\cat{A}$};
    \node[scale=0.8, rounded corners, fill=black!10] (K) at (d) {$\cat{B}$};

    \draw (H) edge[->, "{\tiny $\func{K}$}"] (Grp);
    \draw (G) edge[->, "{\tiny $\func{J}$}"] (H);
    \draw (L) edge[->, "{\tiny $\func{K}$}"{above right}] (Grp);
    \draw (K) edge[->, "{\tiny $\func{J}$}"{above right}] (L);
    \draw (H) edge[->, out=30, in=150, dotted, "{\tiny $\id$}"{name=id}] (L);
    \draw (H) edge[->, out=-30, in=-150, "{\tiny $\func{D}$}"{name=D}] (L);
    \draw (G) edge[->, out=-30, in=-150, "{\tiny $\func{E}$}"] (K);

    \draw (D) edge[Rightarrow, "$\hat{\eta}$"{right}] (id);
  \end{tikzpicture}
  \caption{Extending an isosceles counital to an internal one}
  \label{fig:extend-isosceles}
\end{figure}
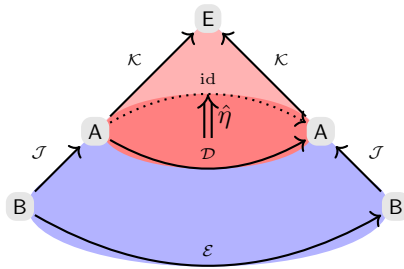

\begin{prop}\label{prop:isosceles-to-internal}
  Let $\acat{E}$ be a category with subcategory $\acat{B}$ and inclusion
  $\func{I}$. Suppose  every object in $\acat{E}$ is also an object in
  $\acat{B}$. Let  $\eta : \func{I}\func{E} \Rightarrow \func{I}$ be a monic
  isosceles counital with $\func{E} : \cat{B} \to \cat{B}$. There exists a
  category $\acat{A}$ with inclusions $\func{J} : \acat{B}\to\acat{A}$ and
  $\func{K}:\acat{A}\to\acat{E}$, a functor $\func{D} : \acat{A}\to\acat{A}$,
  and an internal monic isosceles counital $\hat{\eta} : \func{K}\func{D}
  \Rightarrow \func{K}$ such that $\func{JE}=\func{DJ}$, $\func{I}=\func{KJ}$,
  and $\hat{\eta}\func{J} = \eta$.  
\end{prop}

\begin{proof}
  We define a subcategory $\cat{A}$ of $\cat{E}$ as follows: its objects are the
  objects of~$\acat{E}$; its morphisms are given as finite compositions of
  morphisms $\func{I}(\varphi):\acat{E}$, where $\varphi$ is a morphism in
  $\acat{B}$, and morphisms $\eta_X:\acat{E}$, where $X$ an object in
  $\acat{B}$. Hence, we have inclusions $\func{J} : \cat{B} \to \cat{A}$ and
  $\func{K} : \cat{A} \to \cat{E}$ such that $\func{I} = \func{KJ}$. Since both
  $\acat{A}$ and $\acat{B}$ have the same objects as $\acat{E}$, it follows that
  $\func{I}$, $\func{J}$, and $\func{K}$ are the identities on objects. Moreover, $\func{K}$ is the identity on morphisms.

  We now construct a functor $\fD:\cat{A}\to \cat{A}$ such that $\func{J}
  \func{E}= \fD \func{J}$. It suffices to define $\fD$ on morphisms and then
  verify that $\fD$ is a functor. Set
  \begin{align*} 
    \fD (\varphi) & \defeq \begin{cases}
      \func{JE}(\varphi') & \varphi = \fJ(\varphi')\text{ for a morphism $\varphi'$ in } \cat{B}, \\
      \eta_{\func{E}(X)} & \varphi=\eta_X\text{ for some object $X$ in $\acat{B}$}, \\
      \func{D}(\sigma)\func{D}(\tau) & \varphi=\sigma\tau.
    \end{cases}
  \end{align*}
  If $\fD$ is well defined, then $\func{D}(\varphi)$ is a morphism in
  $\acat{A}$, and $\func{J} \func{E}=\fD \func{J}$ by construction. To verify
  that $\fD$ is well defined, it suffices to consider the case where  $\eta_X$
  (with $X$ an object in $\acat{B}$) is also a morphism in $\acat{B}$:
  specifically, there is a morphism $\beta :\acat{B}$ such that $\eta_X =
  \func{I}(\beta)$. Since $\fI$ is the identity on objects, $\beta : \fE(X) \to
  X$. We will show that $\eta_{\func{E}(X)}=\fK\fD(\eta_X)=\func{IE}(\beta)$. To
  see this, we apply $\eta$ to the morphism $\beta:\func{E}(X)\to X$ and obtain
  the following diagram (see shaded entry $(2,2)$ of
  Figure~\ref{fig:functor-cat-act}).
  \[ 
    \begin{tikzcd}
      \func{I}\func{E}\func{E}(X) \arrow[r, "\func{IE}(\beta)"] \arrow[d, "\eta_{\fE(X)}"] 
        & \func{I}\func{E}(X) \arrow[d, "\eta_{X}"] \\
      \func{I}\func{E}(X) \arrow[r, "\func{I}(\beta)"] & \func{I}(X)
    \end{tikzcd}
  \]
  Since $\eta_X = \fI(\beta)$, the diagram implies that $\eta_X
  \eta_{\func{E}(X)}=\eta_X \func{IE}(\beta)$. Since $\eta_X$ is monic by
  assumption,  $\func{IE}(\beta)=\eta_{\func{E}(X)}$. This proves that
  $\func{D}$ is well defined. 

  We claim that there exists a natural transformation
  $\hat{\eta}:\func{K}\func{D}\Rightarrow \func{K}$ such that
  $\hat{\eta}\func{J} = \eta$. 
  Since the
  objects of $\cat{A}$ are those of $\cat{B}$, we define $\hat{\eta}_X$ to be
  $\eta_X$ and show that this yields the required counital. First, we consider
  the case that $\varphi : X\to Y$ is a morphism in $\cat{B}$. Then $\func{K} \func{D}
  \func{J}(\varphi)=\func{K} \func{J} \func{E}(\varphi)=\func{I}
  \func{E}(\varphi)$, so
  \begin{align*}
    \hat{\eta}_Y \func{K} \func{D}(\func{J}(\varphi)) 
    & = \eta_Y \func{I} \func{E}(\varphi) = \func{I}(\varphi)\eta_X 
     = \func{K}(\func{J}(\varphi))\hat{\eta}_X.
  \end{align*} 
  Now we assume $\varphi = \eta_X: \func{IE}(X)\to \func{I}(X)$ for some
  object $X$ in $\cat{B}$. Since $\fI$ is the identity on objects and $\fK$ is the identity on morphisms, 
  \begin{align*}
    \hat{\eta}_{\func{I}(X)} \func{K} \func{D}(\eta_X) 
    & = \hat{\eta}_X\func{KD}(\eta_X) = \eta_{X} \func{K}(\eta_{ \func{E}(X)})  
     = \eta_{X} \eta_{ \func{E}(X)}  = \func{K}(\eta_X)\hat{\eta}_{ \func{E}(X)}.
  \end{align*}
  Lastly, we consider the case of an arbitrary finite composition
  $\varphi=\varphi_1\cdots \varphi_n$ where each  $\varphi_k$ is either $\fJ(\varphi_k')$ for some morphism $\varphi_k'$ in $\acat{B}$ or a morphism $\eta_X$ for some object $X$ in $\acat{B}$. It suffices to consider
  only the case where $n=2$, say $\varphi = \varphi_1\varphi_2$ with $\varphi_2
  : X\to Z$ and $\varphi_1: Z\to Y$. Now 
  \begin{align*}
    \hat{\eta}_Y \func{K} \func{D}(\varphi) 
    & = \hat{\eta}_Y \func{K} \func{D}(\varphi_1) \func{K} \func{D}(\varphi_2) \\
    & = \func{K}(\varphi_1)\hat{\eta}_{Z} \func{K} \fD(\varphi_2) \\
    & = \func{K}(\varphi_1) \func{K}(\varphi_2)\hat{\eta}_{X}\\
    & = \func{K}(\varphi)\hat{\eta}_X.
  \end{align*} 
 Thus, 
$\hat{\eta}:\func{K} \func{D}\Rightarrow \func{K}$. Since $\eta$ is monic, so is $\hat{\eta}$.  Also, $\hat{\eta}_X$ is a morphism in $\cat{A}$ for every object $X$, so it is  internal, as claimed.
\end{proof}

We now prove that every characteristic substructure of an algebra in a variety 
is induced by a morphism of category biactions.

\begin{thm}\label{thm:char-from-biacts} 
  Let $X$ be an object in a variety
  $\cat{E}$. Let $Y$ be characteristic in $X$ with inclusion $\iota: Y\to X$.
  There exist subcategories $\cat{A}$ and $\cat{B}$ with $\core{E}\leq
  \cat{A},\cat{B}\leq \cat{E}$, and an $(\acat{A},\acat{B})$-morphism $\func{M}
  : \acat{B} \to \acat{A}$ such that $\func{M}(\id_X\cdot \one_{\acat{B}}) =
  \iota$.
\end{thm} 

\begin{proof}
Let $\func{I}: \core{E}\to \acat{E}$ be  the inclusion functor. The proof of
  Theorem~\ref{thm:char-counital-eastern} shows that there exists  a functor
  $\func{E}:\;\core{E}\;\to \;\core{E}$ and a monic counital $\eta:\func{I}
  \func{E}\Rightarrow \func{I}$ such that  $\eta_X=\iota$. We use
  Proposition~\ref{prop:isosceles-to-internal}  (with $\cB=\;\core{E}$) to create
  a category $\cat{A}$ generated from $\core{E}$ and $\eta$, an inclusion
  functor $\func{K} : \cat{A} \to \cat{E}$, a functor $\func{D} : \cat{A} \to
  \cat{A}$, and an internal monic counital $\hat{\eta} : \func{K}\func{D}
  \Rightarrow \func{K}$ with  $\hat{\eta}_{Z}=\eta_Z$ for all objects in
  $\core{E}$. Lastly, we  apply Proposition~\ref{prop:nat-trans-biact}(a) to
  $\hat{\eta}$ to obtain an $\acat{A}$-bimorphism $\mathcal{N} : \acat{A} \to
  \acat{E}$ such that $\hat{\eta} = \mathcal{N}(\one_{\acat{A}})$. Since
  $\hat{\eta}$ is internal, there exists an $\acat{A}$-bimorphism $\mathcal{M} :
  \acat{A} \to \acat{A}$ such that $\mathcal{N}=\mathcal{KM}$. Hence,
  $\hat{\eta}=\func{K}\func{M}(\one_{\acat{A}})$. With $\acat{B} \defeq
  \acat{A}$, it follows that  $\mathcal{M}(\id_X\cdot \one_{\acat{B}})
  =\hat{\eta}_X=\eta_X=\iota$, as claimed.
\end{proof}

\subsection{Proofs of main theorems}\label{sec:proofs} Having developed the
required theory, we can now complete the proofs of our main results.
Theorem~\ref{thm:char-counital} is a special case of
Theorem~\ref{thm:char-counital-eastern}, which we proved 
in the previous section.

\medskip 

\noindent\textit{Proof of Theorem} \ref{thm:char-repn-eastern}.
If (1) holds, then Theorem~\ref{thm:char-from-biacts} yields (3). 
If (3) holds, then (2) follows from Theorem~\ref{thm:counit-capsules}\ref{thmpart:bicap-to-counit}  and the fact that 
$\iota=\mathcal{M}(\id_G\cdot\one_{\cat{B}})$. If (2) holds, then (1) follows from 
Theorem~\ref{thm:char-counital-eastern}. \qed  

\medskip
\noindent
Theorem~\ref{thm:char-repn} follows from 
Theorem~\ref{thm:char-repn-eastern}.

\subsection{Duality} 
\label{sec_dual} 
 Recall from  Section~\ref{sec:proofThm1} that a natural
transformation $\eta : \func{I} \Rightarrow \func{D}$ is a unital if  $\func{I}$ is an
inclusion functor. If $\func{I} = \id$, then $\eta : \id
\Rightarrow \func{D}$ is a \emph{unit}. A unital $\eta : \func{I} \Rightarrow
\func{D}$ is \emph{epic} if $\eta_X : \func{I}(X) \to \func{D}(X)$ is
epic for all objects $X$. Units and unitals are the duals 
of counits and counitals.

We state a dual analogue of Theorem~\ref{thm:char-repn-eastern} for
characteristic quotients of algebras in varieties; its proof follows
\emph{mutatis mutandis} from that of Theorem~\ref{thm:char-repn-eastern}.

\begingroup     
\renewcommand{\themainthm}{2-dual} 
\begin{mainthm}\label{thm:general-dual}
  Let $\acat{E}$ be a variety, and let $G$ be an object of
  $\acat{E}$ with quotient $Q$ and projection $\pi$. There exist categories
  $\acat{A}$ and $\acat{B}$, where $\core{E}\; \leq \acat{A} \leq \acat{E}$,
  such that the following are equivalent.
  \begin{ithm}
    \item[\rm (1)] $Q$ is a characteristic quotient of $G$.
    \item[\rm (2)] There is a functor $\func{U} : \acat{A} \to \acat{A}$ and a unit
    $\epsilon : \id_{\acat{A}} \Rightarrow \func{U}$ such that $Q =
    \mathrm{Coim}(\epsilon_{G})$.
    \item[\rm (3)] There is an $(\acat{A},\acat{B})$-morphism $\func{M} : \acat{A} \to
    \acat{B}$ such that $\pi = \func{M}(\one_{\acat{A}}\cdot \id_G)$.
  \end{ithm}
\end{mainthm}
\endgroup

Although a characteristic subgroup of a group $G$ is associated with a
characteristic quotient of $G$, and vice versa, there are subtle differences in
other categories of algebraic structures. 

\begin{ex}\label{ex:unital-rings}
  Let $\mathbb{Q}$ be the ring of rational numbers and $\mathbb{Z}$ its subring
  of integers. If $\varphi : \mathbb{Q}\to\mathbb{Q}$ is a homomorphism of
  unital rings, then $\varphi(1)=1$. This forces $\varphi=\id_\mathbb{Q}$, so
  $\mathbb{Z}$ is fully invariant in $\mathbb{Q}$. Since $\mathbb{Q}$ is a
  field, its only quotients are itself and the trivial ring. Hence, $\mathbb{Q}$
  has many fully-invariant substructures, but only two fully-invariant
  quotients. More generally, if $R$ is a unital ring, then the kernel of a ring homomorphism with domain $R$ is not necessarily a unital subring of $R$.
  \exqed 
\end{ex} 

  By contrast, the kernel of every group homomorphism is a normal subgroup. Up to equivalence of natural transformations in $\cat{Cat}$, invariant structures of groups are \emph{self-dual}. The next proposition provides a categorical description of this observation for $\cat{Grp}$; we use it in 
Section \ref{sec:examples}.

\begin{prop}\label{prop:kernel}
  The following hold for categories 
  $\lcore{Grp}\;\leq \cat{A}\leq \cat{Grp}$ and $\cat{B}\leq \cat{Grp}$
  with inclusion functors $\func{I}:\acat{A} \to \acat{Grp}$ and
  $\func{J}:\acat{B}\to \acat{Grp}$.
  \begin{ithm} 
    \item Given a unital $\unital:\func{I}\Rightarrow \func{J}\func{U}$, there
    is a sub category $\cat{C}\leq \cat{Grp}$ with inclusion $\func{K}$, and a
    functor $\func{C} : \cat{A} \to \cat{C}$ such that $\ker(\unital) :
    \func{K}\func{C}\Rightarrow \func{I}$ is a counital where $\func{C}(G) = \ker(\unital_G)$ and $(\ker (\unital))_G
    : \ker(\unital_G)\hookrightarrow G$ is the inclusion for every group~$G$.

    \item Given a counital $\counital:\func{J}\func{C}\Rightarrow \func{I}$,
    there is a subcategory $\cat{C}\leq \cat{Grp}$ with inclusion 
    $\func{K}$, and a functor $\func{U} :
    \cat{A}\to \cat{C}$ such that $\coker (\counital): \func{I}\Rightarrow
    \func{K} \func{U}$ is a unital where $\func{U}(G)=G/\im (\counital_G)$ and $(\coker (\counital))_G : G \twoheadrightarrow G/\im (\counital_G)$ for every group $G$.

    \item With the notation of {\rm (a)} and {\rm (b)},  there are unique invertible 
    $\mu,\tau:\cA$ such that   $\coker (\ker (\unital))=\mu (\mathrm{im}(\pi))$ and 
    $\ker
    (\coker (\counital))= \counital\tau$.
  \end{ithm}
\end{prop}

\begin{proof}
  \begin{iprf}
  \item For every morphism $\varphi:G\to H$ in $\cat{A}$, there is an induced
  morphism $\varphi':\im(\unital_G)\to \im (\unital_{H})$ such that
  $\varphi'\unital_{G}=\unital_H\varphi$, so
  \[
    \unital_{H}\varphi(\ker(\unital_G))
    =\varphi' \unital_G (\ker (\unital_G))=1.
  \]
  Therefore $\varphi(\ker(\unital_G))\leq \ker (\unital_{H})$.
  In particular, the restriction 
  \begin{equation*}
    \varphi|_{\ker (\unital_G)}:\ker(\unital_G)\to \ker(\unital_{H})
  \end{equation*}
  is well defined. Let $\cat{C}$ be the category whose objects
  are $\ker(\unital_G)$ for all groups $G$ and whose morphisms are $\varphi|_{\ker (\unital_G)}$ for all morphisms $\varphi : G\to  H$ in $\cat{A}$. Let $\func{K}: \acat{C}\to \acat{Grp}$ be the inclusion functor. Moreover,
  there is a functor $\func{C}: \cat{A} \to \cat{C}$ given
  by $\func{C}(G) = \ker(\unital_G)$ and $\func{C}(\varphi) = \varphi|_{\ker
  (\unital_G)}$. If we define $\counital_G: \ker(\unital_G)\hookrightarrow G$ to be the
  associated inclusion map for the kernel, then  $\counital:
  \func{K}\func{C}\Rightarrow \func{I}$ is the required counital.

  \item The proof is dual to that of (a).

  \item Consider the unital $\unital : \func{I} \Rightarrow
  \func{J} \func{U}$. By Theorem~\ref{thm:Noether},
  for each group $G$ there is an isomorphism
  \[ 
    \mu: \fU(G)=\mathrm{Im}\pi_G \to  G/\ker\unital_G=\coker (\ker\unital_G).
  \] 
  Thus, 
$\coker (\ker (\unital))=\mu (\text{im}(\pi))$; 
   likewise, for $\ker (\coker (\counital))$ and  $\counital$. \qedhere
  \end{iprf}
\end{proof}

\section{Categorification of standard characteristic subgroups}
\label{sec:examples}

Theorem~\ref{thm:char-repn} states that every characteristic subgroup can be studied in
three ways: as a group, as a natural transformation, and as a morphism of category
biactions. In this section, we describe common characteristic subgroups using  
all three forms. In so doing, we reveal insights gained from the categorical perspective. 

Throughout, we use the following notation for restriction and induction. Let $\varphi : G \to H$ be a homomorphism of groups, and let
$\func{C}(G)$ and $\func{C}(H)$ be subgroups of $H$ and $G$, respectively. 
If the restriction of $\varphi$ to $\func{C}(G)$ maps into $\func{C}(H)$, 
then we denote it by 
\begin{align}\label{eqn:restriction} 
  \varphi|_{\func{C}} : \func{C}(G) \to \func{C}(H),\quad c\mapsto \varphi(c).
\end{align} 
Similarly, if $\varphi$ maps a normal subgroup  $\func{Q}(G)$ of $G$ into a normal subgroup $\func{Q}(H)$ of $H$, then the \emph{induction} of $\varphi$ via $\func{Q}$ is
\begin{align}\label{eqn:induction}
  \varphi|^{\func{Q}} : G/\func{Q}(G) \to H/\func{Q}(H),\quad g\func{Q}(G) \mapsto \varphi(g)\func{Q}(H).
\end{align}

\subsection{Abelianization and derived subgroups}\label{sec:abelianization}

Figure~\ref{fig:comm} gives the three perspectives on the derived subgroup. We
develop this example so that we may also treat the lower central series 
and all verbal subgroups in Section~\ref{sec:lower-central}.

The counital $\lambda : \func{D}\Rightarrow \id_{\cat{Grp}}$ 
of Example~\ref{ex:three-chars}
associated with the derived subgroup $\gamma_2(G)$ of a group $G$ 
can be constructed also as the kernel of the unital 
associated with abelianization.
We explore the category biaction interpretation.
Let $\acat{Abel}$ be the category of abelian groups, a subcategory of
$\cat{Grp}$ with inclusion $\func{I} : \acat{Abel} \to \acat{Grp}$. We define a
morphism $\func{A} : \cat{Grp} \to \cat{Abel}$ given by $\varphi\mapsto
\varphi|^{\gamma_2}$. The functors $\func{A}$ and
$\func{I}$ turn the categories $\acat{Grp}$ and $\acat{Abel}$ into $(\acat{Grp},
\acat{Abel})$-bicapsules. 

We show that $\func{A}:\acat{Grp}\to\acat{Abel}$ is a $(\acat{Grp},
\acat{Abel})$-morphism. Let $\varphi$ and $\tau$ be group homomorphisms, and
let $\alpha$ be a homomorphism of abelian groups. Now 
\begin{align*}
  \func{A}(\alpha\cdot \varphi\tau) &= (\func{I}(\alpha)\varphi\tau)|^{\gamma_2} = \alpha\;\varphi|^{\gamma_2}\;  \tau|^{\gamma_2} =\alpha\func{A}(\varphi)\cdot \tau.
\end{align*}
To obtain the counital associated with the derived subgroup, we apply
Proposition~\ref{prop:kernel} and take the kernel of $\func{A}(\one_{\acat{Grp}})$. Since
the unital-counital pair obtained through this process is a unit-counit pair, 
 we
obtain the well-known observation that the derived subgroup is fully invariant.

\begin{figure}[h]
  \centering
  \includegraphics[width=\textwidth]{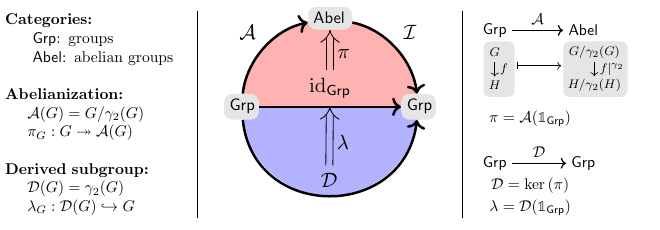}
  \caption{Three perspectives on the derived subgroup}
  \label{fig:comm}
\end{figure}

\subsection{Verbal subgroups}
\label{sec:lower-central}

We generalize the approach taken in Section~\ref{sec:abelianization}. Let $\Omega$ be
the group signature from~Example~\ref{ex:group-grammar}. To each set $W$ of words from
the free group $\Omega\langle X\rangle$ we associate a category $\acat{Var}(W)$ as follows (see Section~\ref{sec:free}). 
For each word $w :W$, group
$G$, and $X$-tuple $g\tin G^X$, define $w_G : G^X \to G$ by $g\mapsto \mathrm{eval}_g(w)$. 
Define $\acat{Var}(W)$ to be the full subcategory of $\cat{Grp}$
with objects 
\[
  \{ G:\cat{Grp} \mid (\forall g\tin G^X) (\forall w\tin W)\; w_G(g)=1\}
\]
with inclusion functor $\func{I}:\acat{Var}(W)\to \cat{Grp}$. The category
$\acat{Var}(W)$ is the \textit{group variety} with laws $W$. Let $\mathrm{Rad}_W(G)$ be the minimal normal subgroup of a group $G$ such that $G/\mathrm{Rad}_W(G)$ is in $\mathsf{Var}(W)$. Let
$\func{R}:\cat{Grp}\to \acat{Var}(W)$ be the functor such that $\func{R}(G)$ is the
largest quotient of $G$ contained in $\acat{Var}(W)$, where the functor carries $G$ 
to $G/\mathrm{Rad}_W(G)$, and morphisms $\varphi$ are sent to $\varphi|^{\mathrm{Rad}_W}$. 

\begin{prop}\label{prop:verbal}
The functors $\func{R}$ and $\func{I}$ defined above form an adjoint functor pair
  $\func{R} : \acat{Grp}\dashv \acat{Var}(W) : \func{I}$.
\end{prop}

\begin{proof}
  By Proposition~\ref{prop:functors-are}, the functors $\func{R}$ and $\func{I}$ turn both
  $\acat{Var}(W)$ and $\acat{Grp}$ into $(\acat{Var}(W),\acat{Grp})$-bicapsules. The
  functor $\func{R}$ is a $(\acat{Var}(W),\acat{Grp})$-morphism: for morphisms
  $\alpha$ in $\acat{Var}(W)$ and $\varphi,\tau$ in $\acat{Grp}$, 
  \begin{align*}
    \func{R}(\alpha\varphi\cdot\tau) &= (\alpha\varphi\func{I}(\tau))|^{\mathrm{Rad}_W} = \alpha|^{\mathrm{Rad}_W}\; \varphi|^{\mathrm{Rad}_W}\; \tau = \alpha\cdot \func{R}(\varphi)\tau.
  \end{align*}
  Since $\func{R}$ and $\func{I}$ are pseudo-inverses,
  the result follows from Theorem~\ref{thm:iso-biacts-adjoints}\ref{thmpart:biaction-to-adjoint}. 
\end{proof}

The adjoint functor pair in Proposition~\ref{prop:verbal} categorifies verbal subgroups.
The dual version of Theorem~\ref{thm:counit-capsules} describes how to obtain the
unit $\pi : \id_{\acat{Grp}}\Rightarrow \func{IR}$ from~$\func{R}$. Applying
Proposition~\ref{prop:kernel}, the kernel of $\pi$ yields a counit $\iota : \func{V}
\Rightarrow \id_{\acat{Grp}}$ for some functor $\func{V}: \acat{Grp} \to
\acat{Grp}$. If $G$ is a group, then  $\func{V}(G)$ is the \emph{$W$-verbal} subgroup.
We conclude that all verbal subgroups are fully invariant.
Thus, from Proposition~\ref{prop:verbal}, we get an \emph{exact sequence} of natural
transformations
\begin{center}
\begin{tikzcd}
  \func{V} \arrow[r,Rightarrow,"\ker(\pi)"] & \id_{\cat{Grp}} 
  \arrow[r,Rightarrow,"\unital"] & \func{IR}.
\end{tikzcd}
\end{center} 
The corresponding diagram appears in Figure~\ref{fig:verb-W}.

\begin{figure}[h]
  \centering
  \includegraphics[width=\textwidth]{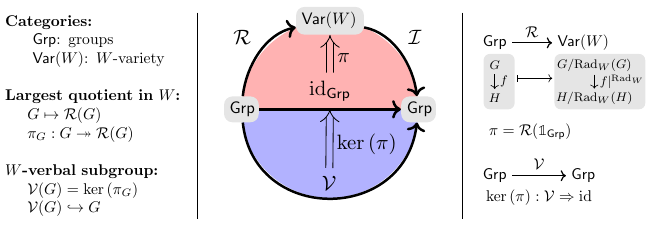}
  \caption{Three perspectives on verbal subgroups}
  \label{fig:verb-W} 
\end{figure}

\subsection{Marginal subgroups}
\label{sec:center}

Now we consider characteristic subgroups such as the center $\zeta(G)$ of a
group $G$.  As seen in Example~\ref{ex:three-chars_first},  there are group homomorphisms 
$\varphi:G\to H$ for which $\varphi(\zeta(G))\not\leq \zeta (H)$, so,
unlike verbal subgroups, the center is not fully invariant. This fact is
revealed by the categorification of the center---it does not yield a counit
between functors $\acat{Grp}\to \acat{Grp}$, but rather a proper counital between functors of
the form $\lepi{Grp}\; \to\acat{Grp}$, where $\lepi{Grp}$ is the category of
groups whose morphisms are epimorphisms. We establish this fact more generally
for the class of \textit{marginal subgroups} introduced 
by P.\ Hall \cite{Hall40}.

\begin{ex}[Hall's Isoclinism]\label{ex:isoclinism}
 For an integer $n>0$ we write $G^n$ for the $n$-fold direct product 
of a group $G$. 
  The commutator map $\kappa:G^2 \to G$ 
is given by $(g,h) \mapsto [g,h]\defeq g^{-1}h^{-1}gh$.
We define a congruence relation
  $\equiv$ on $G$ and write $x\equiv z$ if and only if $[x,z]=[y,z]$ for all $y: G$.
Factoring through
  this congruence relation and restricting the outputs to the verbal subgroups,
  we obtain a 
map $*:(G/\zeta(G))^2\to \gamma_2(G)$ such that the following diagram commutes.
  \begin{equation*} 
    \begin{tikzcd}
      G^2\arrow[r,"{\kappa}"]\arrow[d,two heads] & G\\
      (G/\zeta(G))^2\arrow[r,"*"] & {\gamma_2(G)}\arrow[u,hook]
    \end{tikzcd}
  \end{equation*}  
  Two groups are \emph{isoclinic} if their 
commutator maps are equivalent.~\exqed
\end{ex}

For each group $G$ and each word $w$, there is a unique minimal normal subgroup $w^*(G)$ such that the
  map $\overline{w}_G : (G/w^*(G))^n \to G$ given by 
  \begin{align*}
    (g_1w^*(G),\dots, g_nw^*(G)) &\longmapsto w_G(g_1,\dots, g_n)
  \end{align*}
  is non-degenerate: namely,
fixing any $n-1$ entries of the $n$-tuple argument of 
$\overline{w}_G$ yields an injective map $G/w^*(G) \to G$. 
Here $w_G$ is as defined in Section~\ref{sec:lower-central}.

For a set $W$ of words, the associated \emph{marginal subgroup} of a group $G$
is defined as  $W^*(G)\defeq\bigcap_{w:W} w^*(G)$. Clearly, $W^*(G)$ is
characteristic in $G$. The image of $\overline{w}_G$, and thus also $w_G$, is
the verbal subgroup $w(G)$ associated with $w$.
\begin{figure}[!tbp] 
  \centering
  \includegraphics{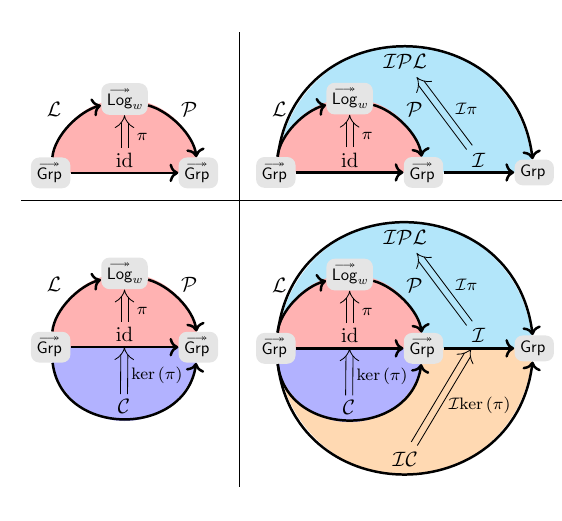}
  \caption{Marginal subgroups and quotients categorified}
  \label{fig:log-W-star} 
\end{figure}

Hall \cite{Hall40}
introduced the general notion of \emph{isologism} for word-map
equivalence. We extend this language to categories.
Each word $w$ determines a category $\lepi{Log}_{w}$ with maps
$\overline{w}_G : (G/w^*(G))^n\to w(G)$ as objects, where the 
morphisms are pairs
$(\varphi_1,\varphi_2)$ of group epimorphisms such that the following diagram
commutes.
\begin{center}
  \begin{tikzcd}
    (G/w^*(G))^n\arrow[r, "\overline{w}_G"]\arrow[d, "\varphi_1^n"] & w(G) \arrow[d, "\varphi_2"] \\ 
    (H/w^*(H))^n \arrow[r, "\overline{w}_H"] & w(H)
  \end{tikzcd}
\end{center}
We define two functors. The first is $\func{L} : \;\lepi{Grp}\; \to \;\lepi{Log}_w$
given by $G\mapsto \overline{w}_G$ and $\varphi\mapsto (\varphi|^{w^*},
\varphi|_w)$. The second is $\func{P}:\;\lepi{Log}_w\;\to \;\lepi{Grp}$ given by
$\overline{w}_G\mapsto G/w^*(G)$ and $(\varphi_1,\varphi_2)\mapsto \varphi_1$.
For a group $G$, let $\pi_G : G \twoheadrightarrow G/w^*(G)$ be the usual projection homomorphism. 
Now $\pi :
\id_{\lepi{Grp}}\Rightarrow \func{PL}$ is a unit. Let $\func{I} :
\;\lepi{Grp}\; \to \acat{Grp}$ be the inclusion functor.
Then the unital $\func{I}\pi : \func{I}\Rightarrow
\func{IPL}$ is a categorification of marginal quotients.

To categorify the marginal
subgroup, we take the kernel of $\pi$ via Proposition~\ref{prop:kernel} and compose with
$\func{I}$: namely, $\func{I}\ker(\pi) : \func{IC}\Rightarrow \func{I}$ for some
functor $\func{C} : \;\lepi{Grp}\; \to \;\lepi{Grp}$. Figure~\ref{fig:log-W-star}
displays the various morphisms and their relationships. This
construction demonstrates that  marginal subgroups are not just
characteristic, but invariant under all epimorphisms.

The construction applies to other algebraic structures by simply involving
formulas in the appropriate signature. However, the notion of congruence does
not always yield a substructure, so the structures are 
more naturally expressed as characteristic quotients.

\section{Composite characteristic structures}
\label{sec:compose}
We now address one remaining powerful feature of our categorical description of characteristic 
structure. It relates to a comment we made 
after Theorem~\ref{thm:char-repn}:  a characteristic subgroup  may 
arise from $(\acat{A},\acat{B})$-morphisms $\acat{B}\to\acat{A}$
where $\acat{B}$ is \textit{not} a category of groups. 
We give one illustration of how this ``transferability'' 
explains techniques currently used in isomorphism tests. 

In~\cite{Wilson:filters}*{\S 4}, it is shown that
a $p$-group $G$ of class at most $2$ with exponent $p$ has a 
characteristic subgroup induced by the Jacobson radical of an 
algebra associated to the bilinear commutator map of $G$. 
Here we construct that characteristic subgroup using a tensor 
product of capsules, as described in~Section~\ref{sec:ext-prob}.

\subsection{From groups to bimaps}\label{sec:delta}

Fix an odd prime $p$, and let  $\cat{G}\defeq \lcore{Grp}_{2,p}$ be the category
whose objects are $p$-groups of class at most $2$ with exponent $p$, and whose
morphisms are isomorphisms. The objects of $\cat{G}$ are groups $G$ with
exponent $p$ and central derived subgroup, so $\gamma_2(G)\leq \zeta(G)$. 

Let $\mathbb{F}_p$ be the field with $p$ elements, and let 
$\cat{B}\defeq\lcore{Bi}(\mathbb{F}_p)$ be the category of alternating $\mathbb{F}_p$-bilinear maps. 
The objects of $\cat{B}$ are bilinear maps $b: V\times V\to W$, 
where $V$ and $W$ are $\mathbb{F}_p$-spaces, such that $b(u,v)=-b(v,u)$ 
for all vectors $u,v$. For
objects $b : V\times V \to W$ and $b' : V'\times V' \to W'$ in $\acat{B}$, 
a morphism $\varphi : b \to b'$ is a pair of invertible linear maps 
$(\alpha : V\to V', \beta: W\to W')$ such that,
 for all $u,v\in V$,
\begin{equation*}
     b'(\alpha u, \alpha v)  = \beta b(u,v).
\end{equation*}

Define a functor  
$\func{B} : \cG\to \cB$ that takes a group $G$ to 
\[
    b_G: G/\gamma_2(G) \times G/\gamma_2(G) \to \gamma_2(G),
    \quad (x\gamma_2(G),y\gamma_2(G))\mapsto [x,y],
\]
and a homomorphism $\varphi : G\to H$ to the pair $(\varphi|^{\gamma_2},\
\varphi|_{\gamma_2})$, as defined in \eqref{eqn:restriction} and
\eqref{eqn:induction}. Since $G$ has exponent $p$ and $\gamma_2(G)\leq \zeta(G)$
by assumption,  
$b_G$ is an alternating $\mathbb{F}_p$-bilinear map.

Next, define a functor $\func{G} : \cB \to \cG$
that takes an $\mathbb{F}_p$-bilinear map
$b : V\times V\to W$ to the group $G_b$ on
$V\times W$ with binary operation
\[
    (v_1,w_1) \cdot (v_2,w_2) = \left(v_1+v_2, w_1+w_2 + \frac{1}{2}b(v_1,v_2)\right).
\] 
A morphism $(\alpha, \beta)$ from $b:V\times V\to W$ to $b':V'\times V'\to W'$ 
in $\cat{B}$ induces a group isomorphism, denoted $\alpha\boxtimes \beta$,
mapping $G_b=V\times W$ to $G_{b'}=V'\times W'$
by
\begin{equation*}
    (\alpha\boxtimes \beta)(v,w) \defeq (\alpha u, \beta w). 
\end{equation*}

\begin{lem}
    \label{lem:Baer}
    The functor $\func{B}:\cG\to \cB$ is  
    a $(\cG,\cB)$-morphism.
\end{lem}

\begin{proof}
    The functor $\func{B}$ induces a left $\acat{G}$-action on (the morphisms of) 
    $\cB$,
    and $\func{G}$ induces a right $\acat{B}$-action on $\cG$,
    so $\cB$ and
    $\cG$ are $(\cB,\cG)$-bicapsules.
    Let $\lambda, \mu$ be morphisms of $\cG$ and let $(\alpha,\beta)$ be a
    morphism of $\cB$ such that $\lambda \mu\cdot (\alpha,\beta) =
    \lambda\mu (\alpha\boxtimes \beta)$ is 
not $\bot$.
     Now
    \begin{align*}
        \func{B}(\lambda\mu\cdot (\alpha,\beta)) 
        &= \left((\lambda\mu(\alpha\boxtimes\beta))|^{\gamma_2},\ (\lambda\mu(\alpha\boxtimes\beta))|_{\gamma_2}\right) \\
        &= \left(\lambda|^{\gamma_2} \mu|^{\gamma_2}\alpha,\ \lambda|_{\gamma_2} \mu|_{\gamma_2} \beta\right) \\ &= \lambda \cdot\func{B}(\mu) (\alpha,\beta),
    \end{align*}
    so $\func{B}$ is a $(\cG,\cB)$-morphism. 
\end{proof}

By applying the dual version of Theorem~\ref{thm:counit-capsules}\ref{thmpart:bicap-to-counit},
    we obtain a unit $\id_{\cG}\Rightarrow \func{BG}$.  
    There is also a 
    counit 
    $\id_{\cG}\Leftarrow \func{BG}$. Together
    these give a categorical interpretation 
    of the {\it Baer correspondence}~\cite{Baer}.

\subsection{From bimaps to algebras}
\label{sec:gamma}
Let $\cA\defeq \lcore{Alge}(\mathbb{F}_p)$ be the category of $\mathbb{F}_p$-matrix algebras with 
algebra isomorphisms.  Using \cite{Wilson:filters}*{\S 4}, define a functor 
$\func{A}:\cB\to \cA$ by
\begin{align*} 
    \func{A}(b)&= \left\{f \in \End(V)
    \mid (\exists f^*\in\End(V)^{{\rm op}})(\forall u,v \in V)\; b(fu, v) = b(u, f^*v) \right\}.
\end{align*}
Invertible morphisms $(\alpha,\beta)$ in $\cB$ from 
$b:V\times V\to W$ to $b':V'\times V'\to W'$ are  sent to 
\begin{equation*}
    \func{A}(\alpha,\beta):
    f\in \func{A}(b) \mapsto f^{\alpha^{-1}}\in \func{A}(b').
\end{equation*}
  
\begin{fact}
\label{fact:adj-ten}
    The functor $\func{A}$ is a $(\cB,\cA)$-morphism.
\end{fact}

\subsection{From matrix algebras to semisimple algebras}
\label{sec:upsilon}
Every matrix algebra $A$ over a field is Artinian, so 
the quotient 
of $A$ by its Jacobson radical $\mathrm{Jac}(A)$ 
is semisimple. The map 
$A\mapsto A/\mathrm{Jac}(A)$ is 
a functor from $\acat{A}$ to the 
category $\cS\defeq$ \mbox{$\lcore{SSAlge}(\mathbb{F}_p)$} of 
semisimple $\mathbb{F}_p$-algebras. It is also an 
$(\cA,\cS)$-morphism.

\subsection{Combining capsules}

Recall that 
\begin{align*}
\cat{G}= \lcore{Grp}_{2,p}, \qquad \cat{B}= \lcore{Bi}(\mathbb{F}_p), \qquad
\cat{A}= \lcore{Alge}(\mathbb{F}_p), \qquad \cat{S}= \lcore{SSAlge}(\mathbb{F}_p).
\end{align*}
Denote by $\Delta$ the bicapsule associated to the $(\cG,\cB)$-morphism in Lemma~\ref{lem:Baer}.
Denote by $\Gamma$ and $\Upsilon$, respectively, the bicapsules associated to the 
$(\cB,\cA)$- and $(\cA,\cS)$-morphisms in Fact~\ref{fact:adj-ten} and Section~\ref{sec:upsilon}. 
These three capsules can now be combined to 
produce the $(\cG,\cS)$-capsule
\begin{align*}
    \Delta \otimes_{\cat{B}} \Gamma \otimes_{\cat{A}} \Upsilon
    = \cat{G}\cdot \mu\cdot  \cat{S}.
\end{align*}
The resulting generator $\mu$ of this cyclic bicapsule is  
a unital. By Theorem~\ref{thm:general-dual} this provides the 
characteristic subgroup used in \cite{Maglione:adjoints} and \cite{Wilson:filters}*{\S 4}.

\section{Implementation}\label{sec:imp}
At the suggestion of the referee, we developed code in Agda~\cite{Agda} that
focuses on the central topic of this paper: modeling characteristic
structure as categories acting on categories. Our documented implementation
is available at \cite{glassbox}. We encountered challenges in achieving both computational utility and verification and believe it is useful to identify 
them.

{\it Refining the decision hierarchy.}
A main goal of the implementation was to build 
the category of all homomorphisms of an algebraic structure. 
This requires a decidable 
composition test: to compose $f:A\to B$ and $g:C\to D$, we must decide whether $B=C$.
The question is decidable if $B$ and $C$ are simple types, but 
is undecidable for dependent types. 
For the particular dependent types used in our implementation, equality
is decidable by case-splitting, but this has a high combinatorial 
cost.  A `decidable universe tower' could potentially address this issue.

{\it Refining types for carrier sets.}
In many computational settings, the carrier set is not fixed in advance, 
but instead is generated by operations. One such setting is free algebras, 
which are 
equivalent to inductive types. Another is finite presentations 
(with explicit relations), which can be handled using higher inductive types. 
However, a third setting crucial in computational algebra is when we form a 
quotient by recognition (with no explicit relations). 
One such instance appears in Section~\ref{sec:int-ext}. 
This presents significant difficulties from a type theory perspective.

{\it Extending tactics to loop invariants.}
Computer algebra systems typically rely on mutable data, while theorem provers 
emphasize functional, stateless constructs. Many stateful algorithms
can be reformulated as loops with invariant properties, where loop termination 
provides the desired proof. Developing tactics and types that express such 
invariants directly, without expanding them into recursion, would enhance both 
efficiency and usability.

Using systems such as Agda reduces the risk of misinterpreting code or relying
on unverified results. Yet complex type systems sometimes create the illusion
that stronger claims have been proved than actually are. Implementing explicit
inhabitants and test cases often revealed gaps in our formal proofs. It is of
course  tempting to use ``proof holes'' (such as {\tt postulate} in Agda or {\tt
sorry} in Lean) to bypass seemingly ``obvious'' proofs, but this can undermine
the computational benefits of formalization. The development of our
implementation made this clear: avoiding postulates forced repeated
reformulation and showed us that treating categories as varieties,
rather than as the essentially algebraic structures we initially studied, 
was essential.  This conceptual shift strengthened our main results
and simplified their exposition. The result was not only completely verified
proofs, but also a deeper understanding of the algebra underlying categories.

Note that we avoid postulates for proofs, but our Agda code uses the tag 
\verb|{-# OPTIONS --allow-unsolved-metas #-}|
which averts warnings about potential holes in our code:
these are confined to proofs of negation only, and satisfy the
type-checker's need to return something if a contradiction is raised.
Since contradictions cannot arise, such holes
are \emph{unreachable} and do not represent a gap.

\section*{Acknowledgements}
We thank the referee for careful reading and valuable comments; 
the suggestion to develop a proof-of-concept implementation especially 
informed our understanding of the theory and its applications.
We thank Chris Liu for fruitful discussions.
 We thank John Power and 
Mima Stanojkovski for comments on a draft.  
Brooksbank was supported by NSF grant DMS-2319371. 
Maglione was supported by DFG grant VO 1248/4-1 (project
number 373111162) and DFG-GRK 2297. 
O’Brien was supported by the Marsden Fund of New Zealand Grant 23-UOA-080
and by a Research Award of the Alexander von Humboldt Foundation.
Wilson was supported by a Simons Foundation Grant
identifier \#636189 and by NSF grant DMS-2319370.

\enlargethispage{2\baselineskip}

\end{document}